\documentclass[oneside,final]{amsart}
\usepackage[T1]{fontenc}
\usepackage[latin9]{inputenc}
\usepackage{geometry}
\usepackage{soul}
\usepackage{xcolor}
\usepackage{amssymb}

\makeatletter

\providecolor{lyxadded}{rgb}{0,0,1}
\providecolor{lyxdeleted}{rgb}{1,0,0}

  \theoremstyle{plain}
  \newtheorem*{thm*}{Theorem}
 \theoremstyle{definition}
 \newtheorem{defn}{Definition}
  \theoremstyle{remark}
  \newtheorem{rem}{Remark}
  \theoremstyle{plain}
  \newtheorem{prop}{Proposition}
  \theoremstyle{plain}
  \newtheorem{lem}{Lemma}
  \theoremstyle{plain}
  \newtheorem{thm}{Theorem}
  \theoremstyle{plain}
  \newtheorem{cor}{Corollary}
  \theoremstyle{remark}
  \newtheorem{claim}{Claim}

\usepackage{graphics}
\usepackage{epsfig}
\usepackage{pstricks}
\usepackage{pst-node}

\psset{linewidth=0.05}
\def\vpic#1#2{\mbox{$\begin{array}[c]{l} \epsfig{file=#1,width=#2}\end{array}$}}

\let\cal\mathcal

\makeatother

\begin{document}

\title{Random matrices, free probability, planar algebras and subfactors.}

\author{A. Guionnet*, V. F. R. Jones$^\dagger$, D. Shlyakhtenko$^\ddagger$.}
\thanks{*UMPA, ENS Lyon, 46 all\'ee d'Italie, 69364 Lyon Cedex 07, FRANCE, aguionne@umpa.ens-lyon.fr.  A part of this project was supported by the Miller
Institute for Basic
Research in Science, University of California Berkeley.\\
$\dagger$ Department of Mathematics, UC Berkeley, Berkeley, CA 94720,  vfr@math.berkeley.edu.  Research supported by NSF grant DMS-0401734\\
$\ddagger$ Department of Mathematics, UCLA, Los Angeles, CA 94720, shlyakht@math.ucla.edu.  Research supported by NSF grant DMS-0555680}
\begin{abstract}
Using a family of graded algebra structures on a planar algebra and
a family of traces coming from random matrix theory, we obtain a tower
of non-commutative probability spaces, naturally associated to a given
planar algebra. The associated von Neumann algebras are II$_{1}$
factors whose inclusions realize the given planar algebra as a system
of higher relative commutants. We thus give an alternative proof to
a result of Popa that every planar algebra can be realized by a subfactor.
\end{abstract}
\maketitle

\section{Introduction}

It has been apparent for quite some time that there exists a strong
connection between subfactors, large random matrices and free probability
theory. Perhaps the most clear instance of this connection is that
all three theories have an underlying \emph{planar} structure. For
example, the standard invariant of a subfactor (i.e., the system of
higher relative commutants) is in a natural way a planar algebra \cite{jones:planar}.
Traces of polynomials in random matrices naturally count certain planar
objects (\cite{thooft-planar,brezin-itzykson-parisi-zuber:planarDiagrams,guionet-eduoard:combRM,Mingo:2006lr,mingo-speicher:2ndorderIII,mingo-speicher:2ndorderII,zvonkin:matrixCombIntro}).
Finally, the combinatorics of free probability theory is intimately
tied with that of non-crossing (i.e., planar) partitions \cite{cumulants}.
Furthermore, techniques from some of these subjects proved useful
for applications to others. For example, there are many connections
between work in free probability theory and certain computations in
the paper \cite{bisch-jones:intermediateSubfact}. Random matrices
and free probability theory were used to construct subfactors \cite{radulescu:subfact,shlyakht-ueda:subfactors,popa:standardlattice,shlyakht-popa:universal}.
More recently, Mingo and Speicher and Guionnet and Maurel-Segala \cite{alice:enumeration,guionet-eduoard:combRM,Mingo:2006lr,mingo-speicher:2ndorderII,mingo-speicher:2ndorderIII}
have found combinatorial expressions, involving planar diagrams, for
the that the large-$N$ asymptotics of moments of polynomials in certain random matrices.

In this paper we exploit for the first time the \emph{graded algebra}
coming from a planar algebra $P$ to obtain a subfactor $N\subset M$
whose standard invariant is $P$. The essential ingredient is a \emph{trace}
on the graded algebra coming from free probability/random matrices,
whose use in this context was inspired by \cite{alice:enumeration,guionet-eduoard:combRM,guionnet-shlyakht:convexPotentials},
which promises to be a source of further developments in this direction.

We take the point of view that all of the three subjects mentioned
above are intimately related to the notion of a planar algebra. Specifically,
the underlying idea is that a planar algebra, endowed with its graded
multiplication $\wedge_{0}$ and trace $Tr_{0}$ is a natural replacement
for the ring of polynomials occurring in both free probability theory
\cite{DVV:book} and the theory of random matrices with a potential
\cite{alice:enumeration,guionet-eduoard:combRM,guionnet-shlyakht:convexPotentials}. 

To be more precise, a \emph{subfactor planar algebra} (SPA) $P$ will
be a graded vector space $P=(P_{n},n>0,P_{0}^{\pm})$ which is an
algebra over the planar operad of \cite{vfr:annularStru,jones:planar,jones:graphPlanarAlg}
and satisfies certain dimension and positivity conditions outlined
in $\S$\ref{sec:PlanarAlgebra}. Every extremal finite index subfactor
has an SPA as its \emph{standard invariant}.

Given an SPA $P$, we define the sequence $Gr_{k}P$, $k=0,1,2,\ldots$
of complex $*$-algebras with $Gr_{k}P=\oplus_{n\geq k}P_{n}$ ($P_{0,+}^{}\oplus\bigoplus_{n\geq1}P_{n}$
if $k=0$) and multiplications $\wedge_{k}:P_{n}\times P_{m}\to P_{n+m-k}$
given by tangles as in $\S$\ref{sec:PlanarAlgebra}. On each $Gr_{k}P$
we define a trace $Tr_{k}:Gr_{k}P\to\mathbb{C}$ using the sum of
all Temperley-Lieb tangles. The trace $Tr_{0}$ comes directly from
Wick's Theorem applied to large $N$ limit of a certain Gaussian matrix
model using Wishart matrices, defined in $\S$\ref{sec:randomMatrixModel}.
But once calculated, this trace can be defined entirely in terms of
planar algebras. 

A rather important special case is when the matrix models may be taken
as $p$ independent Hermitian matrices. Then the algebra $Gr_{0}P$
is the even degree subalgebra of $\mathbb{C}\langle\{X\}\rangle$,
the non-commutative polynomials in $p$ self-adjoint variables $\{X\}$
(with $\#\{X\}=p$).   The trace $Tr_{0}$ is then the one 
discovered by Voiculescu in the context of his free probability theory
\cite{DVV:book,DVV:free,DVV:random}.  
It can be realized as 
the vacuum expectation value on the full Fock space on a real $p$-dimensional
vector space with basis $\{X\}$, by the representation of $\mathbb{C}\langle\{X\}\rangle$
which sends $X$ to $\ell_{X}+\ell_{X}^{*}$, $\ell_{X}$ being the
left creation operator of $X$ (see \cite{DVV:free,DVV:book}). 
The higher multiplications $\wedge_{k}$
are then given (on monomials of even degree $\geq2k$) by\[
(X_{1}X_{2}\cdots X_{r})\wedge_{k}(Y_{1}Y_{2}\cdots Y_{s})=(\prod_{i=1}^{k}\delta_{X_{r-i+1},Y_{i}})X_{1}\cdots X_{r-k}Y_{k+1}\cdots Y_{s}.\]
Our main result is, with notation as above and an SPA $P$, of index
parameter $\delta$,

\begin{thm*}
(i) For each $k$, $tr_{k}$ is a faithful tracial state on $Gr_{k}P$
and the GNS completion of $Gr_{k}P$ is a II$_{1}$ factor $M_{k}$
as long as $\delta>1$;

(ii) There are unital inclusions $Gr_{k}P\subset Gr_{k+1}P$ which
extend to $M_{k}\subset M_{k+1}$and projections $\mathbf{e}_{k}\in Gr_{k+1}P$,
such that $(M_{k+1},\mathbf{e}_{k})$ is the tower of basic constructions
for the subfactor $M_{0}\subset M_{1}$;\\
(iii) The relative commutants $M_{0}'\cap M_{k}$ are canonically
identified with the vector spaces $P_{k}$ and this identification
is a homomorphism of planar $*$-algebras.
\end{thm*}
This theorem gives a new proof of the breakthrough result of Popa
\cite{popa:standardlattice}, showing that any subfactor planar algebra
$P$ can indeed be realized by the system of higher relative commutants
of a II$_{1}$ subfactor.

The key ingredient in the proofs will be representations of the algebras
$Gr_{k}P$ on Fock spaces. In order to define these we will suppose
that $P$ is given as a planar subalgebra of the full planar algebra
$P^{\Gamma}$ of some bipartite graph $\Gamma=\Gamma_{+}\amalg\Gamma_{-}$
as in \cite{jones:graphPlanarAlg}. This is always possible --- one
may for instance take $\Gamma$ to be the principal graph of $P$.
A basis of $P^{\Gamma}$ is formed by loops on $\Gamma$ starting
and ending in $\Gamma_{+}$ but we will define a slightly different
planar algebra structure from that of \cite{jones:graphPlanarAlg},
better adapted to graded multiplication. 

The Fock space will then be spanned (orthogonally) by paths of varying
lengths on $\Gamma$, ending in $\Gamma_{+}$. It is naturally $\mathbb{Z}/2\mathbb{Z}$-graded.
Note that $\Gamma$ may be infinite so we need to make a choice of
Perron-Frobenius eigenvector and eigenvalue for the adjacency matrix
of $\Gamma$. There will not necessarily be a Markov trace on $P^{\Gamma}$
so we work instead with the center-valued trace. This will restrict
to a Markov trace on $P$.

As in the theory of graph $C^{*}$-algebras \cite{Raeburn:graphAlgebras},
each edge $e$ of $\Gamma$ defines an operator $\ell_{e}$ (of grading
$1$) on the Fock space, creating an edge on a path. A loop of edges
$e_{1}\cdots e_{2p}$ in $P^{\Gamma}$ is then represented by the
product $c(e_{1})c(e_{2})\cdots c(e_{2p})$ where $c(e_{i})$ is a
version of $a_i\ell(e_{i})+a_i^{-1}\ell(e_{i}^{*})$ according to the
parity of $i$, where the factors $a_i$ are determined by the Perron-Frobenius eigenvector.
We also make use of the fact that the Fock space of loops and the
resulting II$_{1}$ factor can be embedded into a type III factor
canonically associated to the graph and its Perron-Frobenius eigenvector
using a free version of the second quantization procedure. 

The material is organized as follows:

\tableofcontents{}

\subsection{Notations.}

To aid the reader, we list here some notation used in the
paper.

\begin{itemize}
\item Bi-partite graph (\S\ref{sub:graphPlanar}): $\Gamma$; vertices:
$\Gamma$; even/odd vertices: $\Gamma_{\pm}$; Edges: $E$; positively/negatively
oriented edges: $E_{\pm}$; edges starting at $v$: $\Gamma_{+}(v)$;
edges ending at $v$: $\Gamma_{-}(v)$. All loops of length $k$ starting
at even/odd vertex: $L_{k}^{\pm}$; all loops starting at a positive/negative
vertex: $L^{\pm}$.
\item Planar algebra (Def. \ref{def:planarAlg}): $P$; $k$-box space:
$P_{k,\pm}$; positive/negative part: $P_{\pm}$; $P_{k}=P_{k,+}$.
Planar algebra of a graph (\S\ref{sub:graphPlanar}): $P^{\Gamma}$,
$P_{k,\pm}^{\Gamma}$ etc. Subfactor planar algebra: Def. \ref{subfactorplanar}.
\item Graded multiplications: $\wedge_{k}$ (Def. \ref{def:Grk}), the algebra
$Gr_{k}P$. Trace on $Gr_{k}P$: $Tr_{k}$ (Def. \ref{def:Trk}). 
\end{itemize}

\section{On planar algebras.\label{sec:PlanarAlgebra}}

\subsection{Definition}

We begin with a definition of planar algebra which will be recognizably
equivalent to other definitions \cite{jones:planar} and suited to
the purposes of this paper.

\begin{defn}
\label{tangle} (Planar $k$-tangles.) A planar $k$-tangle  will
consist of a smoothly embedded disc $D$ $(=D_{0})$ in $\mathbb{R}^{2}$
minus the interiors of a finite (possibly empty) set of disjoint smoothly
embedded discs $D_{1},D_{2},\dots,D_{n}$ in the interior of $D$.
Each disc $D_{i}$, $i\geq0$, will have an even number $2k_{i}\geq0$
of marked points on its boundary (with $k=k_{0}$). Inside $D$ and
outside $D_{1},D_{2},\dots,D_{n}$ there is also a finite set of disjoint
smoothly embedded curves called strings which are either closed curves
or whose boundaries are marked points of the $D_{i}$'s. Each marked
point is a boundary point of some string, and the strings meet the
boundaries of the discs transversally, only in the marked points.
The connected components of the complement of the strings in ${\displaystyle {\stackrel{\circ}{D}\backslash\bigcup_{i=1}^{n}D_{i}}}$
are called regions. Those parts of the boundaries of the discs between
adjacent marked points (and the whole boundary if there are no marked
points) will be called intervals. The regions of the tangle will be
shaded black and white so that two regions whose boundaries intersect
are shaded differently. (Such a shading is always possible, since
there is an even number of marked points.) The shading will be considered
to extend to the intervals which are part of the boundary of a region.
Finally, to each disc in a tangle there is a distinguished interval
on its boundary (which may be shaded black or white).
\begin{defn}
The set of internal discs of a tangle $T$ will be denoted ${\cal D}_{T}$. 
\end{defn}
\end{defn}
\begin{rem}
Observe that diffeomorphisms of $\mathbb{R}^{2}$ act on planar tangles
in the obvious way. In particular if $\Phi$ is a diffeomorphism it
induces a map $\Phi:{\cal D}_{T}\rightarrow{\cal D}_{\Phi(T)}$ 
\end{rem}
We will often have to draw pictures of tangles. To indicate the distinguished
interval on the boundary of a disc we will place a {*}, near to that
disc, in the region whose boundary contains the distinguished interval.
An example of a $4$-tangle illustrating all the above ingredients
is given below;\[
\vpic{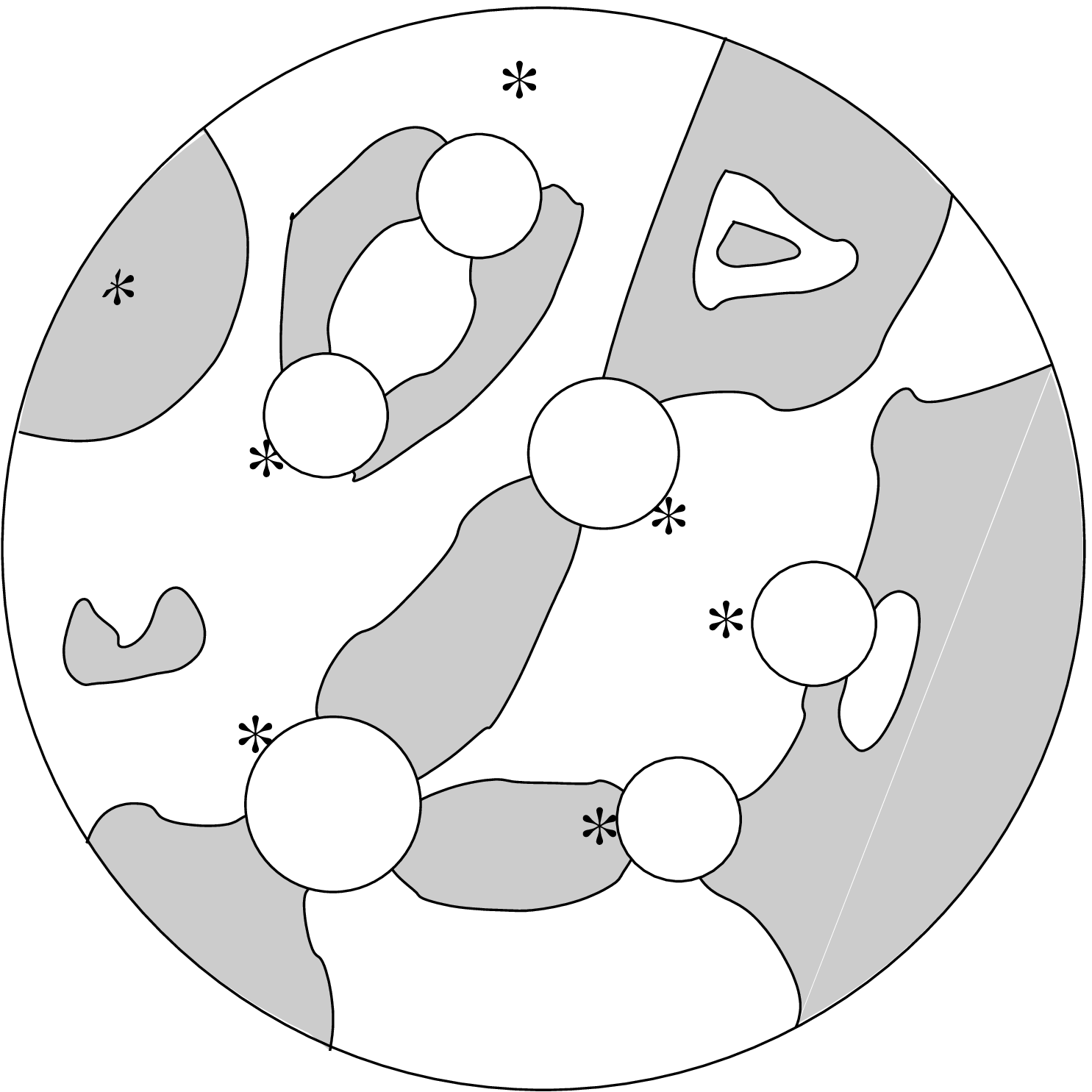}{1.5in}\]

We will often use pictures with a given number of strings to illustrate
a situation where the number of strings is arbitrary. We hope this
will not lead to misinterpretation. Similarly if the shading is implicit
or both possible shadings are intended we will suppress the shading.

Given planar $k$ and $k'$-tangles $T$ and $S$ respectively, we
say they are composable if

\begin{enumerate}
\item The outside boundary of $S$ is equal to the boundary of one of the
inside discs of $T$ where equality means that the marked points are
the same, the shadings of the intervals are the same and the distinguished
intervals are the same. And,
\item The union of the strings of $S$ and those of $T$ are smooth curves.
\end{enumerate}
\begin{defn}
If $T$ and $S$ are composable we define the composition $T\circ S$
to be the union $T\cup S$. The strings of $T\circ S$ are the unions
of the strings of $T$ and $S$. 
\end{defn}
Since the shadings of $T$ and $S$ agree on their common boundary
curve, it is easy to see that $T\cup S$ is a planar $k$-tangle.
This composition operation is often called \char`\"{}gluing\char`\"{}
as one may think of $S$ as being glued inside $T$.

We will now define a notion of planar algebra. Axioms can be subtracted
to obtain more general objects but for convenience in this paper the
term \char`\"{}planar algebra\char`\"{} will imply all the properties.

Before giving the formal definition we recall the notion of the Cartesian
product of vector spaces over an index set ${\cal I}$,\quad{}${\displaystyle \times_{i\in{\cal I}}V_{i}}$.
This is the set of functions $f$ from ${\cal I}$ to the union of
the $V_{i}$ with $f(i)\in V_{i}$. Vector space operations are point-wise.
Multilinearity is defined in the obvious way, and one converts multinearity
into linearity in the usual way to obtain ${\displaystyle {\otimes_{i\in{\cal I}}V_{i}}}$,
the tensor product indexed by ${\cal I}$. A Cartesian product over
the empty set will mean the scalars.

\begin{defn}
\label{def:planarAlg}A (unital) planar algebra $P$ will be a family
of $\mathbb{Z}/2\mathbb{Z}$-graded vector spaces indexed by the set
$\{\mathbb{N}\cup\{0\}\}$, where $P_{k,\pm}$ will denote the $\pm$
graded space indexed by $k$. To each planar tangle $T$ there will
be a multilinear map \[
Z_{T}:\times_{{\scriptstyle {D\in{\cal D}_{T}}}}P_{D}\rightarrow P_{D_{0}}\]
 where $P_{D}$ is the vector space indexed by half the number of
marked boundary points of $D$ and graded by $+$ if the distinguished
interval of $D$ is shaded white and $-$ if it is shaded black.

The maps $Z_{T}$ are subject to the following two requirements:
\end{defn}
\begin{enumerate}
\item (Isotopy invariance) If $\varphi$ is an \underbar{orientation preserving}
diffeomorphism of $\mathbb{R}^{2}$ then \[
Z_{T}=Z_{\varphi(T)}\]
 where the sets of internal discs of $T$ and $\varphi(T)$ are identified
using $\varphi$.

\item (Naturality) \[
Z_{T\circ S}=Z_{T}\circ Z_{S}\]
Where the right hand side of the equation is defined as follows: first
observe that ${\cal D}_{T\circ S}$ is naturally identified with $({\cal D}_{T}\setminus\{D'\})\cup{\cal D}_{S}$,
where $D'$ is the disc of $T$ containing $S$. Thus given a function
$f$ on ${\cal D}_{T\circ S}$ to the appropriate vector spaces, we
may define a function $\tilde{f}$ on ${\cal D}_{T}$ by \[
\tilde{f}(D)=\left\{ \begin{array}{ll}
f(D) & \mbox{if $D\neq D'$}\\
Z_{S}(f|_{{\cal D}_{S}}) & \mbox{if $D=D'$}\end{array}\right.\]
 Finally the formula $Z_{T}\circ Z_{S}(f)=Z_{T}(\tilde{f})$ defines
the right hand side.

The natural notation for $Z_{T}(f)$ is to write in the $\{f(D),\ D\in\mathcal{D}_{T}\}$
into $\mathbb{D}$. This is just like the notation {}``$y(x_{1},\ldots,x_{n})"$
for a function of several variables, where the $x_{i}$ are the $f(D)$,
and the internal discs correspond to the spaces between the commas.
(We also call the internal disks {}``input discs''). Thus if $R_{1}$
and $R_{2}$ are in $P_{2,+}$, $R_{3}$ is in $P_{2,-}$ and $R_{4}$
is in $P_{3,+}$ then the following picture is an element of $P_{4,-}:$\[
\vpic{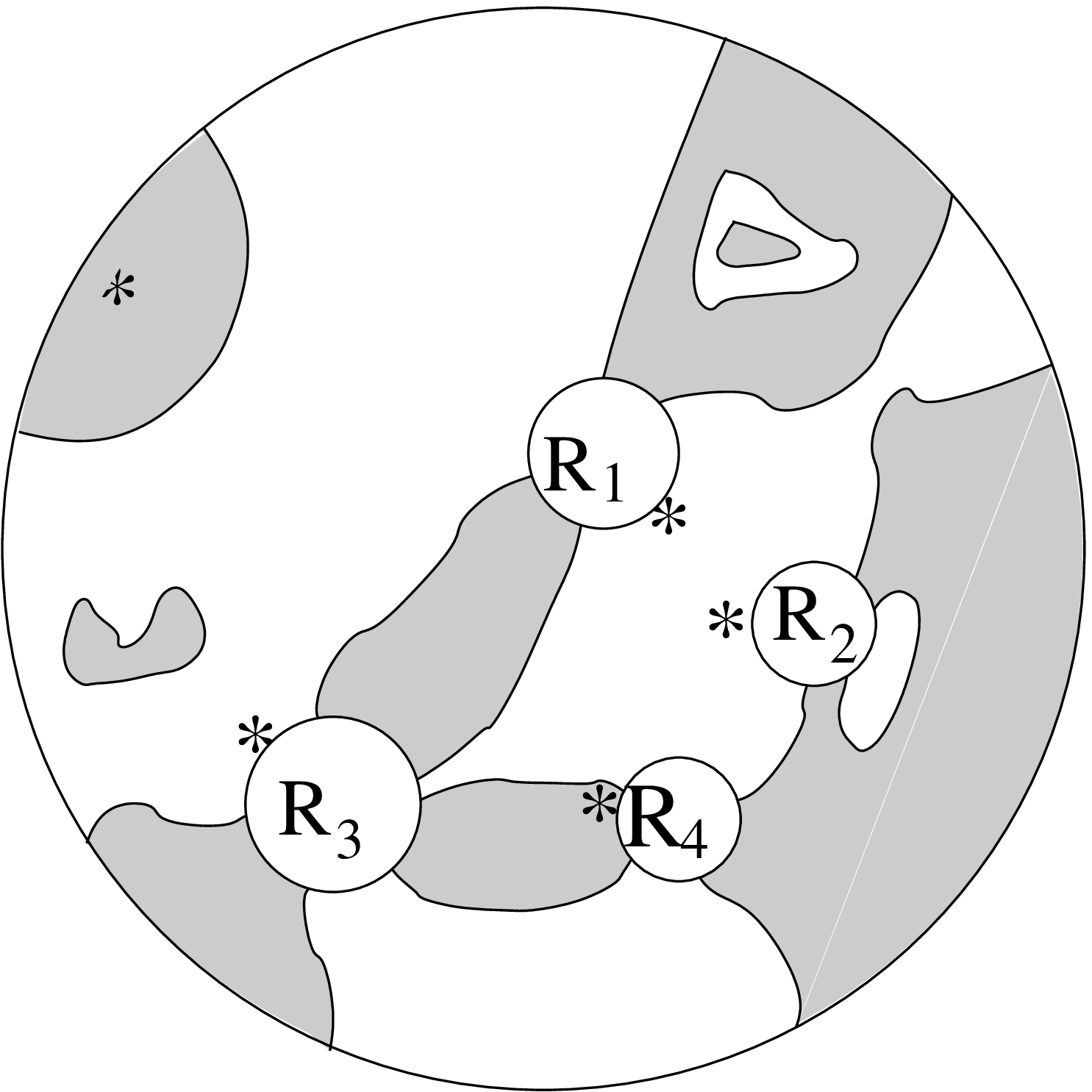}{1.5in}\]

The vector spaces $P_{n,\pm}$ will possess a conjugate linear involution
$*$, $x\rightarrow x^{*}$ with the compatibility requirement: \[
Z_{T}(f^{*})=Z_{\Phi(T)}(f\circ\Phi)^{*}\]
 whenever $\Phi$ is an orientation \underbar{reversing} diffeomorphism.

Observe that $P_{0,\pm}$ become \emph{unital commutative $*$-algebras}
under the multiplication operation (with either shading):\[
ab=\begin{array}{c}
\ovalnode{Outer}{
\psmatrix[rowsep=0.3cm]
\circlenode{a}{a}\\
\circlenode{b}{b}
\endpsmatrix
}\end{array}\]

\begin{defn}
\label{subfactorplanar} A \underbar{subfactor} planar algebra $P$
will be a planar algebra satisfying the following four conditions:
\begin{enumerate}
\item $\dim(P_{n,\pm})<\infty$ for all $(n,\pm)$
\item $\dim(P_{0,\pm})=1$
\end{enumerate}
Condition (\romannumeral 2) allows us to canonically identify $P_{0,\pm}$
with $\mathbb{C}$ as {*}-algebras, $1$ being $Z$ (a 0-tangle with
nothing in it$)$.

This further allows us to define a sesquilinear form on each $P_{n,\pm}$
by: \[
\langle a,b\rangle={\begin{array}{c}{\ovalnode{Outer}{
\psmatrix[rowsep=0.3cm] 
\circlenode{a}{b^*}
\\
\circlenode{b}{a}
\nccurve[angleA=-80,angleB=80]{a}{b}
\nput[labelsep=0]{-160}{b}{*}
\nput[labelsep=0]{190}{a}{*}
\nccurve[angleA=-60,angleB=60]{a}{b} 
\nccurve[angleA=-40,angleB=40]{a}{b}
\nccurve[angleA=260,angleB=100]{a}{b}
\nccurve[angleA=240,angleB=120]{a}{b} 
\nccurve[angleA=220,angleB=140]{a}{b}
\endpsmatrix}}\end{array}}\]
where the outside region is shaded according to $\pm$. 
\end{defn}
\item The form $\langle\ \ ,\ \ \rangle$ is positive definite.
\item $Z_{T_{1}}=Z_{T_{2}}$ where $T_{1}$ and $T_{2}$ are the following
two 0-tangles:\[
T_{1} = \circlenode{outer}{
\begin{array}{c}
\circlenode{inner}{T}
\ncloop[angleA=0,angleB=180,linearc=0.1,arm=0.17,loopsize=0.46]{inner}{inner}
\nput[labelsep=0]{200}{inner}{*}
\end{array}
}\qquad T_2= \circlenode{outer1}{
\begin{array}{c}
\circlenode{inner1}{T}
\nput[labelsep=0]{200}{inner1}{*}
\ncloop[angleA=180,angleB=0,linearc=0.1,arm=0.17,loopsize=0.46]{inner1}{inner1}
\end{array}
}
\]
 The last condition is topologically natural and corresponds to extremality
of the subfactor (\cite{pimsner-popa:entropyIndex},\cite[1.2.5]{popa:classificationIIacta}). This condition means that the partition
function of a fully-labeled zero-tangle (when considered without its
boundary disc) is actually well-defined for that zero-tangle on the
sphere $S^{2}$ obtained by adding a point at $\infty$ to $\mathbb{R}^{2}$.
It is natural then to suppress the outer disc of a $0$-tangle in
pictures.
\end{enumerate}
\begin{rem}
\label{delta} Once $P_{0,\pm}$ have been identified with the scalars
there is a canonical scalar $\delta$ associated with a subfactor
planar algebra with the property that the multilinear map associated
to any tangle containing a closed string is equal to $\delta$ times
the multilinear map of the same tangle with the closed string removed.
By positivity $\delta>0$ and it is well known that in fact the possible
values of $\delta$ form the set $\{4\cos^{2}\pi/n:n=3,4,5,\ldots\}\cup[4,\infty)$
\cite{jones:index}.
\begin{rem}
Since $\delta\neq0$ it is clear that all the spaces $P_{n,-}$ are
redundant and subfactor planar algebra could be axiomatized in terms
of $P_{n,+}$. For this reason we will use in what follows $P_{n}$
to denote $P_{n,+}$ (even in the non-subfactor case). 
\begin{rem}
\label{rem:UsualOps}In the development of planar algebras the following
structures played a major role:
\begin{enumerate}
\item Multiplication: Each $P_{n,\pm}$ is a {*}-algebra with the involution
defined above and the multiplications:
\[
ab={\begin{array}{c}{\ovalnode{Outer}{
\begin{array}{c}
\psmatrix[rowsep=0.3cm] 
\circlenode{a}{b}
\\
\circlenode{b}{a}
\ncdiag[arm=0,angleA=60,angleB=81]{a}{Outer}
\ncdiag[arm=0,angleA=80,angleB=87]{a}{Outer}
\ncdiag[arm=0,angleA=100,angleB=93]{a}{Outer}
\ncdiag[arm=0,angleA=120,angleB=99]{a}{Outer}

\nccurve[angleA=-80,angleB=80]{a}{b}
\nput[labelsep=0]{130}{b}{*}
\nput[labelsep=0]{130}{a}{*}
\nccurve[angleA=-60,angleB=60]{a}{b} 
\nccurve[angleA=260,angleB=100]{a}{b}
\nccurve[angleA=240,angleB=120]{a}{b} 

\ncdiag[arm=0,angleA=-60,angleB=-81]{b}{Outer}
\ncdiag[arm=0,angleA=-80,angleB=-87]{b}{Outer}
\ncdiag[arm=0,angleA=-100,angleB=-93]{b}{Outer}
\ncdiag[arm=0,angleA=-120,angleB=-99]{b}{Outer}

\nput[labelsep=-0.2]{110}{Outer}{*}

\endpsmatrix\end{array}
}}\end{array}}\]
There are two choices of shadings which give in general non-isomorphic
algebra structures. (We shall refer to this multiplication sometimes
as the {}``usual'' multiplication on $P_{n,\pm}$).
\item Trace: Each $P_{n,\pm}$ is equipped with a linear map $Tr:P_{n,\pm}\to P_{0,\pm}$
which is given by \[
Tr(x) = \begin{array}{c}
\circlenode{Outer}{
	\begin{array}{c}
		\circlenode{x}{x}
		\ncarc[arcangle=60,ncurv=1]{x}{x}
		\ncarc[arcangle=30,ncurv=4]{x}{x}
		\nput[labelsep=0]{170}{x}{*}
		\nput[labelsep=0.05]{90}{x}{\cdots}
	\end{array}
}
\end{array}\]
\item The Temperley-Lieb tangles: each tangle consisting of an outside box
all of whose $2k$ boundary points are connected by (non-crossing)
strings inside of the box determines an element of $P_{k,\pm}$, pairing
depending on the shading of the region containing $*$. The set of
such tangles is denoted by $TL(k)$.
\item The Jones projections: $\mathbf{e}_{k}\in P_{k,+}$ is given by the
Temperley-Lieb tangle (having $2k$ boundary points):
\[
\mathbf{e}_k = \begin{array}{c}\rnode{box}{\psframebox[framearc=0.4]{\vbox to2em{\vfill\hbox to10em{\hfill}}}
	\ncbar[ncurv=-1,arm=-0.4,angle=90,offsetA=0.2,offsetB=0.2,linearc=0.2]{box}{box}
	\ncbar[ncurv=-1,arm=-0.8,angle=90,offsetA=0.95,offsetB=0.95,linearc=0.2]{box}{box}
	\ncbar[ncurv=-1,arm=-0.4,angle=90,offsetA=1.5,offsetB=-1.1,linearc=0.2]{box}{box}
	\ncbar[ncurv=-1,arm=-0.4,angle=90,offsetA=-1.5,offsetB=1.1,linearc=0.2]{box}{box}
	\nput[labelsep=-0.75em,offset=0.55]{90}{box}{\cdots}
	\nput[labelsep=-0.75em,offset=-0.55]{90}{box}{\cdots}
	\nput[labelsep=-0.75em,offset=0]{90}{box}{*}
}\end{array}.
\]
\end{enumerate}
\begin{rem}
\label{rectangles} (Rectangles). It is sometimes very convenient
to use rectangles rather than circles for the input and output discs.
Strictly speaking this is not allowed since the boundaries are supposed
to be smooth. But nothing will happen at the corners of the rectangles
so one may simply interpret a picture of a rectangle as one with smoothed
corners. Use of horizontal rectangles also makes it possible to avoid
specifying the first interval which we will always suppose to be the
one containing the left hand vertical part of the rectangle. 
\begin{rem}
(Outer disks and shading). We occasionally omit the outer disk when
describing a planar algebra element, especially in the case that there
are no boundary points on the outer disk. Also, unless the shading
is explicitly indicated in a picture, we follow the convention that
the region adjacent to the boundary region marked with a $*$ is unshaded
(white).
\end{rem}
\end{rem}
\end{rem}
\end{rem}
\end{rem}

\subsection{Graded algebra structures.}

\begin{defn}
If $P$ is a planar algebra we define a graded algebra $GrP$ as follows.
As a graded vector space $GrP=\bigoplus_{n=0}^{\infty}P_{n,+}$ and
the graded product $\wedge:P_{n}\times P_{m}\rightarrow P_{m+n}$
is given by the tangle below which puts the element of $P_{n}$ entirely
to the left of the element of $P_{n}$:\[
a\wedge b = 
\begin{array}{c}
	\ovalnode{Outer}{\vbox to3.5em{\vfill}
		\begin{array}{c}
			\psmatrix[colsep=0.1]
			\circlenode{a}{\vbox to1.1em{\vfill\hbox to1.1em{\hfill$a$\hfill}\vfill} } &
			\circlenode{b}{\vbox to1.1em{\vfill\hbox to1.1em{\hfill$b$\hfill}\vfill}}
			\endpsmatrix
			\ncdiag[arm=0,angleA=110,angleB=140]{a}{Outer}
			\ncdiag[arm=0,angleA=100,angleB=130]{a}{Outer}
			\ncdiag[arm=0,angleA=90,angleB=120]{a}{Outer}
			\ncdiag[arm=0,angleA=80,angleB=110]{a}{Outer}
			\nput[labelsep=0]{140}{a}{*}
			\nput[labelsep=0]{140}{b}{*}
			\nput[labelsep=-0.22]{147}{Outer}{*}
			
			\ncdiag[arm=0,angleA=60,angleB=30]{b}{Outer}
			\ncdiag[arm=0,angleA=70,angleB=40]{b}{Outer}
			\ncdiag[arm=0,angleA=80,angleB=50]{b}{Outer}
			\ncdiag[arm=0,angleA=90,angleB=60]{b}{Outer}
			\ncdiag[arm=0,angleA=100,angleB=70]{b}{Outer}
			\ncdiag[arm=0,angleA=110,angleB=80]{b}{Outer}

		\end{array}
	}
\end{array}
\]
\end{defn}
(The shading in the picture above is determined by saying that the
region adjacent to the marked interval on the outer box is unshaded
(white); as before, $*$'s denote the marked intervals on the disks).
Note that one could also define a dual structure changing $+$ to
$-$ and changing the shading in the above figure.

As a graded algebra a subfactor planar algebra is just the free graded algebra
on a certain graded vector space as we shall see. 

If $\cal P$ is a subfactor planar algebra let $\mathfrak M$ be the 2-sided ideal of
$Gr\cal P$ spanned by all elements of degree $1$ or more. $\mathfrak M$.
Each graded piece of $\mathfrak M$ has an innner product as defined above.
For each $n\geq 1$ let $\mathfrak N_n$ be the orthogonal complement of $(\mathfrak M ^2)_n$
in $\mathfrak M _n$.

\begin{thm} With notation as above, $Gr\cal  P$ is the free graded algebra generated freely
by $\displaystyle \cup_{n=1}^\infty  \mathfrak N_n$.
\end{thm}

\begin{proof} Let $\pi =( \pi_1,\pi_2,...,\pi_k)$ be an ordered $k$-tuple of integers with $\pi_i\geq 1$ and
$\displaystyle\sum_{j=1}^k \pi_k =n$. Then multiplication defines linear maps 
$$mult_\pi: \mathfrak N _{\pi_1}\otimes \mathfrak N_{\pi_2}\otimes....\mathfrak N_{\pi_k} \rightarrow {\cal P}_n.$$
 By induction the images of $mult_\pi$ span $\mathfrak M_n^2$ as $\pi$ varies. So,
 together with $\mathfrak N_n$ they span ${\cal P}_n$. Thus the theorem follows from the 
 two assertions:\\
 \romannumeral 1 ) Each $mult_\pi$ is injective.  \\
 \romannumeral 2 ) The images of the $mult_\pi$ are orthogonal for different $\pi$.\\
 
 To see \romannumeral 1), note that each $mult_\pi$ is an isometry if we give
  $$\mathfrak N _{\pi_1}\otimes \mathfrak N_{\pi_2}\otimes....\mathfrak N_{\pi_k}$$
  the Hilbert space tensor product structure.
 
 To see \romannumeral 2), let $\pi$ and $\rho$ be two distinct partitions of $n$ as above.
 Suppose $\pi_1>\rho_1$.  Consider the following picture:\\
\smallskip\hspace{1in}  \vpic{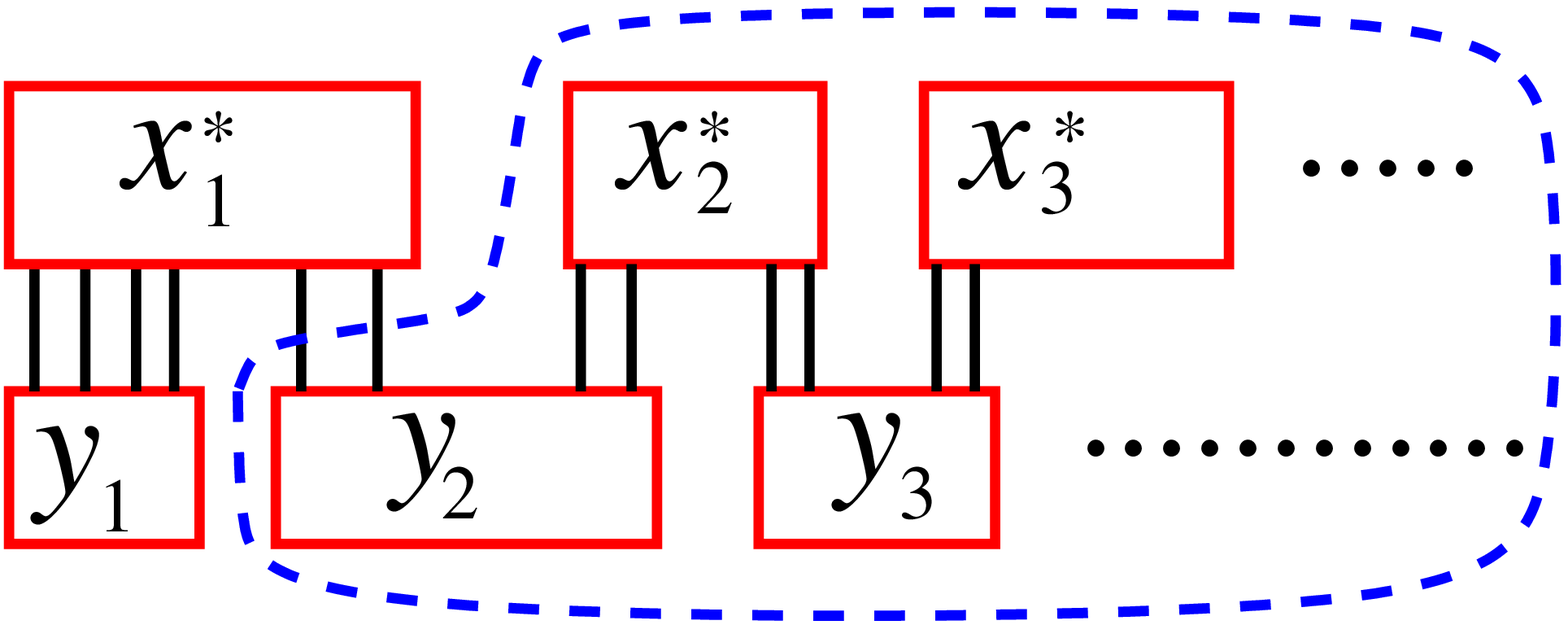} {2in}\smallskip
 
 This is the inner product of an element $y_1\otimes y_2\otimes ...$ in $\mathfrak N _{\rho_1}\otimes \mathfrak N_{\rho_2}\otimes....\mathfrak N_{\rho_k}$
 with an element $x_1\otimes x_2\otimes ...$ in $\mathfrak N _{\pi_1}\otimes \mathfrak N_{\pi_2}\otimes....\mathfrak N_{\pi_k}$.
 (Here $\pi_1=3,\pi_2=2$ and $\rho_1=2=\rho_2=\rho_3$. One may evaluate the tangle inside the dashed curve to obtain
 an element of ${\mathfrak M}_{\pi_1-\rho_1}$. Thus the figure is actually the inner product of $x_1$ with an element of $\mathfrak M^2$,
 thus it is zero. So the images of $mult_\pi$ and $mult_\rho$ are orthogonal unless $\pi_1=\rho_1$. Continuing in this way we see that
 the images of $mult_\pi$ and $mult_\rho$ are orthogonal unless $\pi=\rho$.
\end{proof}
\begin{rem} Writing elements of $Gr\cal P$ as sums of products of elements 
orthogonal to $\mathfrak M^2$ times arbitrary elements gives, by an easy argument
with generating functions, $$ \Psi_{\cal P}(z)= 1-\frac{1}{\Phi_{\cal P}(z)}$$
Where $\Psi_{\cal P}(z)$ is the generating function for $\dim({\mathfrak M}/{\mathfrak M} ^2)_n$,
and $\Phi_{\cal P}(z)$ is the generating function for $\dim P_n$. In general
if $\Phi_n$ is the generating function for the dimensions of the graded vector
space ${\mathfrak M ^n}/{\mathfrak M ^{n+1}}$ we have $\Phi_n=\Phi (\Phi_n - \Phi_{n+1})$
so that $\Phi_n=(1-1/\Phi )^n$.
\end{rem}

Although the graded algebra structure is not commutative even up to 
a sign, the presence of the cyclic group action gives a kind of 
``cyclic commutativity'' as follows where $\rho$ denotes the action of 
the counterclockwise rotation tangle on $P_n$:
\begin{prop} If $\cal P$ is a planar algebra then 
$$\rho^{\deg a}(a\wedge b)= b\wedge a$$
\end{prop}
\begin{proof} Just draw the picture.\end{proof}
\begin{rem} The multiplication in an exterior algebra can be made to
satisfy exactly the same commutativity formula by making the cyclic group act by
the appropriate sign in each degree.
\end{rem}

Besides $Gr{\cal P}$ we will need other \char`\"{}shifted\char`\"{}
graded {*}-algebra structures on ${\cal P}$ in order to define a
subfactor and analyze its tower.

\begin{defn}
\label{def:Grk}Given a planar algebra $P=(P_{n})$ and an integer
$k\geq0$ we make $\oplus_{n=k}^{\infty}P_{n}$ into an associative
(unital) {*}-algebra with multiplication $\wedge_{k}:P_{m}\times P_{n}\rightarrow P_{m+n-k}$
given by the following formula: \[
A\wedge_{k} B = \begin{array}{c}
	\ovalnode{Outer}{\vbox to3em{\vfill}\begin{array}{c}
		\psmatrix[colsep=0.2]
			\circlenode{A}{\vbox to2.2em{\vfill \hbox{$A$}\vfill}} &
			\circlenode{B}{\vbox to2.2em{\vfill\hbox{$B$}\vfill}}
			\ncarc[arcangleA=30,arcangleB=30,ncurv=1]{A}{B}
			\ncarc[arcangleA=40,arcangleB=40,ncurv=1]{A}{B}
			\ncarc[arcangleA=50,arcangleB=50,ncurv=1]{A}{B}
			\nput[labelsep=-0.02]{115}{B}{*}
			\ncdiag[angleA=155,angleB=170,arm=-0,linearc=0.2]{A}{Outer}
			\ncdiag[angleA=145,angleB=160,arm=-0,linearc=0.2]{A}{Outer}
			\ncdiag[angleA=135,angleB=150,arm=-0,linearc=0.2]{A}{Outer}
			\nput[labelsep=-0.02]{115}{A}{*}
			\nput[labelsep=-0.25]{140}{Outer}{*}
			\ncdiag[angleA=100,angleB=130,arm=-0,linearc=0.2]{A}{Outer}
			\ncdiag[angleA=65,angleB=92,arm=0]{A}{Outer}
			\nput[labelsep=-0.25]{112}{Outer}{\cdots}
			\nput[labelsep=-0.05,rot=43]{155}{Outer}{\overbrace{\quad\qquad}^k}
			\nput[labelsep=-0.45,rot=50]{136}{B}{\text{\tiny $\scriptscriptstyle{\underbrace{\quad\!\!}_k}$}}
			\ncdiag[angleA=95,angleB=65,arm=0]{B}{Outer}
			\ncdiag[angleA=25,angleB=10,arm=0]{B}{Outer}
			\nput[labelsep=-0.25,rot=-28]{37}{Outer}{\cdots}
		\endpsmatrix
	\end{array}}
\end{array}
\]
The involution (denoted
by $\dagger$ to distinguish it from the usual involution $*$ on
$P_{k}$) is given by\[
A^\dagger = \circlenode{Outer}{\vbox to2.5em{\vfill} 
	\circlenode{A}{A^*}
	\ncdiag[angleA=155,angleB=170,arm=-0,linearc=0.2]{A}{Outer}
	\ncdiag[angleA=145,angleB=160,arm=-0,linearc=0.2]{A}{Outer}
	\ncdiag[angleA=135,angleB=150,arm=-0,linearc=0.2]{A}{Outer}
	\nput[labelsep=-0.25]{137}{Outer}{*}
	\ncdiag[angleA=120,angleB=120,arm=-0,linearc=0.2]{A}{Outer}
	\ncdiag[angleA=60,angleB=60,arm=0]{A}{Outer}
	\nput[labelsep=-0.25]{90}{Outer}{\cdots}
	\nput[labelsep=-0.0]{52.5}{A}{*}
	\nput[labelsep=-0.05,rot=63]{158}{Outer}{\hbox{\tiny$\overbrace{\quad}^k$}}
		\nput[labelsep=-0.05,rot=-63]{22}{Outer}{\hbox{\tiny$\overbrace{\quad}^k$}}
	\ncdiag[angleA=25,angleB=10,arm=-0,linearc=0.2]{A}{Outer}
	\ncdiag[angleA=35,angleB=20,arm=-0,linearc=0.2]{A}{Outer}
	\ncdiag[angleA=45,angleB=30,arm=-0,linearc=0.2]{A}{Outer}
}	
\]
The shading in both figures above is determined by the condition that
the marked boundary region $*$ is adjacent to an unshaded (white)
region. Here $A^{*}$ means $\phi(A)$ where $\phi$ is an orientation-reversing
diffeomorphism (cf. Def. \ref{def:planarAlg}(2)).

We denote this {*}-algebra by $Gr_{k}P$.
\end{defn}

\subsection{The traces $Tr_{k}$. }

Any planar algebra contains in a canonical way the Temperley-Lieb
planar algebra $TL$. Indeed, $TL$ is spanned by TL diagrams: a TL
diagram is a diagram that has no inner disks, and all of whose strings
connect points on the outer disk. Any such diagram is naturally an
element of $P$.

\begin{defn}
\label{def:Trk}Let $T_{n}$ be the sum of all TL diagrams having
$2n$ points on the outer disk represented pictorially below (for
$n=3$):\[
\ovalnode{OuterT}{
\ovalnode{T}{T_n}
\ncdiag[angleA=-40,angleB=-40,arm=-0,linearc=0.2]{T}{OuterT}
\ncdiag[angleA=-60,angleB=-60,arm=-0,linearc=0.2]{T}{OuterT}
\ncdiag[angleA=-80,angleB=-80,arm=-0,linearc=0.2]{T}{OuterT}
\ncdiag[angleA=-100,angleB=-100,arm=-0,linearc=0.2]{T}{OuterT}
\ncdiag[angleA=-120,angleB=-120,arm=-0,linearc=0.2]{T}{OuterT}
\ncdiag[angleA=-140,angleB=-140,arm=-0,linearc=0.2]{T}{OuterT}
}
\]
(The position of the $*$ is irrelevant, by since the set of TL diagrams is invariant under a rotation by $2\pi/n$).  The trace $Tr_{k}(x)$ is defined for $x\in P_{m}$, $m\geq k$, and
is valued in the zero box space of $P$:
\end{defn}
\[
\rnode{Outer}{\psframebox[framearc=0.4]{\vbox to 6em{\vfill}
\psmatrix[rowsep=0.2,colsep=0.5]
	& \ovalnode{T}{T_n} & \\
	& \rnode{x}{\psframebox[framearc=0.4]{\vbox to2em{\vfill\hbox to 7em{\hfill$x$\hfill}\vfill}}}&
	\ncdiag[angleA=-60,angleB=60,arm=-0,linearc=0.2]{T}{x}
	\ncdiag[angleA=-120,angleB=120,arm=-0,linearc=0.2]{T}{x}
	\nput[labelsep=0]{123}{x}{*}
	\ncbar[angle=90,arm=1.3,offsetA=1.1,offsetB=1.1,linearc=0.2]{x}{x}
	\ncbar[angle=90,arm=1.1,offsetA=0.6,offsetB=0.6,linearc=0.2]{x}{x}
	\nput[labelsep=0.1]{90}{x}{\hbox{\tiny$\cdots$}}
	\nput[offset=0.84]{90}{x}{\hbox{\tiny$\cdots$}}
	\nput[offset=-0.86]{90}{x}{\hbox{\tiny$\cdots$}}
	\nput[labelsep=-0.55,offset=0.84]{90}{x}{\underbrace{\quad}_k}
\endpsmatrix
}}
\]
where $n=m-k$ (in other words, there are $k$ strings surrounding
$T_{n}$).

\begin{lem}
$Tr_{k}$ is a trace on $Gr_{k}P$ if endowed with the multiplication
$\wedge_{k}$.
\end{lem}
\begin{proof}
This follows from the fact that the set of all TL diagrams on $2m$
points is invariant under rotations by $2\pi/m$.
\end{proof}
Before proceeding further, we consider an example. Let us assume that
$P$ is a subfactor planar algebra, so that in particular $P_{0,\pm}$
are one-dimensional and $Tr_{k}$ is scalar-valued. Let $\cup$ be
the following element of TL: 
$
\cup=\rnode{c}{\psframebox[framearc=0.4]{\vbox to0.8em{\vfill\hbox to2em{\hfill}}}}
\ncbar[arm=-0.6em,linearc=0.3em,offsetA=0.5em,offsetB=0.5em,angle=90]{c}{c}
\nput[labelsep=-0.2,offset=0.8em]{90}{c}{*}
$. Let
us denote by $\Phi$ the moment generating function of $\cup$. Thus
we let $\Phi(z)$ be the unique scalar defined by

\[
\Phi(z)=\sum_{n=0}^{\infty}Tr_{0}(\underbrace{\cup\wedge_{0}\cdots\wedge_{0}\cup}_{n\textrm{ times}})z^{n}.\]
We shall presently compute $\Phi(z)$ by using planar algebra methods. 

\begin{defn}
Let $T_{n}$ be the element of the planar algebra defined as the sum
of all the Temperley Lieb diagrams connecting the $2n$ boundary points, 
\end{defn}
\begin{lem}
\label{lem:MGFcup}\[
\Phi(z)=\frac{1-(\delta-1)z}{2z}\left(1-\sqrt{1-\frac{4z}{(1-(\delta-1)z)^{2}}}\right).\]

\end{lem}
\begin{proof}
The trace of $\cup^{n}$ is given by the picture (corresponding to
$n=3$): \[
\ovalnode{Outer}{
\vbox to2.3em{\vfill}
\rnode{T}{\psframebox[framearc=0.4]{\quad T_n \quad}}
\ncbar[offsetA=-0.6,offsetB=0.3,angle=90,ncurv=1,linearc=0.1]{T}{T}
\ncbar[offsetA=0.15,offsetB=0.15,angle=90,ncurv=1,linearc=0.1]{T}{T}
\ncbar[offsetA=0.3,offsetB=-0.6,angle=90,ncurv=1,linearc=0.1]{T}{T}
}
\]
Group the TL diagrams in $T_{n}$ according to where the first boundary
point of $\cup^{n}$ is connected. Adding all those diagrams where
it is connected to its nearest neighbor we get $\delta Tr_{0}(\cup^{n-1})$.
Proceeding similarly we get, for $k=1,2,\ldots,n-1$, contributions of
the form:\[
\psframebox[framearc=0.5]{
	\vbox to5em{
		\vfill\hbox{$
		\psmatrix[colsep=0.6]&
		\rnode{f1}{}\ \  
		\rnode{T}{\psframebox[framearc=0.4]{\qquad T_k \qquad}}
		\ncbar[offsetA=-0.6,offsetB=0.3,angle=90,ncurv=1,linearc=0.1]{T}{T}
		\ncbar[offsetA=0.15,offsetB=0.15,angle=90,ncurv=1,linearc=0.1]{T}{T}
		\ncbar[offsetA=0.3,offsetB=-0.6,angle=90,ncurv=1,linearc=0.1]{T}{T}
		\ \ \rnode{f2}{}&
		\rnode{Tn}{\psframebox[framearc=0.4]{T_{n-k-1}}}
		\ncbar[offsetA=-0.6,offsetB=0.3,angle=90,ncurv=1,linearc=0.1]{Tn}{Tn}
		\ncbar[offsetA=0.15,offsetB=0.15,angle=90,ncurv=1,linearc=0.1]{Tn}{Tn}
		\ncbar[offsetA=0.3,offsetB=-0.6,angle=90,ncurv=1,linearc=0.1]{Tn}{Tn}&
				\ncbar[angle=-90,linearc=0.2]{f1}{f2}
				\ncbar[offsetA=0.9,angle=90,linearc=0.2]{T}{f1}
				\ncbar[offsetA=-0.9,angle=90,linearc=0.2]{T}{f2}
		\endpsmatrix$}
		\vfill
	}
}
\]

If the first term in the picture is rotated by one we may use the
rotational invariance of $T_{k}$ to see that it is just $Tr_{0}(\cup^{k})$.
Thus we have, for each $n>0$, \[
Tr_{0}(\cup^{n})=(\delta-1)Tr_{0}(\cup^{n-1})+\sum_{k=0}^{n-1}Tr_{0}(\cup^{k})Tr_{0}(\cup^{n-k-1})\]
 Multiplying both sides by $z^{n}$ and summing from $n=1$ to $\infty$
we see that \[
\Phi-1=z(\delta-1)\Phi+z\Phi^{2}\]
 Solving the quadratic equation and checking the first term to get
the right solution we obtain our answer.
\end{proof}
The function $\Phi$ in Lemma \ref{lem:MGFcup} is that of a free
Poisson random variable having $R$-transform $\delta(1-z)^{-1}$
(see \cite[p. 311]{dvv:lectures}). We shall give an alternative computation
using free convolution later in the paper (see Lemma \ref{lem:NoAtoms}). 

The following lemma can be easily proved by drawing the appropriate
pictures:

\begin{lem}
The obvious linear embedding of $P_{k}$ into $Gr_{k}P=P_{k}\oplus P_{k+1}\oplus\cdots$
is an algebra $*$-homomorphism from $P_{k}$ endowed with its usual
$*$-algebra structure of Remark \ref{rem:UsualOps}, to $Gr_{k}P$
taken with multiplication $\wedge_{k}$ and conjugation $\dagger$
as in Definition \ref{def:Grk}. Moreover, this embedding carries
the trace $Tr_{k}$ to the usual trace $Tr$ on $P_{k}$.
\end{lem}

\subsection{The planar algebra of a bipartite graph.\label{sub:graphPlanar}}

Let $\Gamma=\Gamma_{+}\cup\Gamma_{-}$ be a (locally finite) bipartite
graph with adjacency matrix $A_{\Gamma}$ possessing an eigenvector
$\mu=\mu_{v}$ ($v$ being a vertex of $\Gamma$) with $\mu_{v}>0$
for all $v$ and $A_{\Gamma}\mu=\delta\mu$. Note that although $\mu$
may be unbounded as a function of $\Gamma$, the ratios $\mu(v)/\mu(v')$
where $v$ and $v'$ are adjacent, are bounded by the eigenvector
condition.

We shall denote by $E$ the set of oriented edges of $\Gamma$, taken
with all possible orientations. Thus $E=E_{+}\cup E_{-}$ where $E_{+}$
consists of all edges of $\Gamma$ oriented so as to start at a vertex
in $\Gamma_{+}$ and end at a vertex in $\Gamma_{-}$, and $E_{-}$
will consists of all edges of $\Gamma$ oriented so as to start in
$\Gamma_{-}$ and end in $\Gamma_{+}$. For $e\in E$ we'll denote
by $e^{o}$ the edge with the opposite orientation. 

In \cite{jones:graphPlanarAlg} a planar algebra was associated with
the above data with the property that closed strings may be removed
multiplicatively as in remark \ref{delta}. We quickly redo this planar
algebra with a slightly different (but isomorphic) structure, emphasizing
those elements that arise when $\Gamma$ is infinite.

With $\Gamma,\mu$ as above we will define the planar algebra $P^{\Gamma}={P}_{n,\pm}^{\Gamma}$
where $P_{n,\pm}^{\Gamma}$ is the vector space of bounded functions
on loops on $\Gamma$ of length $2n$ starting and ending in $\Gamma_{+}$
for the plus sign and $\Gamma_{-}$ for the minus sign.

\begin{defn}
(Spin State) Given a planar tangle $T$, and a bipartite graph $\Gamma$
as above a \emph{spin state} $\sigma$ will be a function from the
regions of $T$ to the vertices of $\Gamma$, shaded regions being
mapped to $\Gamma_{+}$ and unshaded ones to $\Gamma_{-}$, together
with a function from the strings of $T$ to the edges of $\Gamma$
such that if a string $S$ is part of the boundary of the regions
$R_{1}$ and $R_{2}$ then $\sigma(S)$ is an edge connecting $\sigma(R_{1})$
and $\sigma(R_{2})$. 
\end{defn}
Note that a state $\sigma$ determines a function $\ell_{\sigma}:{\cal D}_{T}\cup\{\mbox{boundary disc}\}\rightarrow\{\mbox{loops on }\Gamma\}$
in the obvious way-if we follow a disc of $T$ around clockwise, the
intervals, beginning at the distinguished one, touch regions of $T$
to which $\sigma$ has assigned vertices of $\Gamma$ and the strings
connected to the marked boundary points of a disc $D$ have been assigned
edges of $\Gamma$ connecting the vertices on either side. We will
call $\ell_{\sigma}(D)$ the loop \emph{induced} on $D$ by $\sigma$.

\begin{defn}
(The curvature factor of a spin state.) Given a tangle and a spin
state $\sigma$ as above, define the curvature factor $c(\sigma)$
as follows. First isotope the tangle so that all discs are horizontal
rectangles (with the first boundary interval on the left as in remark
\ref{rectangles}) and all marked points are on the top edges of the
rectangles. Arrange also for all singularities of the y coordinate
on the strings to be generic (maxima or minima). Near such a maximum
(resp. minimum) we see regions above and below, one of which is convex,
labeled by adjacent (on $\Gamma$) vertices $v_{\textrm{convex}}$
and $v_{\textrm{concave}}$ according to $\sigma$. Assign the number
${\displaystyle \sqrt{\frac{\mu(v_{\textrm{convex}})}{\mu(v_{\textrm{concave}})}}}$
to this singularity. Then the curvature factor is\[
c(\sigma)=\textrm{product over all maxima and minima of }{\displaystyle \sqrt{\frac{\mu(v_{\textrm{convex}})}{\mu(v_{\textrm{concave}})}}}.\]

\begin{defn}
(The planar algebra of a bipartite graph.) We now define the action
of a planar tangle $T$ on ${\cal P}^{\Gamma}$. We are given a function
$R:{\cal D}_{T}\rightarrow$ functions on \{loops on $\Gamma$\} and
we have to define a function on loops appropriate to the boundary
of $T$, in a multilinear way.

So given a loop $\gamma$ appropriate to the boundary, define \[
Z_{T}(R)(\gamma)=\sum_{\sigma}\left\{ \prod_{D\in{\cal D}_{T}}R(D)(\ell_{\sigma}(D))\right\} c(\sigma)\]
 where the sum runs over all $\sigma$ which induce $\gamma$ on the
boundary of $T$. 
\end{defn}
\end{defn}
The main thing to note in this definition is that the sum is finite
since there are only a finite number of states inducing $\gamma$
on the boundary, and it defines a bounded function since all the $R$
are bounded and so is the factor $c(\sigma)$.

We leave it as an exercise to show that this definition of $Z_{T}$
is compatible with the gluing of tangles and the {*}-structure where
the {*} of a loop is that loop read backwards. Also that the eigenvector
property of $\mu$ guarantees that contractible closed strings in
tangles can be removed with a multiplicative factor of $\delta$.
Also that this planar algebra structure is isomorphic to that of \cite{jones:planar},
the only change being in how tangles are isotoped in order to define
the factor $c(\sigma)$. The reason for the change is that we are
mostly dealing with the \emph{graded} algebra for which the isotopy
we use is the most natural.

Each of the vector spaces $P_{n,\pm}^{\Gamma}$ is infinite dimensional
if $\Gamma$ is infinite. Moreover $P_{0,\pm}^{\Gamma}$ are the abelian
von Neumann algebras $\ell^{\infty}(\Gamma_{\pm})$ which act on the
$P_{n,\pm}^{\Gamma}$. (Note that the graded product and the usual
product are the same on these subalgebras). The trace tangle when
applied to any element of the planar algebra $P_{0,\pm}^{\Gamma}$
produces an element of $\ell^{\infty}(\Gamma_{\pm})$. We thus get
a bilinear conditional expectation $\mathcal{E}$ from $P_{0,\pm}^{\Gamma}$
(taken with its usual product) onto $\ell^{\infty}(\Gamma_{\pm})$. 

The inner product tangles of definition \ref{subfactorplanar} thus
become $\ell^{\infty}(\Gamma_{\pm})$-valued inner products, satisfying
$\langle a,b\rangle=\mathcal{E}(a^{*}b)$. It will follow from a representation
of the graph planar algebra on a Hilbert space that the conditional
expectation $\mathcal{E}$ (and thus the inner product) is non-negative
definite.

\subsubsection{Representing the planar algebra of a bipartite graph as loops.}

In the next few sections, we shall work out several examples, which
make explicit the operations of graded multiplication on $P^{\Gamma}$,
and which will be useful in the rest of the paper. All of the facts
mentioned below are straightforward consequences of the definition
of the graph planar algebra.

We will sometimes use the word {}``loop'' to also mean the planar
algebra element given by the delta function on the set of all loops
supported on the given loop.

As a matter of convenience, when inserting a loop into an internal
disc of a tangle we will line up the edges of the loop with the boundary
points of the disc, starting with the one first in clockwise order
after $*$. This convention is useful, since given a string meeting
the disc in question at a certain boundary point, any state $\sigma$
which has a nonzero contribution to the sum $Z_{T}$ of will have
to assign the edge of this boundary point to that string.

For an edge $e$ we'll write $s(e)$ for its starting vertex and $t(e)$
for its ending vertex. For a vertex $v$ we'll write $\Gamma_{+}(v)$
for the set of all edges starting at $v$ (i.e., $\Gamma_{+}(v)=\{e:s(e)=v\}$),
and we'll denote by $\Gamma_{-}(v)$ the set of all edges that end
at $v$. We'll also use the notation\[
\sigma(e)=\left[\frac{\mu(t(e))}{\mu(s(e))}\right]^{1/2}.\]

Let $L_{k}^{+}$ be the set of all loops of length $2k$ starting
at an even vertex, and $L_{k}^{-}$ be the set of all loops of length
$2k$ starting at an odd vertex. 

From now on, fix an integer $t$ and consider the algebra $Gr_{t}P^{\Gamma}$
with its graded multiplication $\wedge_{t}$. 

Let $a\in L_{k}^{+}$ be a loop, \[
a = e_{t+1}\cdots e_{k} f_k^{o}\cdots f_1^{o} e_1 \cdots e_{t}
\] where $e_j$ and $f_j$ are edges of $\Gamma$.  Let \[
b = e_{t+1}'\cdots e_{k}' {f'_k}^{o} \cdots {f'_1}^{o} e_1'\cdots e_t'\]
Then the graded product $a \wedge_t b$ is given by:\[
\begin{array}{c}\rnode{o}{\psframebox[framearc=0.4]{
\vbox to4em{\vfill}
\rnode{a}{\psframebox[framearc=0.4] {\vbox to2em{\hbox to 20em{\hfill}\vfill}}}\ 
\nput[offset=9em,labelsep=-0.5]{90}{a}{e_1}
\nput[offset=7.5em,labelsep=-0.4]{90}{a}{\cdots}
\nput[offset=7.5em,labelsep=0.4]{90}{a}{\cdots}
\nput[offset=6em,labelsep=-0.5]{90}{a}{e_t}
\nput[offset=4.8em,labelsep=-0.25]{90}{a}{*}
\nput[offset=4.5em,labelsep=0.5]{90}{a}{*}
\nput[offset=3em,labelsep=-0.5]{90}{a}{e_{t+1}}
\nput[offset=1.5em,labelsep=-0.4]{90}{a}{\cdots}
\nput[offset=1.5em,labelsep=0.4]{90}{a}{\cdots}
\nput[offset=0em,labelsep=-0.5]{90}{a}{e_{k}}
\nput[offset=-1.5em,labelsep=-0.5]{90}{a}{f_{k}^{o}}
\nput[offset=-3em,labelsep=-0.4]{90}{a}{\cdots}
\nput[offset=-3em,labelsep=0.4]{90}{a}{\cdots}
\nput[offset=-4.5em,labelsep=-0.5]{90}{a}{f_{t+1}^{o}}
\nput[offset=-6em,labelsep=-0.5]{90}{a}{f_{t}^{o}}
\nput[offset=-7.5em,labelsep=-0.4]{90}{a}{\cdots}
\nput[offset=-7.5em,labelsep=0]{90}{a}{\cdots}
\nput[offset=-9em,labelsep=-0.5]{90}{a}{f_{1}^{o}}
\ncdiag[arm=0,offsetA=9em,offsetB=-19.5em,angle=90]{a}{o}
\ncdiag[arm=0,offsetA=6em,offsetB=-16.5em,angle=90]{a}{o}
\ncdiag[arm=0,offsetA=3em,offsetB=-13.5em,angle=90]{a}{o}
\ncdiag[arm=0,offsetA=0em,offsetB=-10.5em,angle=90]{a}{o}
\ncdiag[arm=0,offsetA=-1.5em,offsetB=-9em,angle=90]{a}{o}
\ncdiag[arm=0,offsetA=-4.5em,offsetB=-6em,angle=90]{a}{o}
\rnode{b}{\psframebox[framearc=0.4] {\vbox to2em{\hbox to 20em{\hfill}\vfill}}
\nput[offset=9em,labelsep=-0.5]{90}{b}{e'_1}
\nput[offset=7.5em,labelsep=-0.4]{90}{b}{\cdots}
\nput[offset=7.5em,labelsep=0]{90}{b}{\cdots}
\nput[offset=6em,labelsep=-0.5]{90}{b}{e'_t}
\nput[offset=4.8em,labelsep=-0.25]{90}{b}{*}
\nput[offset=3em,labelsep=-0.5]{90}{b}{e'_{t+1}}
\nput[offset=1.5em,labelsep=-0.4]{90}{b}{\cdots}
\nput[offset=1.5em,labelsep=0.4]{90}{b}{\cdots}
\nput[offset=0em,labelsep=-0.5]{90}{b}{e'_{k}}
\nput[offset=-1.5em,labelsep=-0.5]{90}{b}{{f'_{k}}^{o}}
\nput[offset=-3em,labelsep=-0.4]{90}{b}{\cdots}
\nput[offset=-3em,labelsep=0.4]{90}{b}{\cdots}
\nput[offset=-4.5em,labelsep=-0.5]{90}{b}{{f'_{t+1}}^{o}}
\nput[offset=-6em,labelsep=-0.5]{90}{b}{{f'_{t}}^{o}}
\nput[offset=-7.5em,labelsep=-0.4]{90}{b}{\cdots}
\nput[offset=-7.5em,labelsep=0.4]{90}{b}{\cdots}
\nput[offset=-9em,labelsep=-0.5]{90}{b}{{f'_{1}}^{o}}
\ncdiag[arm=0,offsetA=-9em,offsetB=19.5em,angle=90]{b}{o}
\ncdiag[arm=0,offsetA=-6em,offsetB=16.5em,angle=90]{b}{o}
\ncdiag[arm=0,offsetA=-4.5em,offsetB=15em,angle=90]{b}{o}
\ncdiag[arm=0,offsetA=-0em,offsetB=10.5em,angle=90]{b}{o}
\ncdiag[arm=0,offsetA=-1.5em,offsetB=12em,angle=90]{b}{o}
\ncdiag[arm=0,offsetA=3.0em,offsetB=7.5em,angle=90]{b}{o}
\ncbar[linearc=0.2,arm=0.5,offsetA=-6em,offsetB=-6em,angle=90]{a}{b}
\ncbar[linearc=0.2,arm=0.2,offsetA=-9em,offsetB=-9em,angle=90]{a}{b}
}}}\end{array}
\]
which translates into the following formula:\begin{eqnarray*}
a\wedge_t b & = & \delta_{s(f_{1})=s(e_{1}')}\prod_{j=1}^{t}\delta_{f_{j}=e_{j}'}\left[\frac{\mu(s(e_{j}'))}{\mu(t(e_{j}'))}\right]^{1/2}\cdot\\
 &  & \cdot e_{t+1}\cdots e_{k}f_{k}^{o}\cdots f_{t+1}^{o}\ e_{t+1}'\cdots e_{k'}'f_{k'}'^{o}\cdots f_{1}'^{o}\cdots e_1 \cdots e_t .\end{eqnarray*}  
 
 Apart from the Perron-Frobenious factors, $a\wedge_t b$ corresponds to a kind of amalgamated concatenation of paths, although the edges of the path are should be cyclically permuted.  If for a path $a\in L_k^{\pm}$ (parity according to $t$)  we denote by $D_t(a)$
the path that starts at the $t+1$-th edge of $a$, then we have:\[
D_t(a)\wedge_t D_t(b) = \textrm{const} D_t(c)\]
where $c$ is zero if the last $t$ edges of $a$ do not form the inverse of the path formed by the first $t$ segments of $b$, and is the concatenation of $a$ (with last $t$ segments removed) and $b$ (with first $t$ segments removed) otherwise.

In particular, if $t=0$, given two paths $a$, $b$ in $L_k^+$ the graded multiplication $\wedge_{0}$
is just concatenation of paths (note that in this case $D_t$ is the identity map). 

The (usual) trace $Tr$ is given by\[
Tr(e_{1}\cdots e_{k}f_{k}^{o}\cdots f_{1}^{o})=\prod_{j}\delta_{e_{j}=f_{j}}\sigma(e_{j})s(e)\](where again the infinite sums are locally finite).

\subsubsection{$TL \subset Gr_0 P^\Gamma$.}
Let us now set $t=0$ and identify in terms of paths
 the element of $TL(k)\subset P_{k}^{\Gamma}\subset Gr_{0}P^{\Gamma}$
corresponding to any TL picture. Suppose that we are given a box $B$
with $2k$ boundary points (arranged so that all boundary points are
at the top and $*$ is at position $0$ from the top-left). Assume
also that there are $k$ non-crossing curves inside $B$ which connect
pairs of boundary points together. Let $\pi$ be the associated non-crossing
pairing. The associated element of the planar algebra is the function
$w_{B}$ on loops, defined on a loop $a$ by:\[
w_{B}(a)=\begin{cases}
\sigma(e_{1})\cdots\sigma(e_{n}) & \textrm{if }e_{i}=e_{j}^{o}\textrm{ whenever }i\stackrel{\pi}{\sim} j,\ i\neq j,\\
0 & \textrm{otherwise.}\end{cases}\]
If $\pi=\{i_{1},j_{1}\}\cup\cdots\cup\{i_{k},j_{k}\}$ where $i_{1}<i_{2}<\cdots$
and $i_{p}<j_{p}$, then one can think of $w_{B}$ as the following
locally finite sum of delta functions:\[
w_{B}=\sum_{e_{1}\cdots e_{2k}\in L_{k}^{+}}\left\{ \prod\delta_{e_{i_{p}}=e_{j_{p}}^{o}}\sigma(e_{i_{p}})\right\} (e_{1}\cdots e_{2k}).\]
An example of $w_{B}\in Gr_{0}P^{\Gamma}$ associated to the pairing
$\{1,4\},\{2,3\},\{5,12\},\{6,9\},\{7,8\},\{9,10\}$ (thus $k=6$
and $t=0$) is presented below:\[
w_{B}=\sum_{\parbox{4.5cm}{\begin{center} \tiny $e_{1},\ldots,e_{6}:$\\$
e_{1}e_{2}e_{2}^{o}e_{1}^{o}e_{3}e_{4}e_{5}e_{5}^{o}e_{6}e_{6}^{o}e_{3}^{o}\in L_{5}^{+}$\end{center}}}
\sigma(e_{1})\cdots\sigma(e_{6})
\begin{array}{c}\rnode{x}{\psframebox[framearc=0.4]{\vbox to3.5em{\vfill\hbox to 19em{\hfill}}}}
\nput[offset=8.25em,labelsep=-0.35]{90}{x}{e_1}
\nput[offset=6.75em,labelsep=-0.35]{90}{x}{e_2}
\nput[offset=5.25em,labelsep=-0.35]{90}{x}{e_2^o}
\nput[offset=3.75em,labelsep=-0.35]{90}{x}{e_1^o}
\nput[offset=2.25em,labelsep=-0.35]{90}{x}{e_3}
\nput[offset=9em,labelsep=-0.25]{90}{x}{*}
\nput[offset=0.75em,labelsep=-0.35]{90}{x}{e_4}
\nput[offset=-0.75em,labelsep=-0.35]{90}{x}{e_5}
\nput[offset=-2.25em,labelsep=-0.35]{90}{x}{e_5^o}
\nput[offset=-3.75em,labelsep=-0.35]{90}{x}{e_4^o}
\nput[offset=-5.25em,labelsep=-0.35]{90}{x}{e_6}
\nput[offset=-6.75em,labelsep=-0.35]{90}{x}{e_6^o}
\nput[offset=-8.25em,labelsep=-0.35]{90}{x}{e_3^o}
\ncbar[linestyle=dotted,nodesep=-0.4,arm=-0.4,linearc=0.1,angle=90,offsetA=8.25em,offsetB=-3.75em]{x}{x}
\ncbar[linestyle=dotted,nodesep=-0.4,arm=-0.2,linearc=0.1,angle=90,offsetA=6.75em,offsetB=-5.25em]{x}{x}
\ncbar[linestyle=dotted,nodesep=-0.4,arm=-0.6,linearc=0.1,angle=90,offsetA=2.25em,offsetB=8.25em]{x}{x}
\ncbar[linestyle=dotted,nodesep=-0.4,arm=-0.4,linearc=0.1,angle=90,offsetA=0.75em,offsetB=3.75em]{x}{x}
\ncbar[linestyle=dotted,nodesep=-0.4,arm=-0.2,linearc=0.1,angle=90,offsetA=-0.75em,offsetB=2.25em]{x}{x}
\ncbar[linestyle=dotted,nodesep=-0.4,arm=-0.2,linearc=0.1,angle=90,offsetA=-5.25em,offsetB=6.75em]{x}{x}
\end{array}
\]
(The dotted lines are for illustration purposes only and are not part of the planar diagram). 
In this way, given a $TL(k)$ element $B$ we get an associated element
$w_{B}\in P_{k,+}^{\Gamma}\in Gr_{k}P^{\Gamma}$. This embedding is
the canonical inclusion of the Temperley-Lieb planar algebra into
$P^{\Gamma}$.

\subsubsection{The center-valued trace $Tr_{0}$ on $Gr_{0}P^{\Gamma}$.}

As before, we denote by $T_{k}$ the element\[
T_{k}=\sum_{B\in TL(k)}w_{B}\]
obtained by summing over all $TL(k)$ diagrams.

Let $P_{0,\pm}$ be the zero-box space, i.e., as a linear space it
is $\ell^{\infty}(\Gamma_{\pm})$. The algebras $P_{0,\pm}^{\Gamma}$,
when considered with the graded multiplication $\wedge_{0}$, are
abelian, and are in the center of $Gr_{0}P^{\Gamma}$. Recall that
$\mathcal{E}:P_{n}^{\Gamma}\to P_{0}^{\Gamma}$ is a $P_{0}^{\Gamma}$-bilinear
map determined by $\mathcal{E}(ab^{*})=\langle a,b\rangle$; one can
check that $\mathcal{E}(v)=\mu(v)v$, where $v$ denotes the delta
function at $\Gamma_{\pm}$.

The center-valued trace $Tr_{0}:Gr_{0}P^{\Gamma}\to P_{0}^{\Gamma}$
is given by the equation \[
Tr_{0}(x)=\langle x,T_{k}\rangle=\mathcal{E}(x\cdot T_{k}),\qquad v\in V^{+},x\in P_{k}^{\Gamma}.\]
Here as before $T_{k}$ is the sum of all $TL$ diagrams.

\begin{lem}
\label{lemma:Tr0Formula}Let $v\in\Gamma$ and let $\phi_{v}:Gr_{0}P^{\Gamma}\to\mathbb{C}$
be defined by $\phi_{v}(x)v=Tr_{0}(x)\wedge_{0}v$ (i.e., the value of $Tr_0(x)$, viewed as a function on 
$\Gamma$). Let $x=e_{1}\cdots e_{2k}\in L_{k}^{+}$
be a loop. Then if $x$ starts at $v$,\[
\phi_{v}(x)=\sum_{\pi\in NCP(2k)}\prod_{\{i,j\}\subset\pi}\sigma(e_{i})\delta_{e_{i}=e_{j}^{o}},\]
where the sum is over all non-crossing pairings of $2k$ integers
and the product is taken over all tuples $\{i,j\}$, $i<j$ which
are paired by $\pi$. If $x$ does not start at $v$, $\phi_{v}(x)=0$.

Furthermore, $\phi_{v}$ is uniquely determined by the recursive formula\[
\phi_{v}(x)=\sum_{x=ex_{1}e^{o}x_{2}}\sigma(e)\phi_{t(e)}(x_{1})\phi_{v}(x_{2})\]
and the formula $\phi_{v}(ef^{o})=\delta_{e=f}\ \delta_{s(e)=v}\sigma(e).$
\end{lem}
We note that although the support of an element $a\in Gr_0 P^\Gamma$, viewed as a function on paths, may not be finite, the support of  $a\wedge_0 v$ is always finite, since this element is supported on paths of a fixed length starting and ending at $v$.  Thus the value of $\phi_v$ is well-defined.  Moreover, to know the value of $\phi_v$, it is sufficient to know its value on elements of $Gr_0 P^\Gamma$ that have finite support.
\begin{proof}
Clearly, the recursive formula gives rise to a uniquely defined linear
functional on all finitely-supported elements of $Gr_0 P^\Gamma$ (these elements are, of course, viewed as functions on paths). By the comments above, we shall therefore prove the lemma if we prove that both
the functional $\phi_{v}$ and the functional\[
\phi'_{v}(x)=\delta_{s(e_{1})=v}\sum_{\pi\in NCP(2k)}\prod_{\{i,j\}\subset\pi}\sigma(e_{i})\delta_{e_{i}=e_{j}^{o}}\]
satisfy this recursive relation.

Let $\pi\in NCP(2k)$. Then $1$ is paired with some integer $q$.
Thus $NCP(2k)=\sqcup_{q>1}NC\{2,\ldots,q-1\}\times NC\{q+1,\ldots,2k\}$.
Thus\begin{eqnarray*}
\phi_{v}'(x) & = &\delta_{s(e_{1})=v} \sum_{\pi\in NCP(2k)}\prod_{\{i,j\}\subset\pi}\sigma(e_{i})\delta_{e_{i}=e_{j}^{o}}\\ &=& \sum_{q>1}\sum_{
\parbox{3cm}{\tiny\begin{center}$\pi_{1}\in NCP\{1,\ldots,q-1\}$\\$ \pi_{2}\in NCP\{q+1,\ldots,2l\}$\end{center}}
}\delta_{e_{1}  =e_{q}^{o}}\sigma(e_{1})
 \prod_{\{i,j\}\subset\pi_{1}}\sigma(e_{i})\delta_{e_{i}=e_{j}^{o}}\prod_{\{i,j\}\subset\pi_{2}}\sigma(e_{i})\delta_{e_{i}=e_{j}^{o}}\\
 & = & \sum_{x=ex_{1}e^{o}x_{2}}\sigma(e)\phi'_{t(e)}(x_{1})\phi'_{v}(x_{2}).\end{eqnarray*}
Furthermore, $\phi_{v}'(ef^{o})$ is given by the claimed formula.
Thus $\phi_{v}'$ satisfies the recursive relation.

We now turn to showing that $\phi_{v}$ satisfies the same recursive
relation. Note that $\phi_{v}(x)=0$ unless $x$ starts at $v$. 

Note that if $x=e_{1}\cdots e_{2k}$ and $y=f_{1}\cdots f_{2k}$ then
$\langle x, y\rangle=0$ unless $ x=y^{o}$ (an opposite of a path
is a path with the order of edges and also all edges reversed). Furthermore,
if $x=y^{o}$, then\[
\langle x, x^o\rangle =s(e_1) \prod_{i=1}^{2k}\sigma(e_{i})\]

The set $TL$ of all Temperley-Lieb diagrams can be written as a union\[
TL(2k)=\sqcup_{q}TL\{2,\ldots,q-1\}\times TL\{q+1,\ldots,2k\}\]
in a manner similar to decomposing the partitions ($q$ denotes the
other endpoint of the string ending at $1$). Let us assume that
$x$ starts at $v$. Let us denote by $\rnode{a}{1}B_1\rnode{b}{q}B_2\ncbar[linewidth=0.025,angle=90,linearc=0.1,arm=0.1]{a}{b}$
the diagram in which $1$ is connected to $q$ and $B_{1}\in TL\{2,\ldots,q-1\}$,
$B_{2}\in TL\{q+1,\ldots,2k\}$. Then \begin{eqnarray*}
Tr_{0}(x) & = & \langle x, T_{k}\rangle =\sum_{B\in TL(2k)}\langle x, w_{B}\rangle \\
 & = & \sum_{q}\sum_{\parbox{3cm}{\tiny\begin{center}$B_{1}\in TL\{1,\ldots,q-1\}$\\ $B_{2}\in TL\{q+1,\ldots,2k\}$\end{center}}
 }\langle x, w_{\rnode{a}{1}B_1\rnode{b}{q}B_2\ncbar[linewidth=0.025,angle=90,linearc=0.1,arm=0.1]{a}{b}}\rangle.\end{eqnarray*}
Now, recall that\[
w_{B}=\sum_{f_{1}\ldots f_{2k}\in L_{k}^{+}}\left\{ \prod\delta_{e_{i_{p}}=e_{j_{p}}^{o}}\sigma(f_{i_{p}})\right\} \ f_{1}\cdots f_{2k},\]
so that\begin{eqnarray*}
w_{\rnode{a}{1}B_1\rnode{b}{q}B_2\ncbar[linewidth=0.025,angle=90,linearc=0.1,arm=0.1]{a}{b}}
& = & \sum_{f_{1}\ldots f_{q-1}e\ldots f_{2k-1}e^{o}\in L_{k}^{+}}\sigma(e)\prod_{
\parbox{2cm}{\tiny\begin{center}$\{i_{p},j_{p}\}\subset B_{1}$ \\ {or} $\{i_{p,}j_{p}\}\subset B_{2}$\end{center}}
}\delta_{f_{i_{p}}=f_{j_{p}}^{o}}\sigma(f_{i_{p}})
ef_{1}\cdots f_{q-1}e^{o}f_{q+1}\cdots f_{2k}\\
 & = & \sum_{e}\sigma(e)ew_{B_{1}}e^{o}w_{B_{2}}.\end{eqnarray*}
Moreover,\[
\langle x, w_{\rnode{a}{1}B_{1}\rnode{b}{q}{B_{2}}\ncbar[linewidth=0.025,angle=90,linearc=0.1,arm=0.1]{a}{b}}\rangle =0\]
 unless $x$ has the form $x=ex_{1}e^{o}x_{2}$ with $x_{1}$ a loop
having length $q-2$ and $e$ an edge.  In this case,\[
\langle x,w_{\rnode{a}{1}B_{1}\rnode{b}{q}{B_{2}}\ncbar[linewidth=0.025,angle=90,linearc=0.1,arm=0.1]{a}{b}}\rangle=\langle x_{1}, w_{B_{1}}\rangle \langle x_{2}, w_{B_{2}}\rangle  \sigma(e)\sigma(e^{o})\sigma(e) =\langle x_{1}, w_{B_{1}}\rangle \langle x_{2}, w_{B_{2}}\rangle  \sigma(e).\]
Lastly, if $v=s(e)$ then\[
Tr_{0}(ee^{o})=\langle e, e\rangle=\sigma(e)v.\]
It follows that $\phi_{v}$ satisfies the same recursive formula as
$\phi'_{v}$ and, in particular, $\phi_{v}=\phi_{v}'$.
\end{proof}

\subsubsection{Examples.}

Let us denote by $\cup$ the element $\cup=\sum_{ee^{o}\in L^{+}}\sigma(e)ee^{o}$.
Then\begin{eqnarray*}
Tr_{0}(\cup) & =\sum_{e\in E_{+}} & \sum_{f}\mathcal{E}(ee^{o}\cdot ff^{o})\sigma(e)\sigma(f)\\
 & = & \sum_{e}\mathcal{E}(ee^{o}\cdot ee^{o})\left[\frac{\mu(t(e))}{\mu(s(e))}\right]=\sum_{e}\left[\frac{\mu(t(e))}{\mu(s(e))}\right]s(e)\\
 & = & \sum_{v\in\Gamma_{+}}v\frac{1}{\mu(v)}\sum_{s(e)=v}\mu(t(e))=\sum_{v\in\Gamma_{+}}\delta v,\end{eqnarray*}
since $\sum_{s(e)=v}\mu(t(e))=\sum_{w}\Gamma_{vv}\mu(w)=\delta\mu(v)$.

\subsection{Planar subalgebras of $P^{\Gamma}$.}

It is not the planar algebras $P^{\Gamma}$ that are of real interest,
but some of their planar subalgebras. In particular those with finite
dimensional $P_{n}$ and 1-dimensional $P_{0,\pm}$ for which the
inner product is thus scalar valued and inherits positive definiteness
from $P^{\Gamma}$. 

The following theorem, which follows from Popa's work on the theory
of $\lambda$-lattices (see e.g. Theorem 2.9 (4) in \cite{shlyakht-popa:universal}),
shows that any subfactor planar algebra is a sub-planar algebra of
a planar algebra of a discrete bipartite graph.

\begin{thm}
Let $P$ be an (extremal) subfactor planar algebra, realized as the
$\lambda$-lattice $A_{ij}$ with principal graph $\Gamma$ and associated
Perron-Frobenius eigenvector $\mu$. Let $\mathcal{A}_{i}^{j}$ be
as in Theorem 2.9(4) in \cite{shlyakht-popa:universal}. Then:\\
(a) The graph planar algebra $P^{\Gamma}$ is the planar algebra of
the inclusion $\mathcal{A}_{-1}^{-1}\subset\mathcal{A}_{0}^{-1}$;
in other words, $(\mathcal{A}_{-1}^{-1})'\cap\mathcal{A}_{k}^{-1}=\mathcal{P}_{k}^{\Gamma}$;\\
(b) The isomorphism $P_{j,+}=A_{-1j}\cong(\mathcal{A}_{-1}^{0})'\cap\mathcal{A}_{j}^{-1}\subset(\mathcal{A}_{-1}^{-1})'\cap\mathcal{A}_{j}^{-1}$
gives rise to a planar algebra inclusion of $P$ into $P^{\Gamma}$.
\end{thm}
The algebras $\mathcal{A}_{-1}^{-1}$ and $\mathcal{A}_{0}^{-1}$
were constructed in \cite{shlyakht-popa:universal} as certain non-unital
inductive limits of the algebra $A_{ij}$. Pictorially, this construction
corresponds to e.g. taking $\mathcal{A}_{0}^{-1}$ to be the inductive
limit of the algebras $\{P_{k}:k\textrm{ even}\}$ using the non-unital
inclusion given by the following picture (the region containing $*$ is unshaded):\[
P_k \ni 
\begin{array}{c}
\rnode{outer}{
	\psframebox[framearc=0.4]{
		\vbox to3em {
			\vfill
			\hbox{$
				\rnode{x}{
					\psframebox[framearc=0.4]{
						\vbox to1em{
							\vfill
							\hbox to6em{
								\hfill $x$ \hfill
							}
							\vfill
						}
					}
				}
			$}			
			\vfill
		} 
	}
}
\ncdiag[angle=90,arm=0,offsetA=2.5em,offsetB=-2.5em]{x}{outer}
\ncdiag[angle=90,arm=0,offsetA=2.0em,offsetB=-2.0em]{x}{outer}
\ncdiag[angle=90,arm=0,offsetA=1.5em,offsetB=-1.5em]{x}{outer}
\nput[offset=-0em]{90}{x}{\cdots}
\ncdiag[angle=90,arm=0,offsetA=-2.0em,offsetB=2.0em]{x}{outer}
\ncdiag[angle=90,arm=0,offsetA=-2.5em,offsetB=2.5em]{x}{outer}
\nput[offset=3.1em,labelsep=0.1]{90}{x}{*}
\ncdiag[angle=-90,arm=0,offsetA=2.5em,offsetB=-2.5em]{x}{outer}
\ncdiag[angle=-90,arm=0,offsetA=-2.0em,offsetB=2.0em]{x}{outer}
\ncdiag[angle=-90,arm=0,offsetA=2.0em,offsetB=-2.0em]{x}{outer}
\nput[offset=-0em]{-90}{x}{\cdots}
\ncdiag[angle=-90,arm=0,offsetA=-1.5em,offsetB=1.5em]{x}{outer}
\ncdiag[angle=-90,arm=0,offsetA=-2.5em,offsetB=2.5em]{x}{outer}
\end{array}
\mapsto
\begin{array}{c}
\rnode{outer}{
	\psframebox[framearc=0.4]{
		\vbox to3em {
			\vfill
			\hbox{$
				\rnode{x}{
					\psframebox[framearc=0.4]{
						\vbox to1em{
							\vfill
							\hbox to6em{
								\hfill $x$ \hfill
							}
							\vfill
						}
					}
				}
			$\hskip 3em}			
			\vfill
		} 
	}
}
\ncdiag[angle=90,arm=0,offsetA=2.5em,offsetB=-4.0em]{x}{outer}
\ncdiag[angle=90,arm=0,offsetA=2.0em,offsetB=-3.5em]{x}{outer}
\ncdiag[angle=90,arm=0,offsetA=1.5em,offsetB=-3.0em]{x}{outer}
\nput[offset=-0.25em]{90}{x}{\cdots}
\ncdiag[angle=90,arm=0,offsetA=-2.0em,offsetB=0.5em]{x}{outer}
\ncdiag[angle=90,arm=0,offsetA=-2.5em,offsetB=1.0em]{x}{outer}
\nput[offset=3.1em,labelsep=0.1]{90}{x}{*}
\ncdiag[angle=-90,arm=0,offsetA=2.5em,offsetB=-1.0em]{x}{outer}
\ncdiag[angle=-90,arm=0,offsetA=2.0em,offsetB=-0.5em]{x}{outer}
\ncdiag[angle=-90,arm=0,offsetA=-1.5em,offsetB=3.0em]{x}{outer}
\nput[offset=-0.25em]{-90}{x}{\cdots}
\ncdiag[angle=-90,arm=0,offsetA=-2.0em,offsetB=3.5em]{x}{outer}
\ncdiag[angle=-90,arm=0,offsetA=-2.5em,offsetB=4.0em]{x}{outer}
\ncbar[angle=90,arm=-0.3,linearc=0.1,offsetB=2.0em,offsetA=-4.0em]{outer}{outer}
\ncbar[angle=-90,arm=-0.3,linearc=0.1,offsetA=2.0em,offsetB=-4.0em]{outer}{outer}
\end{array}
\in P_{k+2}
\]
The algebra $\mathcal{A}_{-1}^{-1}$ then consists of all diagrams
having a vertical through-string on the left (again, region containing $*$ is unshaded): \[
\mathcal{A}_{-1}^{-1}= \left\{
\begin{array}{c}
	\rnode{outer}{
		\psframebox[framearc=0.4]{
			\psmatrix[colsep=1em,rowsep=0.3]
			\\ 
			&\rnode{f}{}& \rnode{x}{
				\psframebox[framearc=0.4]{\hbox{
					\vbox to1em{
						\vfill
						\hbox to6em{
							\hfill $x$ \hfill
						}
						\vfill
					}
				}
			}}
			\\ \vbox{\vfill}
			\endpsmatrix
		}
	}
\end{array}
\ncdiag[arm=0,angle=90,linearc=0.2,offsetA=3em,offsetB=-2em]{x}{outer}
\ncdiag[arm=0,angle=90,linearc=0.2,offsetA=2em,offsetB=-1em]{x}{outer}
\ncdiag[arm=0,angle=90,linearc=0.2,offsetA=1em,offsetB=-0em]{x}{outer}
\ncdiag[arm=0,angle=90,linearc=0.2,offsetA=-2em,offsetB=3em]{x}{outer}
\ncdiag[arm=0,angle=90,linearc=0.2,offsetA=-3em,offsetB=4em]{x}{outer}
\ncdiag[arm=0,angle=-90,linearc=0.2,offsetA=3em,offsetB=-4em]{x}{outer}
\ncdiag[arm=0,angle=-90,linearc=0.2,offsetA=2em,offsetB=-3em]{x}{outer}
\ncdiag[arm=0,angle=-90,linearc=0.2,offsetA=-1em,offsetB=0em]{x}{outer}
\ncdiag[arm=0,angle=-90,linearc=0.2,offsetA=-2em,offsetB=1em]{x}{outer}
\ncdiag[arm=0,angle=-90,linearc=0.2,offsetA=-3em,offsetB=2em]{x}{outer}
\nput[offset=-0.5em]{90}{x}{\cdots}
\nput[offset=0.5em]{-90}{x}{\cdots}
\ncdiag[arm=0,angle=90,offsetA=3.5em,offsetB=0.25em]{outer}{f}
\ncdiag[arm=0,angle=-90,offsetA=-3.5em,offsetB=-0.25em]{outer}{f}
\nput[offset=4.1em,labelsep=-0.25]{90}{outer}{*}
\right\}
\]

\section{A random matrix model for $Tr_{0}$\label{sec:randomMatrixModel}}

\subsection{Random block matrices associated to the graph.}

Given $\Gamma$ as above, let $\Gamma_{\pm}$. Let $A=A^{+}\oplus A^{-}$
where $A^{\pm}=\ell^{\infty}(\Gamma_{\pm})=P_{0,\pm}^{\Gamma}$. We
shall denote by $E_{\pm}$ the set of edges of $\Gamma$ which are
positively or negatively oriented (according to the sign $\pm$).
We shall make the convention that together with any edge $e\in E_{\pm}$
there is also its opposite edge $e^{o}\in E_{\mp}$.

We endow $A$ with a (semi-finite) trace $tr$ given on the minimal
projections of $A$ by the formula\[
tr(\delta_{v})=\mu(v),\qquad v\in\Gamma.\]
Let $N,M$ be integers. For each choice of $M$ choose integers $\{M_{v}:v\in\Gamma f\}$
with the property that for each fixed vertex $v$, $M_{v}/M\to\mu(v)$
as $M\to\infty$. 

In the foregoing, we will consider (infinite if the graph $\Gamma$
is infinite) matrices whose entries are indexed by the set $\sqcup_{v\in\Gamma}\{1,\dots,N\}\times\{1,\ldots,M_{v}\}$.
Such an entry will be denoted $A_{ij\ mn\ vw}$, where $i,j\in\{1,\dots,N\}$,
$m\in\{1,\dots,M_{v}\}$, $n\in\{1,\ldots,M_{w}\}$ and $v,w\in\Gamma$.
Given such a matrix $A=(A_{ij\ mn\ vw})$, we compute its trace as
follows:\[
tr(A)=\sum_{v}\ \frac{1}{N}\sum_{1\leq i\leq N}\ \frac{1}{M}\sum_{1\leq n\leq M_{v}}A_{ii\ nn\ vv}.\]
Our matrices will be such that $A_{ij\ mn\ vw}=0$ unless $v,w$ belong
to a finite set, so that the sum above is finite. 

For $v\in\Gamma$ consider the diagonal matrix $d_{v}$ given by\[
(d_{v})_{ij\ mn\ uw}=\delta_{i=j}\delta_{m=n}\delta_{u=w=v}.\]
Note that the joint law of $(d_{v}:v\in\Gamma)$ converges as $M\to\infty$
to the joint law of $(\delta_{v}:v\in\Gamma)$.

Consider then for a positively oriented edge $e\in E_{+}$ from $v$
to $w$ the $NM_{v}\times NM_{w}$ matrix $X_{e}$ defined as follows.
The entry $X_{ij\ mn\ tu}^{e}$ is zero unless $t=v$ and $u=w$.
Otherwise, $X_{ij\ mn\ vw}^{e}$ is (up to scaling) a random Gaussian
matrix; in other words, the entries form a family of independent complex
Gaussian random variables, each of variance $(\mu(s(e))\mu(t(e)))^{-1/2}(NM)^{-1}$.
We shall moreover choose the matrices $X_{e}$ in such a way that
the entries of matrices corresponding to different positively oriented
edges are independent. Thus the variables $\{X_{ij\ mn\ vw}^{e}:e\in E_{+},v=s(e),w=t(e),1\leq i,j\leq N,1\leq m\leq M_{v},1\leq n\leq M_{w}\}$
are assumed to be independent.

For a negatively oriented edge $f$, set $X_{f}=X_{e^{o}}^{*}$. For
a loop $w\in L_{k}^{\pm},$$w=e_{1}\cdots e_{2k}$, set $X_{w}=X_{e_{1}}\cdots X_{e_{2k}}$.
Note that $w\mapsto X_{w}$ is a homomorphism from the algebra $(P^{\Gamma},\wedge_{0})$
to the algebra of random matrices.

\subsection{$Tr_{0}$ via random matrices.}

\begin{prop}
Let $E$ denote the expected value of a random variable. Then the
matrices $X_{e}$ satisfy: (a) $d_{v}X_{e}d_{w}=\delta_{v=s(e)}\delta_{w=t(e)}X_{e}$;
(b) $E(tr(X_{e}^{*}X_{e}))=E(tr(X_{e}X_{e}^{*}))$ is independent
of $N$ and converges to $(\mu(s(e))\mu(t(e)))^{1/2}$ as $M\to\infty$;
(c) For any $v\in V$, $w\in L_{k}^{\pm},$ $\lim_{M\to\infty}\lim_{N\to\infty}E(tr(d_{v}X_{w}))=tr(\delta_{v})Tr_{0}(w)(v)$
(here $Tr_{0}(w)(v)$ means the value of the function $Tr_{0}(w)\in\ell^{\infty}(\Gamma)$
at $v\in\Gamma$).
\end{prop}
\begin{proof}
(a) and (b) are both straightforward; note that\[
E(tr(X_{e}X_{e}^{*}))=\frac{1}{(\mu(s(e))\mu(t(e)))^{1/2}}\frac{M_{v}M_{w}}{M^{2}}\to(\mu(s(e))\mu(t(e)))^{1/2}.\]
To see (c), we first note that if $w=ee^{o}$ then $E(tr(X_{w}))\to(\mu(s(e))\mu(t(e))^{1/2}$
as $N\to\infty$ and then $M\to\infty$. On the other hand, $tr(Tr_{0}(w))=\sigma(e)tr(w)=\sigma(e)\mu(s(e))=(\mu(s(e))\mu(t(e)))^{1/2}$. 

Denote by $\mathcal{E}$ the map the conditional expectation onto
the algebra $A$. Then we have that if $v=s(e)$, $\mathcal{E}(X_{e}X_{e}^{*})$
is a multiple of $d_{v}$. Since $\mathcal{E}$ is $tr$-preserving,
we have\[
E(\mathcal{E}(X_{e}X_{e}^{*}))=tr(v)^{-1}tr(\mathcal{E}(X_{e}X_{e}^{*}))\delta_{v}=\sigma(e)d_{v}.\]
In particular, we see that\[
E(\mathcal{E}(X_{e}X_{f}^{*}))=\delta_{e=f}\ \delta_{v}\ \sigma(e).\]
It is known (see e.g. \cite{benaychgeorges:rectRandomMatrices,shlyakht:bandmatrix})
that the variables $\{X_{e}:e\in\Gamma\}$ converge in distribution
(jointly also with elements of $A$) to a family of $A$-valued semicircular
variables with variance\[
\theta_{e}:\delta_{w}\mapsto\delta_{w=v}\delta_{v}\ \sigma(e).\]
Hence if $w=e_{1}\cdots e_{2k}$, then for any $a\in A$, \[
\lim_{N\to\infty}\lim_{M\to\infty}tr(d_{v}(E(\mathcal{E}(X_{w})))=tr\left(\delta_{v}\sum_{\pi\in NC(2k)}\prod_{\{i,j\}\subset\pi}\delta_{e_{i}=f_{i}}\sigma(e)\right).\]
By Lemma \ref{lemma:Tr0Formula}, we see that\[
tr(E(\mathcal{E}(X_{w}))d_{v})\to tr(\delta_{v})Tr_{0}(w)(v),\qquad\forall v\in V,\]
as claimed.
\end{proof}
Since the trace $tr$ is positive and faithful on $A$ we conclude
that the center-valued trace $Tr_{0}$ is non-negative:

\begin{cor}
$Tr_{0}(x^{*}\wedge_{0}x)\geq0$ if $x\in P^{\Gamma}$.
\end{cor}

\subsection{Another construction of random block matrices.}

Recall that a bi-partite graph can be used as a Bratteli diagram to
describe an inclusion of two algebras. 

Let $B\subset C$ be an inclusion of multi-matrix algebras corresponding
to the graph $\Gamma$. This means $B=\bigoplus_{v\in V_{+}}M_{k(v)\times k(v)}$
and $C=\bigoplus_{w\in V_{-}}M_{l(w)\times l(w)}$. In particular,
each $v\in\Gamma_{+}$ corresponds to a central projection $p_{v}$
in $B$ (the unit of the $v$-th direct summand), and each $w\in\Gamma_{-}$
corresponds to a central projection $q_{w}\in C$. The inclusion $B\subset M$
is such that $p_{v}q_{w}=0$ if there is no edge between $v$ and
$w$. If there are $r$ edges between $v$ and $w$, then $M_{k(v)\times k(v)}=p_{v}Bp_{v}$
is included into $q_{w}Cq_{w}=M_{l(w)\times l(w)}$ with index $r$.
In particular, this means that $l(w)=rk(v)$ and also that we can
choose $r$ orthogonal projections $\{P^{e}\}_{s(e)=v,t(e)=w}$ in
$q_{w}Cq_{w}$ with the property that $P^{e}q_{w}Cq_{w}P^{e}\stackrel{\phi_{e}}{\cong}M_{k(v)\times k(v)}$
and the inclusion of $p_{v}Bp_{v}$ into $q_{w}Cq_{w}$ is given by
$x\mapsto\sum_{s(e)=v,t(e)=w}P^{e}\phi_{j}(x)P^{e}$. Choose also
isometries $V_{e,f}$ so that $P^{e}V_{e,f}=V_{e,f}P^{f}$.

Let $Tr$ be the semi-finite trace on $B\oplus C$ determined by the
requirement that $Tr(p_{v})=\mu(v)$, $Tr(q_{w})=\mu(w)$.

Let $Y$ be a semicircular element, free from $B\oplus C$ (this only
makes sense in the case that $Tr(1)<\infty$; more precisely, we shall
consider a large projection $Q$ in the center of $B\oplus C$ and
consider an element $Y$ free from $Q(B\oplus C)Q$ with respect to
$Tr(Q)^{-1}Tr(\cdot)$; our computations will not depend on $Q$ once
it is large enough).

To a positive edge $e$, we associate: (i) a central projection $p_{s(e)}\in B$;
(ii) a projection $ $$P^{e}\in q_{t(e)}Cq_{t(e)}\subset C$.

Let $Y_{e}=(\mu(t(e))\mu(s(e)))^{-1/4}\sum_{s(f)=s(e),t(f)=t(e)}(p_{s(e)}YP_{e})V_{e,f}$
if $e\in E_{+}$ and $Y_{e}=Y_{e^{o}}^{*}$ if $e\in E_{-}$.

Note that $Y_{e}Y_{f}=0$ unless $t(e)=s(f)$. We can think of $Y_{e}$
as a limit of a {}``$\mu(s(e))\times\mu(t(e))$'' random block matrix,
since its left and right support projections, $p_{v}$ and $q_{w}$,
have traces $\mu(s(e))$ and $\mu(t(e))$. In fact, one can model
$Y$ by a suitable GUE random matrix in the limit when its size goes
to infinity, in which case the variables $Y_{e}$ are indeed approximated
in law by random blocks as their sizes go to infinity.

Furthermore, if $e\in E_{+}$,\begin{eqnarray*}
Tr(p_{v}Y_{e}q_{w}Y_{f}^{*}) & = & (\mu(t(e))\mu(s(e)))^{-1/2}Tr(p_{v}p_{s(e)})Tr(\sum_{e',f'}V_{e,e'}q_{w}V_{f',f})\\
 & = & (\mu(t(e))\mu(s(e)))^{-1/2}\delta_{v=s(e)}Tr(p_{v})\delta_{w=t(e)}\delta_{v=f}Tr(q_{w}).\end{eqnarray*}
Thus also $Tr(q_{w}Y_{f^{o}}p_{v}Y_{e^{o}}^{*})=(\mu(t(e))\mu(s(e)))^{-1/2}\delta_{v=s(e)}Tr(p_{v})\delta_{w=t(e)}\delta_{v=f}Tr(q_{w})$.
It follows that if we denote by $\mathcal{E}$ the conditional expectation
onto the center of $B\oplus C$, then, keeping in mind that $Tr(\delta_{v})=\mu(v)$,\[
\mathcal{E}(Y_{e}\delta_{v}Y_{f})=(\mu(t(e))\mu(s(e)))^{-1/2}\delta_{e=f^{o}}\delta_{v=t(e)}\delta_{s(e)}\mu(t(v))\ =\delta_{e=f^{o}}\delta_{v=t(e)}\delta_{s(e)}\sigma(e).\]
Thus the variables $\{Y_{e}:e\in E\}$ have the same joint law as
the variables $\{X_{e}:e\in E\}$ that we constructed in the previous
section.

\section{The Fock Space Model.}

\subsection{A Hilbert bimodule associated to a bi-partite graph.}

Let $\Gamma$ be a bi-partite graph, as before. Consider the real
vector space $H$ with basis given by the (oriented) edges $E$ of
the graph; we denote, as before, by $E_{+}$ the set of positively
oriented edges. Then $H$ is equipped with a natural conjugation which
takes an edge to its opposite, $e\mapsto e^{o}$ and can thus be endowed
with a complex structure: $i(e+e^{o})=(e-e^{o})$ for any positively-oriented
$e$. The inner product on $H$ is determined by requiring that $\langle e,f\rangle=0$
unless $e=f$ and \[
\Vert e\Vert^{2}=\left[\frac{\mu(s(e))}{\mu(t(e))}\right]^{1/2}.\]
(Note that $e\mapsto e^{o}$ is \emph{not} isometric). As before,
we shall use the notation\[
\sigma(e)=\left[\frac{\mu(t(e))}{\mu(s(e))}\right]^{1/2}=\Vert e\Vert^{-2}.\]

Let $A$ denote the abelian algebra $A=\ell^{\infty}(\Gamma)$, where
as before $\Gamma$ denotes the set of vertices of $\Gamma$. Then
$H$ is naturally an $A,A$-bimodule: given $e$ an edge in $E$,
define\[
v\cdot e\cdot v'=\delta_{v=s(e)}\delta_{v'=t(e)}e.\]
Moreover, $H$ has a natural $A$-valued inner product:\[
\langle e,f^{o}\rangle_{A}=\langle e,f^{o}\rangle s(e)=\langle e,f^{o}\rangle t(f).\]

\subsection{The operators $c(e)$, the weight $\phi$, and the $A$-valued conditional
expectation $E$.}

We now consider the Fock space \cite{pimsner}\[
\mathcal{F}=A\oplus\bigoplus_{k\geq0}H^{\otimes_{A}k}\](here $\otimes_A$ denotes the relative
bimodule tensor product).  For $e\in E$ we consider the operator\[
\ell(e):\mathcal{F}\to\mathcal{F},\qquad\ell(e)\xi=e\otimes\xi.\]
Its adjoint is given by\[
\ell(e)^{*}(e_{1}\otimes\cdots\otimes e_{n})=\langle e,e_{1}\rangle_{A}e_{2}\otimes\cdots\otimes e_{n}.\]
Note that the norm of this operator is given by\[
\Vert\ell(e)\Vert=\Vert\ell^{*}(e)\ell(e)\Vert^{1/2}=\Vert e\Vert^{1/2}.\]
Let also\[
c(e)=\ell(e)+\ell(e^{o})^{*}.\]
Note that $c(e)^{*}=c(e^{o})$.

Let $B(\mathcal{F})$ be the algebra of bounded adjointable operators
on $\mathcal{F}$ and let $E:B(\mathcal{F})\to A$ be the natural
conditional expectation given by\begin{equation}
E(X)=\langle1_{A},X1_{A}\rangle_{A}.\label{eq:defOfE}\end{equation}
 Each vertex $v\in\Gamma$ determines a state on $B(\mathcal{F})$
given by\[
\phi_{v}(X)=\delta_{v}\circ E(X),\]
where $\delta_{v}:A\to\mathbb{C}$ is the point evaluation at $v$.
Then\[
\phi=\sum_{v}\phi_{v}\]
 is a weight on $B(\mathcal{F})$. Note that $\phi$ is finite on
all finite words in $c(e):e\in E$, and therefore defines a semifinite
weight on the von Neumann algebra $W^{*}(c(e):e\in E)$ (in the GNS
representation associated to $\phi$).

\begin{lem}
\label{lem:NoAtoms}(i) The weight $\phi$ and the conditional expectation
$E$ are faithful on this algebra. 

(ii) The modular group of $\phi$ is determined by $\sigma_{t}^{\phi}(c(e))=\left[\frac{\mu(t(e))}{\mu(s(e))}\right]^{it}c(e)=\sigma(e)^{2it}c(e)$. 

(iii) Consider $ $$\cup=\bigoplus_{v\textrm{ even}}\sum_{e\in\Gamma_{+}(v)}\sigma(e)c(e)c(e^{o})$,
where $\Gamma_{+}(v)$ denotes the set of all edges that start at
$v$. Then for each $v$, the law of $\cup$ with respect to $\phi_{v}$
has no atoms and is the free Poisson law with $R$-transform $\delta(1-z)^{-1}$.
In particular, $v\cup v$ is bounded for all $v$ and thus the (possibly
infinite) direct sum defining $\cup$ yields a bounded operator.
\end{lem}
\begin{proof}
The GNS vector space $\mathcal{F}_{v}$ associated to $\phi_{v}$
can be identified with the subspace of the Fock space $F(H)=\mathbb{C}v\oplus\bigoplus_{k\geq1}H^{\otimes k}$
spanned by tensors of the form $e_{1}\otimes\cdots\otimes e_{n}$,
$e_{j}\in E$ for which $e_{1}\cdots e_{n}$ form a path (i.e., are
{}``composable'': $s(e_{j})=t(e_{j+1})$). If we denote by $\hat{\ell}(e):F(H)\to F(H)$
the operator $\hat{\ell}(e)\xi=e\otimes\xi$ and by $\hat{c}(e)$
the operator $\hat{c}(e)=\hat{\ell}(e)+\hat{\ell}(e^{o})^{*}$, then
we have\[
P\hat{c}(e)P=c(e),\qquad P:F(H)\to\mathcal{F}_{v}\textrm{ orthogonal projection.}\]
Let $P$ be the set of all paths in $\Gamma$ and $P(v)$ be the set
of all paths starting at $v$. For a path $w=e_{1}\cdots e_{n}\in P(v)$,
let $c(w)=c(e_{1})\cdots c(e_{n})$ and similarly for $\hat{c}.$
We then see that the joint laws associated to the vacuum expectation
state of the variables\[
\{c(w):w\in P(v)\}\textrm{ and }\{\hat{c}(w):w\in P(v)\}\]
have the same law. Indeed, $\hat{c}(w)v=c(w)v$ if $w\in P_{v}$. 

It follows that the von Neumann algebra generated by $(A,c(e):e\in E)$
in the GNS representation $\pi_{v}$ associated to $\phi_{v}$ can
be embedded into the von Neumann algebra $W^{*}(\hat{c}(e):e\in E)$
in such a way that the restriction of the state $\hat{\phi}_{v}=\langle v,\cdot v\rangle$
to the former algebra is exactly $\phi_{v}$. But it is known \cite{shlyakht:quasifree:big}
that $\hat{\phi}_{v}$ is faithful, and so $\phi_{v}$ is faithful
(on the image in the GNS construction $\pi_{v}$). Furthermore, the
modular group of $\hat{\phi}_{v}$ is given by\[
\sigma_{t}^{\hat{\phi}_{v}}(\hat{c}(e))=\left[\frac{\mu(t(e))}{\mu(s(e))}\right]^{it}\hat{c}(e).\]
It follows that\[
\sigma_{t}^{\phi_{v}}(\pi_{v}(c(e)))=\left[\frac{\mu(t(e))}{\mu(s(e))}\right]^{it}\pi_{v}(c(e)).\]

It is clear that the GNS vector space for $\phi$ is just the direct
sum of the GNS vector spaces for $\phi_{v}$ taken over all vertices
$v$. Thus $\phi$ is faithful and so $E$ is faithful (on the possibly
larger algebra $W^{*}(A,c(e):e\in E)$. Thus (i) holds.

Let now\[
Y=\sum_{e\in\Gamma_{+}(v)}\sigma(e)\hat{c}(e)\hat{c}(e^{o}).\]
Then $Y$ has the same law for $\hat{\phi}_{v}$ as does $\cup$ for
$\phi_{v}$. Note that $Y=\sum_{e\in\Gamma_{+}(v)}b(e)$, with $b(e)=\sigma(e)\hat{c}(e)\hat{c}(e^{o})$.
Thus $b(e)$ are free and so the law of $Y$ satisfies\[
\mu_{Y}=\boxplus_{e\in\Gamma_{+}(v)}\mu_{b(e)}.\]
Now, for each $e$, $b(e)^{1/2}$ has free Poisson distribution with
$R$-transform $\mu(t(e))/\mu(s(e))\cdot(1-z)^{-1}$ (see \cite[Remark 4.4 on p. 347]{shlyakht:quasifree:big}).
Thus the law of $b(e)$ has only one atom of mass $\alpha(e)=1-\mu(t(e))/\mu(s(e))$
at zero (this expression is to be interpreted as zero if it is negative).
It follows from additivity of $R$-transform \cite{DVV:book} that
the law of $Y$ is free Poisson with $R$-transform\[
(1-z)^{-1}\sum_{e\in\Gamma_{+}(v)}\frac{\mu(t(e))}{\mu(v)}=\frac{\delta}{1-z},\]
which will have an atom iff $\delta<1$. Since $\delta\geq1$, the
law of $Y$ has no atoms. Thus (iii) holds.

Finally, it is also clear that (ii) holds since a similar formula
holds in the GNS representation of each $\phi_{v}$ and $\phi=\sum\phi_{v}$.
\end{proof}
\begin{lem}
Let $L$ be the set of all loops in $\Gamma$. Then the algebra $W^{*}(c(w):w\in L)$
belongs to the fixed point of the modular group acting on the algebra
$W^{*}(c(w):w\in P)$.
\end{lem}
\begin{proof}
Since $w$ is a loop, the factors $\mu(t(e))/\mu(s(e))$ associated
to each factor in $c(w)$ cancel. 
\end{proof}

\subsection{The conditional expectation $E$ realizes $Tr_{0}$.}

Let $Y_{w}=c(w)$, $w\in L_{k}^{+}$ (the set of all loops starting
at a positive vertex and of length $2k$).

\begin{lem}
Let $w\in L$ be a loop given by $w=e_{1}\cdots e_{n}$. Then $\phi(Y_{w})=\sum_{\pi\in NC(2k)}\prod_{\{i,j\}\subset\pi}\delta_{e_{i}=e_{j}^{o}}\sigma(e_{i})$.
In particular, $E(Y_{w})=Tr_{0}(w)$.
\end{lem}
\begin{proof}
$E(Y_{w})=\sum_{\pi\in NC(2k)}\prod_{\{i,j\}\subset\pi}\delta_{e_{i}=e_{j}^{o}}\cdot E(Y_{e_{i}}Y_{e_{j}})s(e_{i})$.
Moreover, $E(c(e_{i})c(e_{i})^{*})=s(e_{i})\sigma(e_i)$.
The rest follows from Lemma \ref{lemma:Tr0Formula}.
\end{proof}
\begin{lem}
\label{lemma:embedFreeGroup}Let $L^{+}$ be the set of loops starting
at an even vertex. Consider $\mathcal{M}_{0}=W^{*}(c(w):w\in L)$
with its semi-finite weight $\phi$. Then each even $v\in\Gamma$,
defines a central projection in $\mathcal{M}_{0}$ and\[
(\mathcal{M}_{0},\phi)=\bigoplus_{v\textrm{\ even}}(v\mathcal{M}_{0}v,\phi_{v}).\]
For each $v$, the algebra $v\mathcal{M}_{0}v$ can be canonically
embedded into a free group factor.
\end{lem}
\begin{proof}
If $w\in L$ is a loop starting at $v$, then $v'c(w)=c(w)v'=\delta_{v=v'}c(w)$. 

We have seen before that $\pi_{v}(W^{*}(c(w):w\in P))$ with its state
$\phi_{v}$ can be embedded into a free Araki-Woods factor associated
to $H$ and taken with its free quasi-free state, in a state-preserving
way. The image of $\mathcal{M}_{0}$ under $\pi_{v}$ is precisely
$v\mathcal{M}_{0}v$, and this image clearly lies in the centralizer
of the free quasi-free state. The free quasi-free state is periodic
(the modular group, restricted to $c(H)$ has as its eigenvectors
the edges of $\Gamma$) and therefore the centralizer is a free group
factor. 
\end{proof}
Note that $\phi=\bigoplus\phi_{v}$ is faithful. 

\begin{thm}
Let $w\in L^{+}$ be a loop on $\Gamma$, starting at an even vertex.
Then the map $w\mapsto c(w)$ extends to a trace-preserving embedding
with dense range of $(Gr_{0}P^{\Gamma},\wedge_{0},Tr_{0})$ into $(\mathcal{M}_{0},E)$.
Thus $Tr_{0}$ is a faithful center-valued trace.
\end{thm}
\begin{proof}
Clearly the theorem is true on elements of $P_{+}^{\Gamma}$ that
have finite support, i.e., are finite linear combinations of loops.

We have to check that this embedding makes sense for elements of $P_{+}^{\Gamma}$
which, as functions on loops, have infinite support.

Let $w\in P_{k}^{\Gamma}$. Then for any $v\in\Gamma_{+}$, $\delta_{v}\wedge_{0}w=w\wedge_{0}\delta_{v}=\delta_{v}\wedge_{0}w\wedge_{0}\delta_{v}$
has finite support. Moreover, by assumption $\langle w,w^{*}\rangle\in P_{0}^{\Gamma}=\ell^{\infty}(\Gamma_{+})$
has finite $\ell^{\infty}$ norm. But the value of $\langle w,w^{*}\rangle$
at $v$ is exactly $\Vert c(\delta_{v}\wedge_{0}w\wedge_{0}\delta_{v})\Vert_{L^{2}(\phi_{v})}^{2}$
and is therefore uniformly bounded as a function of $v$. Moreover,
note that each $c(\delta_{v}\wedge_{0}w\wedge_{0}\delta_{v})$ belongs
to the span of words of length $2k$ in operators $c(e):e\in E_{+}$. 

The eigenvector condition implies that the ratios $\mu(v)/\mu(w)$
for $v,w$ adjacent are bounded, and also that the valence of the
graph is bounded. 

It follows that the linear dimension of the space of all loops of
length $k$ starting at a vertex $v$ is uniformly bounded, by a constant
independent of $v$. Moreover, the norms of the orthogonal basis for
this space (consisting of the various loops) are bounded both above
and below uniformly in $v$. Thus the restrictions of the operator
norm and the $L^{2}(\phi)$-norm to the finite-dimensional linear
span of loops of length $k$ starting at $v$ are equivalent, and
the constants in the equivalence can be chosen to be uniform in $v$.

It follows that $vc(w)v$ is uniformly bounded in norm (independent
of $v$). 

Since the projections $v:v\in\Gamma_{+}$ are orthogonal, it follows
that $c(w)$, defined as the ultraweakly-convergent sum $\sum vc(w)v$,
is a bounded operator in $\mathcal{M}_{0}$. 

Since the map $w\mapsto c(w)$ is bilinear over $P_{0}^{\Gamma}$,
and is an algebra homomorphism when restricted to finite linear combinations
of loops, it is easy to see that it is an algebra homomorphism on
all of $Gr_{0}P^{\Gamma},\wedge_{0}$.
\end{proof}

\subsection{The operator $\cup$.}

Let $\Gamma_{+}(v)$ denote the set of all edges starting at an even
vertex $v$ and let $E_{+}$ denote the set of all positively oriented
edges (i.e., ones that start at an even vertex). Recall that \[
\cup=\sum_{e\in E_{+}}\sigma(e)c(e)c(e)^{*}.\]
 If we let $\delta$ be the Perron-Frobenius eigenvalue, then\begin{eqnarray*}
(\cup) & = & \sum_{e\in E_{+}}\sigma(e)\left(\ell(e)\ell(e^{o})+(\ell(e)\ell(e^{o}))^{*}+\ell(e)\ell(e)^{*}+\sigma(e)\right)\\
 & = & 2\sum_{e\in E_{+}}\sigma(e)\Re(\ell(e)\ell(e^{o}))+\sum_{v}\sum_{e\in\Gamma_{+}(v)}\left[\frac{\mu(t(e))}{\mu(s(e))}\right]v+\sum_{e\in E_{+}}\sigma(e)\ell(e)\ell(e)^{*}\\
 & = & 2\sum_{e\in\Gamma_{+}(v)}\sigma(e)\Re(\ell(e)\ell(e^{o}))+\delta+\sum_{e\in E_{+}}\sigma(e)\ell(e)\ell(e)^{*}.\end{eqnarray*}
Here we used\[
\sum_{e\in\Gamma_{+}(v)}\mu(t(e))=\sum_{j}\Gamma_{vj}\mu_{j}=\delta\mu(v)\]
so that\[
\sum_{e\in\Gamma_{+}(v)}\frac{\mu(t(e))}{\mu(s(e))}=\frac{1}{\mu(v)}\sum_{e\in\Gamma_{+}(v)}\mu(t(e))=\frac{1}{\mu(v)}\delta\mu(v)=\delta.\]
Let $\mathcal{F}_{+}$ be the set of all vectors in \emph{$\mathcal{F}$
}starting and ending in a positive vertex and let $A_{+}=A\cap\mathcal{F}_{+}$.
Since if $\zeta\in\mathcal{F}\ominus A_{+}$\[
\sum_{e\in\Gamma_{+}(v)}\sigma(e)\ell(e)\ell(e)^{*}\zeta=\zeta,\]
 (because of the normalizations of the lengths of $e,e^{o}$ we have
that the sum $\sum_{e\in\Gamma_{+}(v)}\sigma(e)\ell(e)\ell(e)^{*}$
is the same as the sum $\sum_{f}\ell(f)\ell(f^{*})$ where the summation
is over an orthonormal basis). Thus\begin{equation}
\cup|_{\mathcal{F}_{+}}=2\sum_{e\in\Gamma_{+}(v)}\sigma(e)\Re(\ell(e)\ell(e^{o}))+\delta+(1-P),\label{eq:formulaCup}\end{equation}
where $P:\mathcal{F}_{+}\to\mathcal{F}_{+}$ is the projection onto
$A_{+}\subset\mathcal{F}_{+}$ and $\delta$ is the Perron-Frobenius
eigenvalue.

As consequence, we note that we can now identify the position of $\mathcal{A}_{v}=\pi_{v}(W^{*}(Y))=vW^{*}(Y)v$
inside of \emph{$v\mathcal{F}v=L^{2}(W^{*}(c(w):w\in L),\phi_{v})$}
identified with a subspace of $\mathbb{C}v\oplus\bigoplus_{k\geq1}H^{\otimes k}$:

\begin{lem}
\label{lem:xiandL2A}Let $v$ be an even vertex. Then $L^{2}(\mathcal{A}_{v})$
is the closed linear span of the orthogonal system of vectors\[
\xi^{\otimes k}=\left(\sum_{e\in\Gamma_{+}(v)}\sigma(e)e\otimes e^{o}\right)^{\otimes k},\qquad k=0,1,\ldots.\]
Moreover,\[
\Vert\xi^{\otimes k}\Vert_{2}^{2}=\delta^{k},\]
where $\delta$ is the Perron-Frobenius eigenvalue.
\end{lem}
\begin{proof}
We note that $Yv=\xi$. Moreover, the linear span of $\xi^{\otimes k}$
is clearly stable under the action of $Y$. Thus it is sufficient
to prove that if $\xi^{\otimes r}$ for $r<k$ are in $L^{2}(\mathcal{A})$
then also $\xi^{\otimes k}\in L^{2}(\mathcal{A})$. But this follows
from noting that $Y^{k}v=\xi^{\otimes k}+\zeta$, where $\zeta$ is
a tensor of smaller degree in $L^{2}(\mathcal{A})$. 

Furthermore,\[
\langle\xi,\xi\rangle=\sum_{e\in\Gamma_{+}(v)}\left[\frac{\mu(t(e))}{\mu(s(e))}\right]\Vert e\otimes e^{o}\Vert_{2}^{2}=\sum_{e\in\Gamma_{+}(v)}\frac{\mu(t(e))}{\mu(v)}=\frac{1}{\mu(v)}\sum_{j}\Gamma_{vj}\mu_{j}=\delta.\]

\end{proof}

\subsection{Relative commutant of $\cup$.}

Recall that $L^{2}(\mathcal{M}_{0},\phi)=\bigoplus_{v}L^{2}(v\mathcal{M}_{0}v,\phi_{v})=\bigoplus_{v}v\mathcal{F}v$.
Let as before $Y=\cup$ , $\mathcal{A}=W^{*}(Y)$ and $\mathcal{A}_{v}=v\mathcal{A}v$. 

\begin{lem}
\label{lem:relCommcup}(i) $\mathcal{A}_{v}$ is a singular MASA in
$v\mathcal{M}_{0}v$. \\
(ii) $W^{*}(\cup)'\cap\mathcal{M}_{0}=\bigoplus_{v\textrm{ even}}v\mathcal{A}v$.
\\
(iii) Consider the algebra $\mathcal{N}_{+}=W^{*}(c(w):w\textrm{ path in }\Gamma\ \textrm{starting and ending at an even vertex})$.
Then $\mathcal{A}'\cap\mathcal{N}_{+}=\bigoplus_{v\in\Gamma_{+}}v\mathcal{A}v$.
\end{lem}
\begin{proof}
We first note that any $v\in V$ commutes with $Y=\cup$. In particular,
$v\in\mathcal{A}'\cap\mathcal{N}_{+}$. Hence $[Y,x]=0$ implies that
$v[Y,x]w=[Y,vxw]=0$ for all $v,w\in V$. Hence $\mathcal{A}'\cap\mathcal{N}_{+}$
is the closure of $\sum_{v,w}(\mathcal{A}'\cap v\mathcal{N}_{+}w)$.

Consider the full Fock space $F(H)$ as in the proof of Lemma \ref{lem:NoAtoms},
where $H$ is as before a Hilbert space having as basis edges of $\Gamma$.
Thus $F(H)$ is spanned by all tensors of the form $e_{i_{1}}\otimes\cdots\otimes e_{i_{m}}$,
where $e_{i_{k}}\in E$. Let $\tilde{H}=H\otimes H$, and consider
$\tilde{\mathcal{F}}\subset F(H)$ given by $\tilde{\mathcal{F}}=\bigoplus_{k\geq0}\tilde{H}^{\otimes k}$.
Let $T=W^{*}(\hat{\ell}(e)+\hat{\ell}(e^{o})^{*}:e\in E)$ acting
on $F(H)$, and consider the subalgebra $Q\subset T$ given by $Q=W^{*}([\hat{\ell}(e)+\hat{\ell}(e^{o})^{*}][\hat{\ell}(f)+\hat{\ell}(f^{o})^{*}]:e,f\in E)$.
Then clearly $L^{2}(Q)\subset L^{2}(P)$ can be identified with $\tilde{F}\subset F(H)$.
Furthermore, $Q$ is invariant under the modular group associated
to $\phi_{v}$ (the vector state associated to the vacuum vector in
$\tilde{\mathcal{F}}\subset F(H)$). Thus the modular group of $Q$
is the restriction to $Q$ of the modular group of $P$. 

Fix $v,w\in V$. Denote by $\lambda_{v}$ the element $\sum_{e\in\Gamma_{+}(v)}\sigma(e)(\hat{\ell}(e)+\hat{\ell}(e^{o})^{*})(\hat{\ell}(e^{o})+\hat{\ell}(e)^{*})\in Q$.
Denote by $\rho_{w}$ the element $\sum_{e\in\Gamma_{+}(w)}\sigma(e)(\hat{r}(e)+\hat{r}(e^{o})^{*})(\hat{r}(e^{o})+\hat{r}(e)^{*})\in JQJ$
(here $\hat{r}$ denotes the right creation operator and $J$ is the
modular conjugation). Note that $\rho_{v}=J\lambda_{v}J$. 

We now make the identification $U:L^{2}(v\mathcal{N}_{+}w)\hookrightarrow L^{2}(Q)$
obtained by sending a tensor $e_{1}\otimes_{A}\cdots\otimes_{A}e_{2n}$
associated to a \emph{path} $e_{1}\cdots e_{2n}$ starting at $v$
and ending at $w$ to the tensor $e_{1}\otimes\cdots\otimes e_{2n}$.
It is not hard to see that\[
\lambda_{v}U=UY,\qquad\rho_{w}U=UJYJ.\]
It follows that the laws of $\rho_{w}$ and $\lambda_{v}$ (with respect
to the vacuum state on $\tilde{\mathcal{F}}(H)$) are the same as
that of $Y$ and have no atoms; thus $W^{*}(\lambda_{v})$ and $W^{*}(\rho_{w})$
are diffuse. In particular, if $\Xi\in L^{2}(v\mathcal{N}_{+}w)$
satisfies $Y\Xi=JYJ\Xi$, then $U\Xi$ satisfies $\lambda_{v}U\Xi=\rho_{w}U\Xi$. 

Consequently, we would prove (iii) if we could show: 

(a) if $v\neq w$, $\lambda_{v}\zeta=\rho_{w}\zeta$ for $\zeta\in UL^{2}(v\mathcal{N}_{+}w)$
only occurs if $\zeta=0$ and 

(b) if $v=w$ and $\lambda_{v}\zeta=\rho_{v}\zeta$ for some $\zeta\in UL^{2}(v\mathcal{N}_{+}v),$
then $\zeta\in UL^{2}(v\mathcal{A}v)$.

Let $\xi_{v}=\sum_{e\in\Gamma_{+}(v)}\sigma(e)e\otimes e^{o}$. 

Assume first that $u=v$. Let $K_{v}=\tilde{H}\ominus\mathbb{C}\xi_{v}$.
Put $\mathcal{H}_{v}=\mathbb{C}\Omega\oplus\mathbb{C}\xi_{v}\oplus\mathbb{C}\xi_{v}^{\otimes2}\oplus\cdots$.
Then $\mathcal{H}_{v}=UL^{2}(v\mathcal{A}v)$ in such a way that the
left and right multiplication by $Y$ on $L^{2}(v\mathcal{A}v)$ correspond
to the actions of $\lambda_{v}$ and $\rho_{v}$. In particular, $\mathcal{H}_{v}$
is invariant under both $\lambda_{v}$ and $\rho_{v}$.

The image of $L^{2}(v\mathcal{N}_{+}v)$ under $U$ lies in the closure
of the direct sum\[
\mathcal{H}_{v}\oplus(\mathcal{H}_{v}\otimes K_{v}\otimes\mathcal{H}_{v})\oplus(\mathcal{H}_{v}\otimes K_{v}\otimes\mathcal{H}_{v}\otimes K_{v}\otimes\mathcal{H}_{v})\oplus\cdots.\]
(This direct sum is identified with a subspace $\tilde{\mathcal{F}}$
by identifying $\Omega\otimes\zeta$ and $\zeta\otimes\Omega$ with
$\zeta$ if $\zeta\in F(H)$). Each direct summand in this sum is
invariant under both $\rho_{v}$ and $\lambda_{v}$ since their actions
respect the tensor product decompositions $\mathcal{H}_{v}\otimes K_{v}\otimes\cdots\otimes K_{v}\otimes\mathcal{H}_{v}$:
$\rho_{v}$ acts as $\textrm{id}\otimes\rho_{v}|_{\mathcal{H}_{v}}$
and $\lambda_{v}$ acts as $\lambda_{v}|_{\mathcal{H}_{v}}\otimes\textrm{id}$. 

Now, for any choice of orthonormal basis for $K_{v}\otimes\mathcal{H}_{v}\otimes\cdots\otimes K_{v}$,
$\zeta_{\alpha}$, we have for all $h,g\in\mathcal{H}_{v}$:\[
\langle h\otimes\zeta_{\alpha}\otimes g,h'\otimes\zeta_{\alpha'}\otimes g'\rangle=\delta_{\alpha=\alpha'}\langle h,h'\rangle\langle g,g'\rangle\]
and consequently $\mathcal{H}_{v}\otimes K_{v}\otimes\cdots\otimes\mathcal{H}_{v}$
is isomorphic to an (infinite) multiple of $\mathcal{H}_{v}\otimes\mathcal{H}_{v}$
as a bimodule over $W^{*}(\lambda_{v})$ acting on the left copy of
$\mathcal{H}_{v}$ and $W^{*}(\rho_{v})=JW^{*}(\lambda_{v})J$ acting
on the right copy of $\mathcal{H}_{v}$. Since the spectral measure
of $\lambda_{v}$ is non-atomic, it follows that there can be no vector
$\Xi$ contained in $\mathcal{H}_{v}\otimes K_{v}\otimes\cdots\otimes\mathcal{H}_{v}$
satisfying $\lambda_{v}\Xi=\rho_{v}\Xi$, since such a vector would
give rise (via an isomorphism of $\mathcal{H}_{v}\otimes\mathcal{H}_{v}$
with Hilbert-Schmidt operators on this space) to a Hilbert-Schmidt
operator on $\mathcal{H}_{v}$, commuting with $\lambda_{v}$. 

Thus the only possible $\Xi$ satisfying $\lambda_{v}\Xi=\rho_{v}\Xi$
and lying in the image of $UL^{2}(v\mathcal{N}_{+}v)$ must be contained
in $\mathcal{H}_{v}=UL^{2}(v\mathcal{A}v)$. Thus we have proved (b).

To prove (a), we note that if $v\neq w$, and we let $K_{v,w}=\tilde{H}\ominus(\mathbb{C}\xi_{v}\oplus\mathbb{C}\xi_{w})$,
the image of $L^{2}(v\mathcal{N}_{+}w)$ lies in\[
\mathcal{H}_{v}\otimes\left(\bigoplus_{k\geq0}\bigoplus_{u_{1},\ldots,u_{k}\in\{v,w\}}K_{v,w}\otimes\mathcal{H}_{u_{1}}\otimes K_{v,w}\otimes\mathcal{H}_{u_{2}}\cdots\otimes K_{v_{w}}\right)\otimes\mathcal{H}_{w}\]
(once again identified with a subspace of $\tilde{\mathcal{F}}$ as
before), which is isomorphic to an infinite multiple of $\mathcal{H}_{v}\otimes\mathcal{H}_{w}$
as a bimodule over $W^{*}(\lambda_{v})$ acting on the left copy of
$\mathcal{H}_{v}$ and $W^{*}(\lambda_{w})$ acting on the right copy
of $\mathcal{H}_{w}$. Once again, we see that there can be no vector
$\Xi$ satisfying $\lambda_{v}\Xi=\rho_{w}\Xi$ in this space. Thus
(a) is also proved. Thus we have proved (iii).

Note that we have actually proved that $L^{2}(v\mathcal{M}_{0}v,\phi_{v})$
when viewed as a bimodule over $\mathcal{A}_{v}=W^{*}(vYv)$ is the
direct sum of $L^{2}(\mathcal{A}_{v})$ and an (infinite) multiple
of the coarse $\mathcal{A}_{v},\mathcal{A}_{v}$-bimodule $L^{2}(\mathcal{A}_{v})\bar{\otimes}L^{2}(\mathcal{A}_{v})$.
Recall (see e.g. \cite{popa-shlyakht:cartan} or \cite{feldman-moore})
that the normalizer of $\mathcal{A}_{v}$ is contained in its quasi-normalizer
$\mathcal{NQ}(\mathcal{A}_{v})$, which consists of those elements
$\zeta$ in $v\mathcal{M}_{0}v$ for which the associated bimodule
$\overline{\mathcal{A}_{v}\zeta\mathcal{A}_{v}}^{L^{2}}$ is {}``discrete''.
This bimodule cannot be discrete if it contains a sub-bimodule isomorphic
to a compression of the coarse $\mathcal{A}_{v},\mathcal{A}_{v}$-bimodule.
Thus the only $\zeta\in\mathcal{QN}(\mathcal{A}_{v})$ must lie in
$L^{2}(\mathcal{A}_{v})$ and thus in $\mathcal{A}_{v}$. It follows
that the normalizer of $\mathcal{A}_{v}$ in $v\mathcal{M}_{0}v$
is contained in $\mathcal{A}_{v}$. Thus $\mathcal{A}_{v}$ is a singular
MASA and so (i) follows. Now (ii) easily follows from (i). 
\end{proof}

\subsection{The operator $\Cup$, relative commutant of $W^{*}(\cup,\Cup)$ and
factoriality.}

We now consider the following sum \[
\Cup=\sum_{eff^{o}e^{o}\in L^{+}}\left[\frac{\mu(t(f))}{\mu(s(e))}\right]^{1/2}c(e)c(f)c(f)^{*}c(e)^{*}\]
taken over all loops that start at an even vertex. The pictorial representation
of this planar algebra element is:
\[
\Cup=\rnode{c}{\psframebox[framearc=0.4]{\vbox to0.8em{\vfill\hbox to3.2em{\hfill}}}}
\ncbar[arm=-0.6em,linearc=0.3em,offsetA=0.5em,offsetB=0.5em,angle=90]{c}{c}
\ncbar[arm=-0.9em,linearc=0.3em,offsetA=0.9em,offsetB=0.9em,angle=90]{c}{c}
\nput[labelsep=-0.20]{167}{c}{*}
.\]
\begin{lem}
\label{lem:factorWhenNotStar}Let $v$ be a fixed even vertex. Assume
that there is a path of length $2$ from $v$ to $v$ not of the form
$ee^{o}ff^{o}$. Then algebra $vW^{*}(\cup,\Cup)v$ has a trivial
relative commutant inside of the algebra $v\mathcal{N}_{+}v$, where
$\mathcal{N}_{+}=W^{*}(c(w):w\textrm{ path in }\Gamma\ \textrm{starting and ending at an even vertex})$. 
\end{lem}
\begin{proof}
Because of Lemma \ref{lem:relCommcup}, we know that the relative
commutant of $vW^{*}(\cup,\Cup)v$ inside of $v\mathcal{N}_{+}v$
is contained in $vW^{*}(\cup)v=\mathcal{A}_{v}$. 

Let $\eta=e_{1}\otimes f_{1}\otimes f_{1}^{o}\otimes e_{1}^{o}$ where
$f_{1}\neq e_{1}^{o}$ and $e_{1}f_{1}f_{1}^{o}e_{1}^{o}$ is a path
from $v$ to $v$. 

Set\[
Z=v\Cup v=\sum_{e,f^{o}}\sigma(e)\sigma(f)c(e)c(f)c(f^{o})c(e^{o}),\]
where the sum is over all paths $eff^{o}e^{o}$ from $v$ to $v$.
Then if $k,l>0$, and $\xi\in L^{2}(vW^{*}(\cup)v)=L^{2}(\mathcal{A}_{v})$
is as in Lemma \ref{lem:xiandL2A}, we have: \begin{eqnarray*}
\langle\eta\otimes\xi^{\otimes k},[Z,\xi^{\otimes l}]\rangle & = & \langle\eta\otimes\xi^{\otimes k},\sum_{e,f}\sigma(e)\sigma(f)c(e)c(f)c(f^{o})c(e^{o})\xi^{\otimes l}-\sigma(e)\sigma(f)\xi^{\otimes l}c(e)c(f)c(f^{o})c(e^{o})\rangle\\
 & = & \langle\eta\otimes\xi^{\otimes k},\sum_{e,f}\sigma(e)\sigma(f)e\otimes f\otimes f^{o}\otimes e^{o}\otimes\xi^{\otimes l}\\
 &  & +\sum_{e,f}\sigma(e)\sigma(f)\frac{\mu(t(e))^{1/2}}{\mu(v)^{1/2}}e\otimes f\otimes f^{o}\otimes\xi^{\otimes(l-1)}\rangle.\end{eqnarray*}
Thus\[
\langle\eta\otimes\xi^{\otimes k},[Z,\xi^{\otimes l}]\rangle=\begin{cases}
\frac{\mu(t(f_{1}))^{1/2}}{\mu(v)^{1/2}}, & l=k\\
\frac{\mu(t(f_{1}))^{1/2}}{\mu(v)^{1/2}}\cdot\frac{\mu(t(e_{1}))^{1/2}}{\mu(v)^{1/2}}, & l=k+1\\
0, & \textrm{otherwise}.\end{cases}\]
Thus if we consider\[
a=\sum\alpha_{k}\xi^{\otimes k}\in L^{2}(\mathcal{A}_{v})\]
and assume that $[a,Z]=0$ and $a\perp\mathbb{C}v$ (so that $\alpha_{0}=0$),
we obtain:\begin{eqnarray*}
0 & = & \langle\eta\otimes\xi^{\otimes k},[Z,a]\rangle\\
 & = & \frac{\mu(t(f_{1}))^{1/2}}{\mu(v)^{1/2}}\left(\alpha_{k}+\frac{\mu(t(e_{1}))^{1/2}}{\mu(v)^{1/2}}\alpha_{k+1}\right),\qquad k\geq1.\end{eqnarray*}
Since the choice of $e_{1}$ was arbitrary, we find that \[
\alpha_{k}=-\frac{\mu(t(e))^{1/2}}{\mu(v)^{1/2}}\alpha_{k+1},\qquad\forall e\in\Gamma_{+}(v).\]
If $a\neq0$, not all $\alpha_{k}$ are zero; from this recursive
relation we deduce that $\mu(t(e))$ are all equal to the same value,
$\mu'$, independent of $e\in\Gamma_{+}(v)$ and that (after rescaling
$a$ by a non-zero constant) we may assume that $\alpha_{k+1}=(-1)^{k}\lambda^{-(k+1)}$
where $\lambda=(\mu(t(e))/\mu(v))^{1/2}=(\mu'/\mu(v))^{1/2}$. 

On the other hand,\[
\sum\Gamma_{vj}\mu'=\delta\mu(v)\]
so that\[
(\sum\Gamma_{vj})\mu'/\mu(v)=(\sum\Gamma_{vj})\lambda^{2}=\delta.\]
Thus if $N\geq1$ is the valence of $\Gamma$ at $v$, we find that
$N\lambda^{2}=\delta$, so that $\lambda^{2}=\delta/N$. 

Using the fact that $\Vert\xi\Vert_{2}^{2}=\delta$, we compute:\[
\Vert a\Vert_{2}^{2}=\sum_{k}|\alpha_{k}|^{2}\Vert\xi^{\otimes k}\Vert_{2}^{2}=\sum_{k}\lambda^{-2k}\delta^{k}=\sum_{k}(N/\delta)^{k}\delta^{k}=\sum_{k}N^{k}=\infty,\]
which is impossible. Thus $[Z,a]=0$ forces $a\in\mathbb{C}v$.
\end{proof}
\begin{lem}
\label{lem:whenStar}Let $\Gamma$ be a connected bi-partite graph
with $N+1$ vertices, $v\in\Gamma$ even and assume that the hypothesis
of Lemma \ref{lem:factorWhenNotStar} is not satisfied. Then the remaining
vertices $e_{1},\ldots,e_{N}$ of $\Gamma$ are all connected to $v$
by a single edge, and $\Gamma$ has no other edges.
\end{lem}
We can now prove that the relative commutant of $W^{*}(\cup,\Cup)$
can be controlled, if the graph $\Gamma$ is not too small. The cases
we exclude are $A_{1}$ (a graph with a single vertex and no edges)
and $A_{2}$ (a graph with exactly two vertices connected by a single
edge). In these cases, $\cup$ and $\Cup$ commute (in fact, they
are equal). In either of these cases, the Perron-Frobenius eigenvalue
is $1$, which is of little interest to us. 

\begin{thm}
\label{thm:relCommutantTrivial}Assume that $\Gamma\neq A_{2}$ and
$\Gamma\neq A_{1}$ and let $v\in\Gamma$ even. Then (i) the relative
commutant $(vW^{*}(\cup,\Cup)v)'\cap v\mathcal{M}_{0}v$ is trivial.
In particular, the center of $\mathcal{M}_{0}$ is the algebra $A_{+}=\ell^{\infty}(\textrm{even vertices})$.
(ii) $W^{*}(\cup,\Cup)'\cap\mathcal{N}_{+}=P_{0}^{\Gamma}$ (where
$\mathcal{N}_{+}=W^{*}(c(w):w\textrm{ path that starts and ends at an even vertex})$,
and $P_{0}^{\Gamma}=\bigoplus_{v\textrm{ even}}v\mathbb{C}$).
\end{thm}
\begin{proof}
Because of Lemma \ref{lem:factorWhenNotStar} and Lemma \ref{lem:whenStar},
it remains to consider the case in which $\Gamma$ is a graph with
$N+1>3$ vertices $v,e_{1},\ldots,e_{N}$ with a single edge between
$v$ and each $e_{j}$ and no other edges. Since $\Gamma=\left[1\ \cdots\ 1\right]$,
we find that $\Vert\Gamma\Vert=N$ and therefore $\delta=N$. Moreover,
one can normalize the Perron-Frobenius eigenvector to be $\mu(e)=1$
for all $e\in\{v,e_{1},\ldots,e_{n}\}$. 

Thus\[
\xi=\sum_{j}e_{j}\otimes e_{j}^{o},\qquad Z=v\Cup v=\sum_{i}c(e_{i})c(e_{i}^{o})c(e_{i})c(e_{i}^{o}).\]
Let $k>1$. Then \begin{multline*}
[Z,\xi^{\otimes k}] =  \sum_{i}c(e_{i})c(e_{i}^{o})c(e_{i})c(e_{i}^{o})\xi^{\otimes k}-\xi^{\otimes k}c(e_{i})c(e_{i}^{o})c(e_{i})c(e_{i}^{o})\\
 = \xi^{\otimes k-2}+4\xi^{\otimes k-1}+6\xi^{\otimes k}+3\xi^{\otimes k+1}+\sum_{i}e_{i}\otimes e_{i}^{o}\otimes e_{i}\otimes e_{i}^{o}\otimes\xi^{\otimes k-1}+\sum_{i}e_{i}\otimes e_{i}^{o}\otimes e_{i}\otimes e_{i}^{o}\otimes\xi^{\otimes k}\\
  -\xi^{\otimes k-2}-4\xi^{\otimes k-1}-6\xi^{\otimes k}-3\xi^{\otimes k+1}-\sum_{i}\xi^{\otimes k-1}\otimes e_{i}\otimes e_{i}^{o}\otimes e_{i}\otimes e_{i}^{o}+\sum_{i}\xi^{\otimes k}e_{i}\otimes e_{i}^{o}\otimes e_{i}\otimes e_{i}^{o}\\
  = \sum_{i}e_{i}\otimes e_{i}^{o}\otimes e_{i}\otimes e_{i}^{o}\otimes(\xi^{\otimes k}+\xi^{\otimes k-1})-(\xi^{\otimes k}+\xi^{\otimes k-1})\otimes\sum_{i}e_{i}\otimes e_{i}^{o}\otimes e_{i}\otimes e_{i}^{o}\\
  =  \zeta\otimes(\xi^{\otimes k}+\xi^{\otimes k-1})-(\xi^{\otimes k}+\xi^{\otimes k-1})\otimes\zeta,\end{multline*}
where we have set $\zeta=\sum_{i}e_{i}\otimes e_{i}^{o}\otimes e_{i}\otimes e_{i}^{o}$. 

Let $\eta=e_{1}\otimes e_{1}^{o}\otimes e_{2}\otimes e_{2}^{o}$.
Then $\eta\otimes\xi^{\otimes l}\perp\zeta\otimes\xi^{\otimes k}$
for all $l,k$. On the other hand, $\langle\eta\otimes\xi^{\otimes l},\xi^{\otimes k}\otimes\zeta\rangle=\delta_{k=l}\Vert\xi^{\otimes k-1}\Vert$.

It follows that for any $k>1$,

\begin{align*}
\langle\eta\otimes\xi^{\otimes l},[Z,\xi^{\otimes k}]\rangle & =\langle\zeta\otimes(\xi^{\otimes k}+\xi^{\otimes k-1})-(\xi^{\otimes k}+\xi^{\otimes k-1})\otimes\zeta,\eta\otimes\xi^{\otimes l}\rangle\\
= & -\langle\xi^{\otimes k}\otimes\zeta,\eta\otimes\xi^{\otimes l}\rangle-\langle\xi^{\otimes k-1}\otimes\zeta,\eta\otimes\xi^{\otimes l}\rangle\\
= & -\delta_{l=k}\Vert\xi^{\otimes(l-1)}\Vert-\delta_{l=k-1}\Vert\xi^{\otimes(l-1)}\Vert.\end{align*}
It follows that if $a=\sum\alpha_{k}\xi^{\otimes k}\in L^{2}(\mathcal{A}_{v})$,
and we assume that $[Z,a]=0$, then we get for all $l\geq2$:\begin{eqnarray*}
0=\langle[Z,a],\eta\otimes\xi^{\otimes l}\rangle & = & \sum_{k}\alpha_{k}\langle[Z,\xi^{\otimes k}],\eta\otimes\xi^{\otimes l}\rangle\\
 & = & -\alpha_{l}\Vert\xi^{\otimes(l-1)}\Vert-\alpha_{l+1}\Vert\xi^{\otimes(l-1)}\Vert.\end{eqnarray*}
It follows that $\alpha_{k}$, $k\geq2$, is a constant sequence.
But the sequence $\{\alpha_{k}\}$ is in $L^{2}$ and thus must be
zero.

It follows that $a\in L^{2}(\mathcal{A}_{v})$ commutes with $Z$,
then $a=\alpha_{0}1+\alpha_{1}\cup$. But in this case, $[a,Z]\Omega=\alpha_{1}[\cup,Z]\Omega$.
The only tensors of degree $6$ in $\cup Z\Omega$ are $\xi\otimes\zeta$,
and the only terms of this degree in $Z\cup\Omega$ are $\zeta\otimes\xi$,
which are not equal. Thus $[a,Z]=0$ implies that also $\alpha_{1}=0$
and so $a$ must be a scalar. 

To see (ii), note first that any $v\in\Gamma_{+}$ is a projection
in the relative commutant of $W^{*}(\cup,\Cup)'\cap\mathcal{N}_{+}$.
Since the projections corresponding to different vertices are orthogonal,
it follows that any element $x$ in the relative commutant is weakly-convergent
infinite sum $\sum_{v\in\Gamma_{+}}vxv$, where $vxv\in vW^{*}(\cup,\Cup)'v\cap v\mathcal{N}_{+}v=\mathbb{C}v$. 
\end{proof}

\subsection{Factoriality of $M_{0}$.}

Let $P$ be an (extremal) subfactor planar algebra embedded into the
planar algebra of some graph $\Gamma$. Thus $(Gr_{0}P,Tr_{0})$ can
be viewed as a subalgebra of $(Gr_{0}P^{\Gamma},Tr_{0})\subset\mathcal{M}_{0}$.
Moreover, $TL(1),TL(2)\subset Gr_{0}P$, and so $\cup$ and $\Cup$
both belong to $Gr_{0}P$. Therefore, the center of $W^{*}(Gr_{0}P,Tr_{0})$
is contained in the relative commutant of $W^{*}(\cup,\Cup)$ inside
of \emph{$\mathcal{M}_{0}$}. By Theorem \ref{thm:relCommutantTrivial},
this relative commutant is the intersection of the algebra $A_{0}$
identified with the zero box space in $P^{\Gamma}$. 

\begin{lem}
Assume that the zero-box space of $P$ is one-dimensional. Then $W^{*}(Gr_{0}P,Tr_{0})\cap A_{+}=\mathbb{C}1_{A_{+}}$.
\end{lem}
\begin{proof}
Note that $tr_{0}$ is the restriction to $W^{*}(Gr_{0}P,\wedge_{0},Tr_{0})$
of the conditional expectation $E$ from $\mathcal{M}_{0}$ onto $A_{+}$
(which is the center of $\mathcal{M}_{0}$). Since this conditional
expectation is normal, if $z\in W^{*}(Gr_{0}P,Tr_{0})\cap A_{+}$,
then $z=E(z)$. On the other hand, $z$ is the limit (in the weak-operator
topology) of some sequence $z_{i}\in Gr_{0}P$. For each $i$, $E(z_{i})=Tr_{0}(z_{i})$
belongs to the zero-box space of $P_{+}$. Since $\cup\in P$ and
the zero-box space is one-dimensional $E(z_{i})$ must be a multiple
of $\mathcal{E}(\cup)=\delta\mathbb{C}1_{A_{+}}$ and hence $z=E(z)\in\mathbb{C}1_{A_{+}}$.
\end{proof}
Thus if the zero box space of $P$ is one-dimensional, and since $W^{*}(\cup)$
is diffuse, we automatically get:

\begin{thm}
Let $P$ be a planar algebra with one-dimensional zero box space and
of index $\delta>1$. Then $M_{0}=W^{*}(Gr_{0}P,Tr_{0})$ is a type
II$_{1}$ factor.
\end{thm}
Since $M_{0}\subset\mathcal{M}_{0}$, we see by Lemma \ref{lemma:embedFreeGroup}
that $M_{0}$ can be actually embedded into a direct sum of free group
factors. In particular, $M_{0}$ has the Haagerup property and is
$R^{\omega}$-embeddable.

\section{Higher relative commutants.}

\subsection{The algebra $\mathfrak{M}_{1}$ and the trace $\phi_{1}$.}

We now proceed to define the algebra $M_{1}=W^{*}(Gr_{1}P,Tr_{1})$,
which will contain $M_{0}=W^{*}(Gr_{0}P,Tr_{0})$ as a subfactor. 

Let us denote by $\mathfrak{M}_{0}$ the image of the algebra $Gr_{0}P$
inside $M_{0}\subset\mathcal{M}_{0}$ acting on the Fock space $\mathcal{F}$
as in the previous section.

We first recall from Section \ref{sec:PlanarAlgebra} that if we identify
elements of $P^{\Gamma}$ with paths, then the multiplication $\wedge_{1}$
on $GrP_{1}^{\Gamma}$ can be expressed as follows. Let $w=e_{1}\cdots e_{n}$
and $w'=e_{1}'\cdots e_{m}'$ be two paths.  Denote by $D_1(w)$ the path obtained from 
$w$ by following the path $w$, but starting at the first point of $w$ (rather than its starting point).
 Then\[
D_1^{-1}(D_1(w)\wedge_{1}D_1(w'))=\sigma(e_{n})^{-1}\delta_{e_{n}=e_{1}'}e_{1}\cdots e_{n-1}e_{2}'\cdots e_{m}'.\]
(note that the factor $\sigma(e_{n})^{-1}$ is exactly the norm $\Vert e_{n}\Vert^{2}$).

To a path $w=e_{1}\cdots e_{n}=e_{1}w_{0}e_{n}$, where $w_{0}=e_{2}\cdots e_{n-1}$
we associate the variable\[
c_{1}(w)=\ell(e_{1})c(w_{0})\ell(e_{n}^{o})^{*}\in B(\mathcal{F}).\]

\begin{lem}
$Y_{D_1^{-1}(w)}^{(1)}Y_{D_1^{-1}(w')}^{(1)}=Y_{D_1^{-1}(w\wedge_{1}w')}^{(1)}$. 
\end{lem}
\begin{proof}
This follows from the relation $\ell(e)^{*}\ell(g)=\delta_{e=g}\Vert e\Vert^{2}$.
\end{proof}
Let us introduce the notation\[
\mathfrak{M}_{1}=\operatorname{span}\{c_{1}(w):w\in L_{-}\}\]
where $L_{-}$ is the set of all loops starting at an odd vertex. 

The vector space $\mathfrak{M}_{1}$ is an algebra with multiplication
$\wedge_{1}$. Thus $w\mapsto c_{1}(D_1^{-1}(w))$ is a $*$-homomorphism from
$Gr_{1}P^{\Gamma}$ onto $\mathfrak{M}_{1}$. The unit of $\mathfrak{M}_{1}$
is the element\[
\sum_{e\in E_{-}}\sigma(e)\ell(e)\ell(e)^{*}\]
(here $E_{-}$ is the set of all odd edges, i.e., ones \emph{ending}
at an even vertex).

\begin{lem}
Let $E_{-}$ be the set of all odd edges. Then the map\[
i:Y\mapsto\sum_{e\in E_{-}}\sigma(e)\ell(e)Y\ell(e)^{*}\]
defines a unital $*$-homomorphism from the algebra $\mathfrak{M}_{0}$
to the algebra $\mathfrak{M}_{1}$.
\end{lem}
\begin{proof}
We note that\begin{eqnarray*}
i(Y_{w})\cdot i(Y_{w'}) & = & \sum_{e\in E_{-}}\sigma(e)^{2}\Vert e\Vert^{2}\ell(e)Y_{w}Y_{w'}\ell(e)^{*}\\
 & = & \sum_{e\in E_{-}}\sigma(e)\ell(e)Y_{w}Y_{w'}\ell(e)^{*}.\end{eqnarray*}
Thus $i$ is a homomorphism. Moreover, $i$ is clearly $*$-preserving.
\end{proof}
We now define a tracial weight $\phi_{1}$ on $\mathfrak{M}_{1}$:\[
\phi_{1}(\ell(e)Y_{w}\ell(f)^{*})=\delta^{-1}\delta_{e=f}\sigma(e)^{-1}\phi(Y_{w})\]
(the first $\delta$ is the Perron-Frobenius eigenvalue; note that
$e=f$ forces $Y_{w}\in\mathfrak{M}_{0}$). In other words,\[
\phi_{1}(X)=\delta^{-1}\sum_{f\in E_{-}}\sigma(f)^{-1}\phi(\langle f,Xf\rangle_{A}).\]
The last observation shows that $\phi_{1}(X)$ is a non-negative functional. 

Moreover, for any $v\in\mathfrak{M}_{0}$, \begin{eqnarray*}
\phi_{1}(i(v)) & = & \delta^{-1}\sum_{f\in\Gamma_{-}(v)}\sigma(f)^{-1}\Vert f\Vert_{2}^{2}\\
 & = & \delta^{-1}\sum_{f\in\Gamma_{-}(v)}\left(\frac{\mu(s(f))}{\mu(t(f))}\right)=\delta^{-1}\sum_{j}\Gamma_{jv}\frac{\mu(j)}{\mu(v)}=\delta^{-1}\delta=1.\end{eqnarray*}
(here $\Gamma_{-}(v)$ denotes the set of edges ending at $v$). 

Finally, $\phi_{1}$ is a trace, since if $w,w'$ are two loops of
the form $e\hat{w}f^{o}$ and $e'\hat{w}'f'^{o}$ with $s(e)=t(f^{o})$,
$s(e')=t(f'^{o})$ then\begin{eqnarray*}
\phi_{1}(ww') & = & \delta_{e=f'}\delta_{f=e'}\delta^{-1}\Vert e\Vert^{4}\Vert f\Vert^{2}\phi(\hat{w}\hat{w}')\\
 & = & \delta_{e'=f}\delta_{f'=e}\delta^{-1}\Vert e'\Vert^{6}\phi(\hat{w}'\hat{w})=\phi_{1}(w'w)\end{eqnarray*}
since $e=f'$ and $\Vert f'\Vert=\Vert e'\Vert$.

We finally note that\begin{eqnarray*}
\phi_{1}(i(Y_{w})) & = & \sum_{e}\delta^{-1}\left(\frac{\mu(s(e))}{\mu(t(e))}\right)\phi(w)\\
 & = & \phi(w)\end{eqnarray*}
because $\mu$ is an eigenvector for the graph matrix. We summarize
these observations as the following

\begin{lem}
The weight $\phi_{1}$ is a semifinite faithful trace, and the inclusion
$i:(\mathfrak{M}_{0},\phi)\to(\mathfrak{M}_{1},\phi_{1})$ is trace-preserving.
\end{lem}
\begin{proof}
Let us consider $Y=\sum_{g,h\in E_{-}}\ell(g)x_{g,h}\ell(h)^{*}$,
$x_{g,h}\in W^{*}(c(e):e\in\Gamma)$ with $Y^{*}Y$ in the domain
of $\phi_{1}$. Then $Y^{*}Y=\sum_{g,h,g'}\ell(g)x_{g,h}x_{g',h}^{*}\ell(g')^{*}\Vert h\Vert^{2}$.
Moreover,\[
\phi_{1}(YY^{*})=\delta^{-1}\sum_{g,h}\Vert h\Vert^{2}\Vert g\Vert^{4}\phi(x_{g,h}x_{g,h}^{*})\]
and $x_{g,h}x_{g,h}^{*}\in\mathcal{M}_{0}$. Thus if $\phi_{1}(Y^{*}Y)=0$,
each of the positive terms in the sum above must be zero and so $\phi(x_{g,h}x_{g,h}^{*})=0$
for all $g,h$. It follows that $Y=0$.
\end{proof}
Define now the map $E_{1}:\mathfrak{M}_{1}\to\mathfrak{M}_{0}$ by\[
E_{1}(\ell(e)c(w)\ell(f)^{*})=\delta_{e=f}\delta^{-1}\left(\frac{\mu(s(e))}{\mu(t(e))}\right)c(w).\]
Note that\[
E_{1}(i(Y))=Y\]
 and moreover\begin{eqnarray*}
E_{1}(i(c(w))\ell(e)c(w)\ell(f)^{*}) & = & Y_{w}E_{1}(\ell(e)c(w')\ell(f)^{*})\\
E_{1}(\ell(e)c(w)\ell(f)^{*}i(c(w)')) & = & E_{1}(\ell(e)c(w)\ell(f)^{*})c(w')\end{eqnarray*}
so that $i\circ E_{1}:\mathfrak{M}_{1}\to i(\mathfrak{M}_{0})\subset\mathfrak{M}_{1}$
is an $\mathcal{A}_{0}$-linear projection. Moreover, we see that
$\phi_{1}=\phi\circ(i\circ E_{1})$ so that $\mathcal{E}_{1}=i\circ E_{1}$
is the trace-preserving conditional expectation from $\mathfrak{M}_{1}\to\mathfrak{M}_{0}$.
It follows that $\mathcal{E}_{1}$ extends also to the von Neumann
algebra generated by $\mathfrak{M}_{1}$.

\subsection{The algebras $\mathfrak{M}_{n}$ and traces $\phi_{n}$.}

The algebras $\mathfrak{M}_{n}$ with semi-finite traces $\phi_{k}$
are defined in a similar way. The algebra $\mathfrak{M}_{n}$ is the
linear span\[
\mathfrak{M}_{n}=\operatorname{span}\{\ell(e_{1})\cdots\ell(e_{n})c(w)\ell(f_{n})^{*}\cdots\ell(f_{1})^{*}:e_{1}\cdots e_{n}wf_{n}^{0}\cdots f_{1}^{o}\in L^{\pm}\}\]
where the parity of the loops is chosen to match the parity of $n$.
For a loop $e_{1}\cdots e_{n}wf_{n}^{o}\cdots f_{1}^{o}\in L_{\pm}$,
we set\[
c_{n}(e_{1}\cdots e_{n}wf_{n}^{o}\cdots f_{1}^{o})=\ell(e_{1})\cdots\ell(e_{n})c(w)\ell(f_{n})^{*}\cdots\ell(f_{1})^{*}\in\mathfrak{M}_{n}.\]
 Then map $L^+\ni w\mapsto c(D^{-1}_k(w))$ defines a $*$-homomorphism from $Gr_{k}P^{\Gamma}$ to
$\mathfrak{M}_{k}$. (Recall that $D_k(w))$ denotes the loop obtained by replacing the starting point in $w$ by the $k$-th point on the path $w$).

Let\begin{eqnarray*}
\phi_{n} & = & \delta^{-n}\sum_{w=f_{1}\cdots f_{n}\ }\left(\frac{\mu(s(f_{n}))}{\mu(t(f_{1}))}\right)^{1/2}\phi\circ\langle\cdot f_{n}\otimes\cdots\otimes f_{1},f_{n}\otimes\cdots\otimes f_{1}\rangle_{A}.\end{eqnarray*}
The inclusion $i=i_{n}^{n-1}:\mathfrak{M}_{n-1}\to\mathfrak{M}_{n}$
is given by\[
c_{n-1}(w)\mapsto\sum_{ewe^{o}\in L}\sigma(e)^{-1}\ell(e)c_{n-1}(w)\ell(e)^{*}.\]
One can check that $i$ is again a trace-preserving inclusion. The
conditional expectation $E_{n}:\mathcal{M}_{n}\to\mathcal{M}_{n-1}$
is given by\[
E_{n}(\ell(e)c_{n-1}(w)\ell(f)^{*})=\delta_{e=f}\delta^{-1}\left(\frac{\mu(s(e))}{\mu(t(e))}\right)c_{n-1}(w).\]
As before, set\[
\mathcal{E}_{n}=i\circ E_{n}:\mathfrak{M}_{n}\to i(\mathfrak{M}_{n-1})\subset\mathfrak{M}_{n}.\]
It is not hard to check that this is the unique trace-preserving $\mathfrak{M}_{n-1}$
linear conditional expectation from $\mathfrak{M}_{n}$ to $i(\mathfrak{M}_{n-1})$
and that the trace $\phi_{n}$ is faithful (the argument is exactly
the same as in the case $n=1$). Moreover, one can easily check that
$\mathcal{E}_{n}=Tr_{n}$ if we identify $\mathfrak{M}_{n}$ with
$Gr_{n}P^{\Gamma}$.

Let us set\[
i_{n}^{j}=i_{n}^{n-1}\circ\cdots\circ i_{j+1}^{j}:\mathfrak{M}_{j}\to\mathfrak{M}_{n},\qquad i_{n}=i_{n}^{0}.\]

Comparing these with the definitions of section \ref{sec:PlanarAlgebra}
we get:

\begin{thm}
The map $w\mapsto c_{n}(w)$ is a $*$-isomorphism from $Gr_{k}P^{\Gamma}$
onto $\mathfrak{M}_{k}$. The semifinite weight $\phi_{k}$ satisfies
$\phi_{k}(v\wedge_{0}Tr_{k}(x))=\phi_{k}(v\wedge_{0}x)$. In particular,
the trace $Tr_{k}$ is positive and faithful.
\end{thm}

\subsection{Higher relative commutants.}

We now let\[
\mathcal{M}_{k}=W^{*}(Gr_{k}P^{\Gamma},\phi_{k})=W^{*}(\mathfrak{M}_{k}).\]
Given a planar subalgebra $P\subset P^{\Gamma}$, we'll denote by
$M$ the subalgebra of $\mathcal{M}_{k}$ generated by elements from
$P$. In other words, $M_{k}=W^{*}(Gr_{k}P,Tr_{k})$.

We'll denote by $\cup_{n}$ and $\Cup_{n}$ the images in $\mathfrak{M}_{n}$
of $\cup,\Cup\in\mathfrak{M}_{0}$. Note that $\cup_{n},\Cup_{n}\in M_{n}$.

\begin{lem}
\label{lem:relComContainsPn}Let $e_{1}\cdots e_{n}f_{1}^{o}\cdots f_{n}^{o}$
be a loop in $L_{\pm}$ (parity according to $n$). Then the element  $Z=\ell(e_{1})\cdots\ell(e_{n})\ell(f_{1})^{*}\cdots\ell(f_{n})^{*}\in W^{*}(\cup_{n},\Cup_{n})'\cap\mathfrak{M}_{n}\subset W^{*}(\cup_{n},\Cup_{n})'\cap\mathcal{M}_{n}$. 
\end{lem}
The proof is a straightforward computation and is omitted.

\begin{lem}
\label{lem:relCommcupCup}Let $P\subset P^{\Gamma}$ be a subfactor
planar algebra with index, and let $M_{n}=W^{*}(Gr_{n}P,Tr_{n})$
as above. Then $W^{*}(\cup_{n},\Cup_{n})'\cap M_{n}=P_{n,+}$. 
\end{lem}
\begin{proof}
Let $Q_{n}$ be the set of all paths in $\Gamma$ of length $n$ ending
at an even vertex (and starting at an even or odd vertex, according
to the parity of $n$). For $w=e_{1}\cdots e_{n}\in Q_{n}$, let $F_{w}=\ell(e_{1})\cdots\ell(e_{n})$.
Then for any $Y\in\mathfrak{M}_{n}$,\begin{equation}
\hat{Y}_{w,w'}=F_{w}^{*}YF_{w'}\in\mathcal{N}_{+}.\label{eq:Yw}\end{equation}
Moreover,\begin{equation}
Y=\sum_{w,w'\in Q_{n}}c_{w,w'}F_{w}\hat{Y}_{w,w'}F_{w'}^{*}\label{eq:YfromYw}\end{equation}
where $c_{w,w'}$ are some constants. Since the sum above is finite,
it follows that equations \eqref{eq:Yw} and \eqref{eq:YfromYw} continue
to hold whenever $Y\in\mathcal{M}_{n}$, i.e. after passing to weak
limits.

Thus if $Y\in\mathcal{M}_{n}$, and we set $Z=F_{w}F_{w}^{*}$, $Z'=F_{w'}F_{w'}^{*}$,
then $ZYZ'=F_{w}^{*}\hat{Y}F_{w}$, where $\hat{Y}\in\mathcal{N}_{+}$.
Moreover, $Y$ is equal to a fixed finite linear combination of terms
$\{ZYZ':w,w'\in Q_{n}\}$.

Let us assume now that $Y\in W^{*}(\cup_{n},\Cup_{n})'\cap\mathcal{M}_{n}$.
Then by choosing $Z,Z'$ as above, we see from Lemma \ref{lem:relComContainsPn}
that $ZYZ'\in W^{*}(\cup_{n},\Cup_{n})'\cap\mathcal{M}_{n}$. Using
equations \eqref{eq:Yw} and \eqref{eq:YfromYw}, we conclude that
$Y$ is a finite linear combination of terms of the form\[
\ell(e_{1})\cdots\ell(e_{n})X\ell(f_{n})^{*}\cdots\ell(f_{1})^{*},\qquad X\in\mathcal{N}_{+},\]
and that each such term must belong to the relative commutant $W^{*}(\cup_{n},\Cup_{n})'\cap\mathcal{M}_{k}$.

We can thus assume that $Y=\ell(e_{1})\cdots\ell(e_{n})X\ell(f_{n})^{*}\cdots\ell(f_{1})^{*}$
with $X\in\mathcal{N}_{+}$. Then \[
[Y,i_{n}(\cup)]=\left(\frac{\mu(t(e_{n}))}{\mu(s(e_{1}))}\right)^{1/2}\ell(e_{1})\cdots\ell(e_{n})[X,\cup]\ell(f_{n})^{*}\cdots\ell(f_{1})^{*}\]
and similarly\[
[Y,i_{n}(\Cup)]=\left(\frac{\mu(t(e_{n}))}{\mu(s(e_{1}))}\right)^{1/2}\ell(e_{1})\cdots\ell(e_{n})[X,\Cup]\ell(f_{n})^{*}\cdots\ell(f_{1})^{*}.\]
Thus if $Y$ is in the relative commutant of $i_{n}(\cup,\Cup)\cap\mathcal{M}_{n}$,
then $X$ must be in the relative commutant of $\{\cup,\Cup\}$ in
$\mathcal{N}_{+}$, which we know to be $A_{+}$ (Theorem \ref{thm:relCommutantTrivial}).
It follows that\begin{eqnarray*}
\{\cup,\Cup\}'\cap\mathcal{M}_{n} & \subset & \operatorname{span}\{\ell(e_{1})\cdots\ell(e_{n})\ell(f_{n})^{*}\cdots\ell(f_{1})^{*}:e_{1}\cdots e_{n}f_{n}^{o}\cdots f_{1}^{o}\in L_{\pm}\}\\
 & = & \{c_{n}(w):w\in L_{\pm}\textrm{ loop of length }2n\textrm{ starting at even/odd vertex}\}.\end{eqnarray*}
Since the reverse inclusion holds by Lemma \ref{lem:relComContainsPn},
equality holds. In particular, $\{\cup_{n},\Cup_{n}\}'\cap\mathcal{M}_{n}=\{\cup_{n},\Cup_{n}\}'\cap\mathfrak{M}_{n}$.
We now see from the definitions in section \ref{sec:PlanarAlgebra}
that the latter algebra is exactly the planar algebra $P_{n,+}^{\Gamma}$
taken with its usual multiplication.

Thus it follows that\[
\{\cup,\Cup\}'\cap M_{n}=P_{n,+}^{\Gamma}\cap M_{n}.\]
We claim that the latter intersection is exactly $P_{n,+}$. To see
this, write any $Y\in\mathcal{M}_{n}$ as\[
Y=\sum_{w,w'\in Q_{n}}c_{w,w'}F_{w}\hat{Y}_{w,w;}F_{w'}^{*}\]
with $\hat{Y}_{w,w'}=F_{w}^{*}YF_{w'}\in\mathcal{N}_{+}$, as before.
Let $E_{n}(Y)=\sum_{w,w'\in Q_{n}}c_{w,w'}F_{w}E(\hat{Y}_{w,w'})F_{w}^{*}$,
where $E:\mathcal{N}_{+}\to A_{+}$ is the (normal) conditional expectation
given by \eqref{eq:defOfE}. Then $E_{n}$ is a weakly-continuous
map, and moreover $E_{n}(Y)=Y$ if $Y\in P_{n,+}^{\Gamma}$. Thus
if $Y\in P_{n,+}^{\Gamma}\cap M_{n}$, then $Y$ is the limit (in
the weak operator topology) of a sequence $Y^{(j)}\in(P_{k_{j},+})\subset M_{n}$.
But then $Y=E(Y)=\lim_{k}E(Y^{(k)})$. Since the zero-box space of
$P$ is one-dimensional, it follows that $E(Y^{(k)})\in P_{n,+}$
(since then $E(\hat{Y}_{w,w'}^{(k)})\in\mathbb{C}1_{A_{+}})$ and
so $Y\in P_{n,+}$. Thus $P_{n,+}^{\Gamma}\cap M_{n}=P_{n,+}$ and
the theorem is proved.
\end{proof}
\begin{thm}
\label{thm:HRcommutsAsAlgArePks}Let $P\subset P^{\Gamma}$ be a subfactor
planar subalgebra of index $\delta\neq1$. Let $M_{k}=W^{*}(Gr_{k}P,tr_{k})$.
Then $M_{0}'\cap M_{k}=P_{k,+}$ as algebras (here $P_{k,+}$ is taken
with ordinary multiplication) in a way that preserves Jones projections.
\end{thm}
\begin{proof}
Since $\cup_{k},\Cup_{k}\in M_{k}$, Lemma \ref{lem:relCommcupCup}
shows that $P_{k,+}\supset M_{0}'\cap M_{k}$. Thus it is enough to
prove that $P_{k,+}\subset M_{0}'\cap M_{k}$. But this is immediate,
since $P_{k}$ commutes with $i_{k}(\mathfrak{M}_{0})$ and thus with
$M_{0}$. The correspondence takes Jones projections to Jones projections
(as is immediate from the pictures).
\end{proof}
\begin{lem}
Let $\mathbf{e}_{k}\in P_{k,+}$, $k\geq2$ be the Jones projection.
Then $\mathbf{e}_{k}$ is the Jones projection for the inclusion $M_{k-2}\subset M_{k-1}$.
\end{lem}
\begin{proof}
We first check that $\mathbf{e}_{k}\in M_{k-2}'\cap M_{k}$. Indeed, \begin{align*}
\mathbf{e}_{k} & =\delta^{-1}\overbrace{\begin{array}{c}\rnode{box}{\psframebox[framearc=0.4]{\vbox to2em{\vfill\hbox to10em{\hfill}}}
	\ncbar[ncurv=-1,arm=-0.4,angle=90,offsetA=0.2,offsetB=0.2,linearc=0.2]{box}{box}
	\ncbar[ncurv=-1,arm=-0.8,angle=90,offsetA=0.95,offsetB=0.95,linearc=0.2]{box}{box}
	\ncbar[ncurv=-1,arm=-0.4,angle=90,offsetA=1.5,offsetB=-1.1,linearc=0.2]{box}{box}
	\ncbar[ncurv=-1,arm=-0.4,angle=90,offsetA=-1.5,offsetB=1.1,linearc=0.2]{box}{box}
	\nput[labelsep=-0.75em,offset=0.55]{90}{box}{\cdots}
	\nput[labelsep=-0.75em,offset=-0.55]{90}{box}{\cdots}
	\nput[labelsep=-0.75em,offset=0]{90}{box}{*}
}\end{array}}^{k\textrm{ strings total}}\end{align*}
and since  $x\in M_{k-2}$, it has the form  \[
x = \begin{array}{c}
	\rnode{x}{\psframebox[framearc=0.4]{\vbox to3em{\vskip 0.3em \hbox to10em{\hfill$\circlenode{a}{A}$\hfill}\vfill}}}
	\ncbar[linearc=0.2, arm=-1.2, angle=90, offsetA=1.5, offsetB=1.5]{x}{x}
	\ncbar[linearc=0.2, arm=-1.0, angle=90, offsetA=1.1, offsetB=1.1]{x}{x}
	\ncdiag[armA=0,armB=-0.1,linearc=0.1,angleA=140,angleB=90,offsetB=-0.9]{a}{x}
	\ncdiag[armA=0,armB=-0.1,linearc=0.1,angleA=40,angleB=90,offsetB=0.9]{a}{x}
	\nput[labelsep=0.05]{88}{a}{*}
	\nput[labelsep=0.17]{55}{a}{\cdots}
	\nput[labelsep=0.17]{125}{a}{\cdots}
		\nput[labelsep=0.37]{115}{a}{\overbrace{\ \qquad\ }^{k-2}}
	\ncdiag[armA=0,armB=-0.1,angleA=80,angleB=90,offsetB=0.2]{a}{x}
	\ncdiag[armA=0,armB=-0.1,angleA=100,angleB=90,offsetB=-0.1]{a}{x}
\end{array}.
\] We now compute:
\begin{eqnarray*} 
\delta \mathbf{e}_{k} \wedge_k x &=& 
\begin{array}{c}
	\rnode{outer}{\psframebox[framearc=0.4]{\vbox to6em{\vfill}
	\rnode{box}{\psframebox[framearc=0.4,linestyle=dashed,linewidth=0.02]{\vbox to3em{\vfill\hbox to10em{\hfill}}}
	\ncbar[ncurv=-1,arm=-0.4,angle=90,offsetA=0.2,offsetB=0.2,linearc=0.2]{box}{box}
	\ncbar[ncurv=-1,arm=-0.8,angle=90,offsetA=0.95,offsetB=0.95,linearc=0.2]{box}{box}
	\ncbar[ncurv=-1,arm=-0.4,angle=90,offsetA=1.5,offsetB=-1.1,linearc=0.2]{box}{box}
	\ncbar[ncurv=-1,arm=-0.4,angle=90,offsetA=-1.5,offsetB=1.1,linearc=0.2]{box}{box}
	\nput[labelsep=-0.75em,offset=0.55]{90}{box}{\cdots}
	\nput[labelsep=-0.75em,offset=-0.55]{90}{box}{\cdots}
	} \qquad
	\rnode{x}{\psframebox[framearc=0.4,linestyle=dashed,linewidth=0.02]{\vbox to3em{\vskip 0.3em \hbox to10em{\hfill$\circlenode{a}{A}$\hfill}\vfill}}}
	\ncbar[linearc=0.2, arm=-1.2, angle=90, offsetA=1.5, offsetB=1.5]{x}{x}
	\ncbar[linearc=0.2, arm=-1.0, angle=90, offsetA=1.1, offsetB=1.1]{x}{x}
	\ncdiag[armA=0,armB=-0.1,linearc=0.1,angleA=140,angleB=90,offsetB=-0.9]{a}{x}
	\ncdiag[armA=0,armB=-0.1,linearc=0.1,angleA=40,angleB=90,offsetB=0.9]{a}{x}
	\nput[labelsep=0.17]{55}{a}{\cdots}
	\nput[labelsep=0.17]{125}{a}{\cdots}
	\ncdiag[armA=0,armB=-0.1,linearc=0.1,angleA=80,angleB=90,offsetB=0.2]{a}{x}
	\ncdiag[armA=0,armB=-0.1,linearc=0.1,angleA=100,angleB=90,offsetB=-0.1]{a}{x}
	\ncbar[linearc=0.2,arm=0.2,angle=90,offsetA=-1.5,offsetB=-1.5]{box}{x}
	\ncbar[linearc=0.2,arm=0.4,angle=90,offsetA=-1.1,offsetB=-1.1]{box}{x}
	\ncbar[linearc=0.2,arm=0.6,angle=90,offsetA=-0.95,offsetB=-0.9]{box}{x}
	\ncbar[linearc=0.2,arm=0.8,angle=90,offsetA=-0.2,offsetB=-0.1]{box}{x}
	}}
	\ncdiag[arm=0,angle=90,offsetA=1.5,offsetB=-3.8]{box}{outer}
	\ncdiag[arm=0,angle=90,offsetA=1.1,offsetB=-3.4]{box}{outer}
	\ncdiag[arm=0,angle=90,offsetA=0.95,offsetB=-3.25]{box}{outer}
	\ncdiag[arm=0,angle=90,offsetA=0.2,offsetB=-2.5]{box}{outer}
	\ncdiag[arm=0,angle=90,offsetA=-1.1,offsetB=3.4]{x}{outer}
	\ncdiag[arm=0,angle=90,offsetA=-0.2,offsetB=2.45]{x}{outer}
	\ncdiag[arm=0,angle=90,offsetA=-0.9,offsetB=3.15]{x}{outer}
	\ncdiag[arm=0,angle=90,offsetA=-1.5,offsetB=3.75]{x}{outer}
	\nput[labelsep=-0.25,offset=2.25]{90}{outer}{*}
\end{array} \\
&=&\begin{array}{c}
\rnode{x}{\psframebox[framearc=0.4]{\vbox to3em{\vskip 0.3em \hbox to10em{\hfill$\circlenode{a}{A}$\hfill}\vfill}}}
	\ncbar[linearc=0.2, arm=-0.5, angle=90, offsetA=1.5, offsetB=-1.1]{x}{x}
	\ncbar[linearc=0.2, arm=-0.5, angle=90, offsetA=-1.1, offsetB=1.5]{x}{x}
	\ncdiag[armA=0,armB=-0.1,linearc=0.1,angleA=140,angleB=90,offsetB=-0.9]{a}{x}
	\ncdiag[armA=0,armB=-0.1,linearc=0.1,angleA=40,angleB=90,offsetB=0.9]{a}{x}
	\nput[labelsep=0.05]{88}{a}{*}
	\nput[labelsep=0.17]{55}{a}{\cdots}
	\nput[labelsep=0.17]{125}{a}{\cdots}
	\ncdiag[armA=0,armB=-0.1,linearc=0.1,angleA=80,angleB=90,offsetB=0.2]{a}{x}
	\ncdiag[armA=0,armB=-0.1,linearc=0.1,angleA=100,angleB=90,offsetB=-0.1]{a}{x} 
\end{array} \\
&=& \delta x\wedge_k \mathbf{e}_k \qquad \textrm{(by symmetry)}.
\end{eqnarray*}
(here dashed lines indicate removed boxes).

Next, we check that $\mathbf{e}_{k}\wedge_k x\wedge_k \mathbf{e}_{k}=\mathcal{E}_{k-2}(x)\wedge_k \mathbf{e}_{k}$
for $x = \begin{array}{c}
\rnode{x}{\psframebox[framearc=0.4]{\vbox to3em{\vskip 0.3em \hbox to10em{\hfill$\circlenode{a}{A}$\hfill}\vfill}}}
	\ncbar[linearc=0.2, arm=-1.2, angle=90, offsetA=1.5, offsetB=1.5]{x}{x}
	\ncdiag[armA=0,armB=-0.1,linearc=0.1,angleA=140,angleB=90,offsetB=-0.9]{a}{x}
	\ncdiag[armA=0,armB=-0.1,linearc=0.1,angleA=40,angleB=90,offsetB=0.9]{a}{x}
	\nput[labelsep=0.05]{88}{a}{*}
	\nput[labelsep=0.17]{55}{a}{\cdots}
	\nput[labelsep=0.17]{125}{a}{\cdots}
	\ncdiag[armA=0,armB=-0.1,linearc=0.1,angleA=80,angleB=90,offsetB=0.2]{a}{x}
	\ncdiag[armA=0,armB=-0.1,linearc=0.1,angleA=100,angleB=90,offsetB=-0.1]{a}{x} 
			\nput[labelsep=0.37]{115}{a}{\overbrace{\ \qquad\ }^{k-1}}
\end{array} \in M_{k-1}.$  Note that it follows from the formula for $\mathcal{E}_{k-2}$ (or from an explicit computation using the trace) that \[
{E}_{k-2}(A) = \delta^{-1}  \begin{array}{c}
\rnode{x}{\psframebox[framearc=0.4]{\vbox to3em{\vskip 0.3em \hbox to8em{\hfill$\circlenode{a}{A}$\hfill}\vfill}}}
	\ncdiag[armA=0,armB=-0.1,linearc=0.1,angleA=140,angleB=90,offsetB=-0.9]{a}{x}
	\ncdiag[armA=0,armB=-0.1,linearc=0.1,angleA=40,angleB=90,offsetB=0.9]{a}{x}
	\nput[labelsep=0.05]{88}{a}{*}
	\nput[labelsep=0.17]{55}{a}{\cdots}
	\nput[labelsep=0.17]{125}{a}{\cdots}
	\ncdiag[armA=0,armB=-0.1,linearc=0.1,angleA=80,angleB=90,offsetB=0.2]{a}{x}
	\ncdiag[armA=0,armB=-0.1,linearc=0.1,angleA=100,angleB=90,offsetB=-0.1]{a}{x} 
	\nput[labelsep=0.37]{115}{a}{\overbrace{\ \qquad\ }^{k-2}}
	\ncloop[linearc=0.1,angleA=180,angleB=0,loopsize=0.6]{a}{a}
\end{array}
\]  Now,
\begin{eqnarray*}
\mathbf{e}_k \wedge_k x \wedge_k \mathbf{e}_k &=& \delta^{-2}
	\begin{array}{c}\vbox to4.5em{\vfill}
	\rnode{outer}{\psframebox[framearc=0.4]{\vbox to6em{\vfill}
	\rnode{box}{\psframebox[framearc=0.4,linestyle=dashed,linewidth=0.02]{\vbox to3em{\vfill\hbox to10em{\hfill}}}}
	\ncbar[ncurv=-1,arm=-0.4,angle=90,offsetA=0.2,offsetB=0.2,linearc=0.2]{box}{box}
	\ncbar[ncurv=-1,arm=-0.8,angle=90,offsetA=0.95,offsetB=0.95,linearc=0.2]{box}{box}
	\ncbar[ncurv=-1,arm=-0.4,angle=90,offsetA=1.5,offsetB=-1.1,linearc=0.2]{box}{box}
	\ncbar[ncurv=-1,arm=-0.4,angle=90,offsetA=-1.5,offsetB=1.1,linearc=0.2]{box}{box}
	\nput[labelsep=-0.75em,offset=0.55]{90}{box}{\cdots}
	\nput[labelsep=-0.75em,offset=-0.55]{90}{box}{\cdots}
	%
	\quad
	\rnode{x}{\psframebox[framearc=0.4,linestyle=dashed,linewidth=0.02]{\vbox to3em{\vskip 0.3em \hbox to10em{\hfill$\circlenode{a}{A}$\hfill}\vfill}}}
	\ncbar[linearc=0.2, arm=-1.2, angle=90, offsetA=1.5, offsetB=1.5]{x}{x}
	\ncdiag[arm=0,angleA=140,angleB=90,offsetB=-0.9]{a}{x}
	\ncdiag[arm=0,angleA=40,angleB=90,offsetB=0.9]{a}{x}
	\nput[labelsep=0.17]{55}{a}{\cdots}
	\nput[labelsep=0.17]{125}{a}{\cdots}
	\ncdiag[arm=0,angleA=60,angleB=90,offsetB=0.2]{a}{x}
	\ncdiag[arm=0,angleA=120,angleB=90,offsetB=-0.2]{a}{x}
	\ncdiag[arm=0,angleA=160,angleB=90,offsetB=-1.3]{a}{x}
	\ncdiag[arm=0,angleA=20,angleB=90,offsetB=1.3]{a}{x}
	\quad
	\rnode{box1}{\psframebox[framearc=0.4,linestyle=dashed,linewidth=0.02]{\vbox to3em{\vfill\hbox to10em{\hfill}}}}
	\ncbar[ncurv=-1,arm=-0.4,angle=90,offsetA=0.2,offsetB=0.2,linearc=0.2]{box1}{box1}
	\ncbar[ncurv=-1,arm=-0.8,angle=90,offsetA=0.95,offsetB=0.95,linearc=0.2]{box1}{box1}
	\ncbar[ncurv=-1,arm=-0.4,angle=90,offsetA=1.5,offsetB=-1.1,linearc=0.2]{box1}{box1}
	\ncbar[ncurv=-1,arm=-0.4,angle=90,offsetA=-1.5,offsetB=1.1,linearc=0.2]{box1}{box1}
	\nput[labelsep=-0.75em,offset=0.55]{90}{box1}{\cdots}
	\nput[labelsep=-0.75em,offset=-0.55]{90}{box1}{\cdots}
	}}
	\ncbar[linearc=0.2,arm=0.2,angle=90,offsetA=-1.5,offsetB=-1.5]{box}{x}
	\ncbar[linearc=0.2,arm=0.4,angle=90,offsetA=-1.1,offsetB=-1.3]{box}{x}
	\ncbar[linearc=0.2,arm=0.6,angle=90,offsetA=-0.95,offsetB=-0.9]{box}{x}
	\ncbar[linearc=0.2,arm=0.8,angle=90,offsetA=-0.2,offsetB=-0.2]{box}{x}
	\ncbar[linearc=0.2,arm=0.2,angle=90,offsetA=-1.5,offsetB=-1.5]{x}{box1}
	\ncbar[linearc=0.2,arm=0.4,angle=90,offsetA=-1.3,offsetB=-1.1]{x}{box1}
	\ncbar[linearc=0.2,arm=0.6,angle=90,offsetA=-0.9,offsetB=-0.95]{x}{box1}
	\ncbar[linearc=0.2,arm=0.8,angle=90,offsetA=-0.2,offsetB=-0.2]{x}{box1}
	\ncdiag[arm=0,angle=90,offsetA=1.5,offsetB=-5.67]{box}{outer}
	\ncdiag[arm=0,angle=90,offsetA=1.1,offsetB=-5.27]{box}{outer}
	\ncdiag[arm=0,angle=90,offsetA=0.95,offsetB=-5.12]{box}{outer}
	\ncdiag[arm=0,angle=90,offsetA=0.2,offsetB=-4.37]{box}{outer}
		\nput[labelsep=-0.25,offset=4.1]{90}{outer}{*}
	\ncdiag[arm=0,angle=90,offsetA=-1.5,offsetB=5.67]{box1}{outer}
	\ncdiag[arm=0,angle=90,offsetA=-1.1,offsetB=5.27]{box1}{outer}
	\ncdiag[arm=0,angle=90,offsetA=-0.95,offsetB=5.12]{box1}{outer}
	\ncdiag[arm=0,angle=90,offsetA=-0.2,offsetB=4.37]{box1}{outer}
	\ncdiag[armA=1.13,nodesepA=-0.02,linearc=0.1,armB=-0.15,angleA=90,offsetA=0.075,angleB=90,offsetB=-0.5]{a}{outer}
	\ncdiag[armA=1.13,nodesepA=-0.02,linearc=0.1,armB=-0.15,angleA=90,offsetA=-0.075,angleB=90,offsetB=0.5]{a}{outer}
	\nput[labelsep=-0.2]{90}{outer}{\cdots}
	\end{array}\\
	&=&
	\delta^{-2} \begin{array}{c}
\rnode{x}{\psframebox[framearc=0.4]{\vbox to3em{\vskip 0.3em \hbox to10em{\hfill$\circlenode{a}{A}$\hfill}\vfill}}}
	\ncbar[linearc=0.2, arm=-0.5, angle=90, offsetA=1.5, offsetB=-1.1]{x}{x}
	\ncbar[linearc=0.2, arm=-0.5, angle=90, offsetA=-1.1, offsetB=1.5]{x}{x}
	\ncdiag[armA=0,armB=-0.1,linearc=0.1,angleA=140,angleB=90,offsetB=-0.9]{a}{x}
	\ncdiag[armA=0,armB=-0.1,linearc=0.1,angleA=40,angleB=90,offsetB=0.9]{a}{x}
	\ncloop[linearc=0.1,angleA=180,angleB=0,loopsize=0.6]{a}{a}
	\nput[labelsep=0.05]{88}{a}{*}
	\nput[labelsep=0.17]{55}{a}{\cdots}
	\nput[labelsep=0.17]{125}{a}{\cdots}
	\ncdiag[armA=0,armB=-0.1,linearc=0.1,angleA=80,angleB=90,offsetB=0.2]{a}{x}
	\ncdiag[armA=0,armB=-0.1,linearc=0.1,angleA=100,angleB=90,offsetB=-0.1]{a}{x} 
	\end{array} 
	\quad=\quad \delta^{-1}
	\begin{array}{c}
\rnode{x}{\psframebox[framearc=0.4]{\vbox to3em{\vskip 0.3em \hbox to10em{\hfill$\ovalnode{a}{E_{M_{k-2}} (A)}$\hfill}\vfill}}}
	\ncbar[linearc=0.2, arm=-0.4, angle=90, offsetA=1.5, offsetB=-1.1]{x}{x}
	\ncbar[linearc=0.2, arm=-0.4, angle=90, offsetA=-1.1, offsetB=1.5]{x}{x}
	\ncdiag[armA=0,armB=-0.1,linearc=0.1,angleA=150,angleB=90,offsetB=-0.9]{a}{x}
	\ncdiag[armA=0,armB=-0.1,linearc=0.1,angleA=30,angleB=90,offsetB=0.9]{a}{x}
	\nput[labelsep=0.05]{88}{a}{*}
	\nput[labelsep=0.1]{55}{a}{\cdots}
	\nput[labelsep=0.1]{125}{a}{\cdots}
	\ncdiag[armA=0,armB=-0.1,linearc=0.1,angleA=80,angleB=90,offsetB=0.2]{a}{x}
	\ncdiag[armA=0,armB=-0.1,linearc=0.1,angleA=100,angleB=90,offsetB=-0.1]{a}{x} 
	\end{array} \\
	&=& \mathcal{E}_{k-2} (x)\wedge_{k}\mathbf{e}_{k}.
\end{eqnarray*}
Finally, we check that the trace $Tr_{k}$ is $\lambda$-Markov.  Let $x\in M_{k-1}$.  Then: \
\begin{eqnarray*}
\delta Tr_k(x \wedge_k \mathbf{e}_k) &=& Tr_k  \left\{\begin{array}{c}
\rnode{outer}{\psframebox[framearc=0.4]{
	\vbox to4.3em{\vfill\hbox{$
	\begin{array}{cc}\rnode{box}{\psframebox[framearc=0.4]{\vbox to2em{\vfill\hbox to10em{\hfill}\vfill}
	\ncbar[ncurv=-1,arm=-0.4,angle=90,offsetA=0.2,offsetB=0.2,linearc=0.2]{box}{box}
	\ncbar[ncurv=-1,arm=-0.8,angle=90,offsetA=0.95,offsetB=0.95,linearc=0.2]{box}{box}
	\ncbar[ncurv=-1,arm=-0.4,angle=90,offsetA=1.5,offsetB=-1.1,linearc=0.2]{box}{box}
	\ncbar[ncurv=-1,arm=-0.4,angle=90,offsetA=-1.5,offsetB=1.1,linearc=0.2]{box}{box}
	\nput[labelsep=-0.75em,offset=0.55]{90}{box}{\cdots}
	\nput[labelsep=-0.75em,offset=-0.55]{90}{box}{\cdots}
}} &  \rnode{f1}{}\quad\vbox to2em{\hbox{$ \rnode{x}{\psframebox[framearc=0.4]{\vbox to 1em{\vfill \hbox to6em{\hfill $x$\hfill} \vfill}}}
	\quad\rnode{f2}{}\quad$}\vfill}
	\ncbar[angle=90,offsetA=-1.5,offsetB=0,arm=0.2,linearc=0.1]{box}{f1}
	\ncbar[angle=90,offsetA=-0.95,offsetB=-0.8,arm=0.6,linearc=0.1]{box}{x}
	\ncbar[angle=90,offsetA=-1.1,offsetB=-1.0,arm=0.4,linearc=0.1]{box}{x}
	\ncbar[angle=90,offsetA=-0.2,offsetB=-0.2,arm=0.8,linearc=0.1]{box}{x}
	\nput[offset=0.5]{90}{x}{\cdots}
	\ncbar[angle=-90,arm=0.13,linearc=0.2]{f1}{f2}
	\ncdiag[arm=0,angle=90,offsetA=-0.2,offsetB=2.125]{x}{outer}
	\ncdiag[arm=0,angle=90,offsetA=-1,offsetB=2.925]{x}{outer}	
	\ncdiag[arm=0,angle=90,offsetA=0,offsetB=3.45]{f2}{outer}
	\nput[offset=-0.6]{90}{x}{\cdots}
	\ncdiag[arm=0,angle=90,offsetA=1.5,offsetB=-3.4]{box}{outer}
	\ncdiag[arm=0,angle=90,offsetA=1.1,offsetB=-3.0]{box}{outer}
	\ncdiag[arm=0,angle=90,offsetA=0.95,offsetB=-2.85]{box}{outer}
	\ncdiag[arm=0,angle=90,offsetA=0.2,offsetB=-2.1]{box}{outer}
	\nput[labelsep=-0.25,offset=1.9]{90}{outer}{*}
\end{array}$}}}}\end{array}\right\} \\[1em]
 &=& Tr_k \left\{ \begin{array}{c}
 	\rnode{outer}{\psframebox[framearc=0.4]{
			\psmatrix[colsep=0.4,rowsep=0.3] 
			\\				& 
				\qquad
				\rnode{f}{} 
				\quad 
				\rnode{x}{\psframebox[framearc=0.4]{
					\vbox to 1em{\vfill \hbox to6em{\hfill $x$\hfill} \vfill}
				}}
				\quad
				\rnode{f2}{}
				\quad
			\\ \vbox{\vfill}
			\endpsmatrix
		}
	}
 \end{array}
	\ncbar[angle=90,linearc=0.1,arm=0.4,offsetB=-0.8]{f}{x}
	\ncbar[angle=-90,linearc=0.1,arm=0.4]{f}{f2}
	\ncdiag[angle=90,arm=0,offsetB=1.92]{f2}{outer}
	\ncdiag[angle=90,arm=0,offsetA=0.385,offsetB=0]{x}{outer}
	\ncdiag[angle=90,arm=0,offsetA=-0.2,offsetB=0.585]{x}{outer}
	\nput[offset=-0.42,labelsep=-0.25]{90}{outer}{*}
	\nput[offset=0.1,labelsep=0.1]{90}{x}{\cdots}
	\ncdiag[angle=90,arm=0,offsetA=-0.8,offsetB=1.185]{x}{outer}
		\nput[offset=-0.5,labelsep=0.1]{90}{x}{\cdots}
	\ncbar[angle=90,linearc=0.1,arm=-0.4,offsetA=2.02,offsetB=-1.52]{outer}{outer}
 \right\} \\[1em]
 &=& \begin{array}{c}
\psmatrix[rowsep=0.4,colsep=0.2]
	\\
 	\rnode{f1}{}
	\quad&
	\rnode{f2}{}
	&
	\rnode{f3}{}
	&
	\rnode{T}{
		\psframebox[framearc=0.4]{
			\vbox to1em {
				\vfill
				\hbox to4em{\hfill$T_n$\hfill}
				\vfill
			}
		}
	} \\
	&&& \rnode{x} {
		\psframebox[framearc=0.4]{
			\vbox to1em {
				\vfill
				\hbox to10em{\hfill$x$\hfill}
				\vfill
			}
		}
	}
	& \rnode{f4}{}
	\ncdiag[angleA=90,angleB=-90,arm=0,offsetA=0.6,offsetB=0.6]{x}{T}
	\ncdiag[angleA=90,angleB=-90,arm=0,offsetA=0.3,offsetB=0.3]{x}{T}
	\nput{90}{x}{\cdots}
	\nput[offset=0.7,labelsep=0.1]{90}{x}{*}
	\ncdiag[angleA=90,angleB=-90,arm=0,offsetA=-0.3,offsetB=-0.3]{x}{T}
	\ncdiag[angleA=90,angleB=-90,arm=0,offsetA=-0.6,offsetB=-0.6]{x}{T}
	\ncbar[angle=-90,linearc=0.1,arm=0.3]{f1}{f2}
	\ncbar[angle=90,arm=1.2,linearc=0.2,offsetA=1.0,offsetB=1.0]{x}{x}
	\ncbar[angle=90,arm=1.4,linearc=0.2,offsetA=1.4,offsetB=1.4]{x}{x}
	\ncbar[angle=90,arm=0.4,linearc=0.2,offsetA=0,offsetB=-1.6]{f3}{x}
	\ncbar[angle=-90,linearc=0.4]{f3}{f4}
	\ncbar[angle=90,linearc=0.4,arm=1.2]{f4}{f1}
	\ncbar[angle=90,linearc=0.4,arm=1.0,offsetB=1.6]{f2}{x}
	\nput[offset=1.2]{90}{x}{\hbox{\tiny$\cdots$}}
	\nput[offset=-1.2]{90}{x}{\hbox{\tiny$\cdots$}}
\endpsmatrix
\end{array}
\\[1em]
&=& 
\begin{array}{c}
\psmatrix[rowsep=0.4,colsep=0.2]
	\\
	&
	\rnode{T}{
		\psframebox[framearc=0.4]{
			\vbox to1em {
				\vfill
				\hbox to4em{\hfill$T_n$\hfill}
				\vfill
			}
		}
	} \\
	\rnode{f3}{}
         & \rnode{x} {
		\psframebox[framearc=0.4]{
			\vbox to1em {
				\vfill
				\hbox to10em{\hfill$x$\hfill}
				\vfill
			}
		}
	}
	& \rnode{f4}{}
	\ncdiag[angleA=90,angleB=-90,arm=0,offsetA=0.6,offsetB=0.6]{x}{T}
	\ncdiag[angleA=90,angleB=-90,arm=0,offsetA=0.3,offsetB=0.3]{x}{T}
	\nput{90}{x}{\cdots}
	\nput[offset=0.7,labelsep=0.1]{90}{x}{*}
	\ncdiag[angleA=90,angleB=-90,arm=0,offsetA=-0.3,offsetB=-0.3]{x}{T}
	\ncdiag[angleA=90,angleB=-90,arm=0,offsetA=-0.6,offsetB=-0.6]{x}{T}
	\ncbar[angle=90,arm=1.2,linearc=0.2,offsetA=1.0,offsetB=1.0]{x}{x}
	\ncbar[angle=90,arm=1.4,linearc=0.2,offsetA=1.4,offsetB=1.4]{x}{x}
	\ncbar[angle=-90,linearc=0.2]{f3}{f4}
	\ncbar[angle=90,linearc=0.2,arm=0.4,offsetB=1.6]{f4}{x}
	\ncbar[angle=90,arm=0.4,linearc=0.2,offsetA=0,offsetB=-1.6]{f3}{x}
	\nput[offset=1.2]{90}{x}{\hbox{\tiny$\cdots$}}
	\nput[offset=-1.2]{90}{x}{\hbox{\tiny$\cdots$}}
\endpsmatrix
\end{array}
=
\begin{array}{c}
\psmatrix[rowsep=0.4,colsep=0.2]
	\\
	\rnode{T}{
		\psframebox[framearc=0.4]{
			\vbox to1em {
				\vfill
				\hbox to4em{\hfill$T_n$\hfill}
				\vfill
			}
		}
	} \\
         \rnode{x} {
		\psframebox[framearc=0.4]{
			\vbox to1em {
				\vfill
				\hbox to10em{\hfill$x$\hfill}
				\vfill
			}
		}
	}
	\ncdiag[angleA=90,angleB=-90,arm=0,offsetA=0.6,offsetB=0.6]{x}{T}
	\ncdiag[angleA=90,angleB=-90,arm=0,offsetA=0.3,offsetB=0.3]{x}{T}
	\nput{90}{x}{\cdots}
	\nput[offset=0.7,labelsep=0.1]{90}{x}{*}
	\ncdiag[angleA=90,angleB=-90,arm=0,offsetA=-0.3,offsetB=-0.3]{x}{T}
	\ncdiag[angleA=90,angleB=-90,arm=0,offsetA=-0.6,offsetB=-0.6]{x}{T}
	\ncbar[angle=90,arm=1.2,linearc=0.2,offsetA=1.0,offsetB=1.0]{x}{x}
	\ncbar[angle=90,arm=1.4,linearc=0.2,offsetA=1.4,offsetB=1.4]{x}{x}
	\ncbar[angle=90,arm=1.6,linearc=0.2,offsetA=1.6,offsetB=1.6]{x}{x}
	\nput[offset=1.2]{90}{x}{\hbox{\tiny$\cdots$}}
	\nput[offset=-1.2]{90}{x}{\hbox{\tiny$\cdots$}}
	\endpsmatrix
\end{array}
\\[1em]
&=&Tr_{k-1}(x).
\end{eqnarray*}
\end{proof}
\begin{lem}
The algebras $M_{0}\subset M_{1}\subset M_{2}\subset\cdots$ are exactly
the tower obtained by iterating the basic construction for $M_{0}\subset M_{1}$. 
\end{lem}
\begin{proof}
We first note that because of the Markov property and the Jones relations
between the projections $\mathbf{e}_{n}$, the algebras $\hat{M}_{n}=\langle M,\mathbf{e}_{1},\ldots,\mathbf{e}_{n-1}\rangle$,
$n\geq2$ are exactly the algebras appearing in the basic construction
for $M_{0}\subset M_{1}$. Hence clearly $\hat{M}_{n}\subset M_{n}$.
Now suppose that for some $n$ this inclusion were strict; choose
smallest such $n$ (necessarily $>1$ since $M_{0}=\hat{M}_{0}$ and
$M_{1}=\hat{M}_{1}$). Then the projection $\mathbf{e}_{n+1}$ is
the Jones projection for $M_{n-1}\subset M_{n}$ and also for $M_{n-1}=\hat{M}_{n-1}\subset\hat{M}_{n}$.
Thus the index of $M_{n-1}\subset M_{n}$ is the same as that of $M_{n-1}\subset\hat{M}_{n}$.
But since $\hat{M}_{n}\subset M_{n}$, multiplicativity of index entails
$[M_{n}:\hat{M}_{n}]=1$ and thus $M=\hat{M}_{n}$, a contradiction. 
\end{proof}

\subsection{The planar algebra structure on the higher relative commutants.}

At this stage we have constructed a (II$_{1}$) subfactor $M_{0}\subset M_{1}$
and its tower $M_{k}$ as the completions of $Gr_{k}P$. We have also
shown that $M_{0}'\cap M_{k}$ is precisely subspace $P_{k}\subset Gr_{k}(P)$. 

\begin{thm}
The linear identification of $P_{k}$ and $M_{0}'\cap M_{j}$ constructed
in Theorem \ref{thm:HRcommutsAsAlgArePks} is an isomorphism between
$P$ and the planar algebra of the subfactor $P(M_{0}\subset M_{1})$. 

In particular, any subfactor planar algebra can be naturally realized
as the planar algebra of the II$_{1}$ subfactor $P(W^{*}(Gr_{0}P,Tr_{0})\subset W^{*}(Gr_{1}P,Tr_{1}))$
. 
\end{thm}
The second part of the theorem gives an alternative proof of a result
of Popa \cite{popa:markov,popa:standardlattice,shlyakht-popa:universal}.

\begin{proof}
We have seen in Theorem \ref{thm:HRcommutsAsAlgArePks} that the multiplication
induced by $Gr_{k}(P)$ (hence $M_{k}$) on $P_{k}$ is precisely
that of the multiplication tangle. By \cite{jones:planar}, to conclude that the
planar algebra structure defined on $P$ by this identification with
the higher relative commutants for $M_{0}\subset M_{1}$ we have to
check the following.

1) That $M_{0}\subset M_{1}$ is extremal (which means there is only
one trace on the $M_{0}'\cap M_{k}$, that of $M_{k}$).

2) The Jones projections $\mathbf{e}_{i}$ of the tower are ($\frac{1}{\delta}$
times) the diagrammatic $\mathbf{e}_{i}$'s.

3) The inclusion of $M_{0}'\cap M_{k}$ in $M_{0}'\cap M_{k+1}$ is
given by the appropriate tangle.

4) The trace on $M_{0}'\cap M_{k}$ given by restricting the trace
on $M_{k}$ is given by the appropriate tangle.

5) The projection from $M_{0}'\cap M_{k}$ onto $M_{1}'\cap M_{k}$
is given by the appropriate tangle.

For these, 1) follows from the definition of extremality in \cite{pimsner-popa:entropyIndex,popa:classificationIIacta} 
and a simple diagrammatic manipulation involving spherical invariance
of the partition function. 2) was proved as part of Theorem \ref{thm:HRcommutsAsAlgArePks}.
Properties 3) and 4) are just obvious pictures. The only one that
requires any thought is 5) which we now prove.
\begin{claim}
Any element in $M_{1}'\cap M_{k}$ is in the image of the map from
$M_{0}'\cap M_{k-1}$ to $M_{0}'\cap M_{k}$ defined by the following
annular tangle:\[
A\mapsto \begin{array}{c}\rnode{outer}{\psframebox[framearc=0.4]{\vbox to3em{\vfill}
\qquad\qquad\circlenode{x}{\ A\ }\qquad\qquad
\ncdiag[arm=0,angleA=30,angleB=30]{x}{outer}
\ncdiag[arm=0,angleA=60,angleB=60]{x}{outer}
\ncdiag[arm=0,angleA=120,angleB=120]{x}{outer}
\ncdiag[arm=0,angleA=150,angleB=150]{x}{outer}
\nput[labelsep=-0.25]{135}{outer}{\cdots}
\nput[labelsep=-0.25]{45}{outer}{\cdots}
\ncbar[arm=-0.3,linearc=0.2,angle=90,offsetA=0.2,offsetB=0.2]{outer}{outer}
\nput[labelsep=-0.22]{90}{outer}{*}
}}\end{array}.
\] (The shading is determined by the stars being in unshaded regions,
the position of $*$ on the inside box being irrelevant.)
\end{claim}
\begin{proof}
[Proof of claim.] It is a simple diagrammatic calculation to show
that the image of this tangle does indeed commute with $M_{1}$. On
the other hand the tangle defines an injective map (the inverse tangle
is obvious) and from general subfactor theory the dimensions of $M_{0}'\cap M_{k-1}$
and $M_{1}'\cap M_{k}$ are the same. 
\end{proof}
\begin{claim}
If $A$ is in $M_{0}'\cap M_{k}$, identified with $P_{k}$, then\[
E_{M_{1}'}(A)=\delta^{-1}\begin{array}{c}\rnode{outer}{\psframebox[framearc=0.4]{\vbox to3em{\vfill}
\qquad\qquad\ovalnode{x}{\quad A\quad }\qquad\qquad
\ncdiag[arm=0,angleA=20,angleB=30]{x}{outer}
\ncdiag[arm=0,angleA=45,angleB=60]{x}{outer}
\ncdiag[arm=0,angleA=135,angleB=120]{x}{outer}
\ncdiag[arm=0,angleA=160,angleB=150]{x}{outer}
\nput[labelsep=-0.25]{135}{outer}{\cdots}
\nput[labelsep=-0.25]{45}{outer}{\cdots}
\ncbar[arm=-0.3,linearc=0.2,angle=90,offsetA=0.2,offsetB=0.2]{outer}{outer}
\nput[labelsep=-0.22]{90}{outer}{*}
\ncbar[arm=0.3,linearc=0.2,angle=90,offsetA=0.2,offsetB=0.2]{x}{x}
\nput[labelsep=0]{90}{x}{*}
}}\end{array}.
\]
\end{claim}
\begin{proof}
[Proof of claim.] By extremality $E_{M_{1}'}=E_{M_{1}'\cap M_{k}}$
for elements of $M_{0}'\cap M_{k}$. Drawing the picture for $tr(AB)$
for $A\in M_{0}'\cap M_{k}$ and $B\in M_{1}'\cap M_{k}$, the result
is visible. 
\end{proof}
This concludes the proof of the Theorem.
\end{proof}

\bibliographystyle{alpha}

\end{document}